\documentclass[11pt]{article}
\usepackage[margin=0.75in]{geometry}

\usepackage{amsthm}
\usepackage{amssymb}
\usepackage{amsmath}
\usepackage{comment}
\usepackage{mathtools}
\usepackage{thm-restate}
\usepackage{url} 
\usepackage{hyperref}
\usepackage[noabbrev,capitalise]{cleveref}
\usepackage{verbatim}

\usepackage[utf8]{inputenc}

\usepackage{color}
\usepackage[normalem]{ulem}

\usepackage[shortlabels]{enumitem}

\usepackage[dvipsnames, table]{xcolor}

\def\TKsuppress#1{}
\def\LPsuppress#1{}
\def\MDsuppress#1{}

\def\COMMENT#1{}

\numberwithin{equation}{section}

\newcommand{\eps}{\varepsilon}
\newcommand{\mc}[1]{\mathcal{#1}}

\newcounter{i}
\theoremstyle{plain}
\newtheorem{thm}{Theorem}[section]
\newtheorem{lem}[thm]{Lemma}
\newtheorem{claim}{Claim}[thm]
\newtheorem{proposition}[thm]{Proposition}
\newtheorem{fact}[thm]{Fact}
\newtheorem{cor}[thm]{Corollary}

\newtheorem{conj}[thm]{Conjecture}

\newenvironment{proofclaim}[1][]%
{\noindent \emph{Proof.} {}{#1}{}}{\hfill
	$\Diamond$\vspace{1em}}

\theoremstyle{plain} 
\newcommand{\thistheoremname}{}
\newtheorem{genericthm}{\thistheoremname}

\theoremstyle{definition}
\newtheorem{definition}[thm]{Definition}

\newcommand{\cQ}{\mathcal{Q}} 
\newcommand{\cC}{\mathcal{C}}
\newcommand{\cD}{\mathcal{D}} 
\newcommand{\cB}{\mathcal{B}}

\newcommand{\cF}{\mathcal{F}}

\newcommand{\cH}{\mathcal{H}}
\newcommand{\cG}{\mathcal{G}}

\newcommand{\cW}{\mathcal{W}}

\newcommand{\Prob}[1]{\ensuremath{%
\mathbb P\left[#1\right]
}}
\newcommand{\ProbCond}[2]{\ensuremath{%
    \mathbb P\left[#1\:\middle|\:#2\right]
  }}

\newcommand{\Expect}[1]{\ensuremath{%
\mathbb E\left[#1\right]
}}

\newcommand{\aas}{a.a.s.\ }

\title{Clique Decompositions in Random Graphs via Refined Absorption}

\author{
Michelle Delcourt
\thanks{Department of Mathematics, Toronto Metropolitan University (formerly named Ryerson University),
Toronto, Ontario M5B 2K3, Canada {\tt mdelcourt@torontomu.ca}. Research supported by NSERC under Discovery Grant No. 2019-04269.}
\and
Tom Kelly
\thanks{School of Mathematics, Georgia Institute of Technology, Atlanta, GA 30332, USA {\tt tom.kelly@gatech.edu}.
Research supported by the National Science Foundation under Grant No. DMS-2247078.}
\and
Luke Postle
\thanks{Combinatorics and Optimization Department,
University of Waterloo, Waterloo, Ontario N2L 3G1, Canada {\tt lpostle@uwaterloo.ca}. Partially supported by NSERC
under Discovery Grant No. 2019-04304.}}
\date{February 29, 2024}

\begin{document}

\maketitle

\begin{abstract} 
We prove that if $p\ge n^{-\frac{1}{3}+\beta}$ for some $\beta > 0$, then asymptotically almost surely the binomial random graph $G(n,p)$ has a $K_3$-packing containing all but at most $n + O(1)$ edges. Similarly, we prove that if $d \ge n^{\frac{2}{3}+\beta}$ for some $\beta > 0$ and $d$ is even, then asymptotically almost surely the random $d$-regular graph $G_{n,d}$ has a triangle decomposition provided $3 \mid d \cdot n$. We also show that $G(n,p)$ admits a fractional $K_3$-decomposition for such a value of $p$. We prove analogous versions for a $K_q$-packing of $G(n,p)$ with $p\ge n^{-\frac{1}{q+0.5}+\beta}$ and leave of $(q-2)n+O(1)$ edges, for $K_q$-decompositions of $G_{n,d}$ with $(q-1)~|~d$ and $d\ge n^{1-\frac{1}{q+0.5}+\beta}$ provided $q\mid d\cdot n$, and for fractional $K_q$-decompositions.
\end{abstract}

\section{Introduction}

For (hyper)graphs $F$ and $G$, an \textit{$F$-decomposition} of $G$ is a collection of pairwise edge-disjoint copies of $F$ in $G$ whose edge sets partition $E(G)$, and an \textit{$F$-packing} of $G$ is an $F$-decomposition of a subgraph of $G$.  Decompositions of graphs are closely related to combinatorial designs.  For example, a decomposition of a complete graph into triangles is a \textit{Steiner triple system}.  A \textit{balanced incomplete block design (BIBD)} with parameters $n,q,\lambda$, or a \textit{$2$-$(n,q,\lambda)$ design}, is a $K_q$-decomposition of the $n$-vertex multi-graph with $\lambda$ parallel edges between every pair of distinct vertices.  More generally, an \textit{$r$-$(n,q,\lambda)$ design} is a $K_q^r$-decomposition of the $n$-vertex hypergraph with $\lambda$ hyperedges containing every $r$-set of vertices.

If a graph $G$ admits an $F$-decomposition, then $e(F)\mid e(G)$, where $e(G)$ denotes the number of edges of $G$; moreover, if $F$ is $r$-regular, then the degree $d(v)$ of every vertex in $G$ is divisible by $r$.  To that end, we say $G$ is \textit{$K_q$-divisible} if $\binom{q}{2} \mid e(G)$ and $q - 1 \mid d(v)$ for every $v \in V(G)$.  In the 1970s, Wilson~\cite{WiI, WiII, WiIII} famously proved that a $2$-$(n,q,\lambda)$ design exists for sufficiently large $n$ satisfying these necessary divisibility conditions.  This result is the $r = 2$ case of the Existence Conjecture, dating back to the 1850s, which states that an $r$-$(n,q,\lambda)$ design exists for sufficiently large $n$ such that $\binom{q - i}{r - i} \mid \lambda\binom{n - i}{r - i}$ for every $i \in \{0,\dots,r-1\}$.  The Existence Conjecture remained open until a breakthrough of Keevash~\cite{K14} in 2014.

In his proof of the Existence Conjecture, Keevash introduced an \textit{absorption} technique he called ``Randomized Algebraic Construction''.  In 2016, Glock, K\"uhn, Lo, and Osthus~\cite{GKLO16} provided a new, purely combinatorial, proof of the Existence Conjecture using ``Iterative Absorption''.  The absorption method was first systematized by R\"odl, Ruci\'{n}ski, and Szemer\'edi~\cite{RRS06, RRS06pm, RRS08, RRS09} in the 2000s for finding perfect matchings and Hamilton cycles in hypergraphs, which may be considered ``embedding'' problems, and it has driven many advances in the last two decades.  The absorption method was first used for decomposition problems by Knox, K\"uhn, and Osthus~\cite{KKO15} and K\"uhn and Osthus~\cite{KO13} and has been used to settle many longstanding decomposition problems since, including the Existence Conjecture \cite{KO13, K14, Ke18, Ke18EOD, GJKKO21, MPS21, KS20, KKKMO23}.

An important area of research is the study of embeddings in and decompositions of \textit{random graphs}. 
One of the most highly studied models is the \textit{binomial random graph $G(n,p)$}, which is an $n$-vertex random graph in which each edge is included independently with probability $p$.
In 1966, Erd\H{o}s and R\'{e}nyi~\cite{ER66} proved that if $p = \omega(n^{-1}\log n)$, then asymptotically almost surely (which we abbreviate as a.a.s.) $G(2n,p)$ contains a perfect matching.  In 1976, P\'osa~\cite{Po76} generalized this result by showing that if $p = \omega(n^{-1}\log n)$, then \aas $G(n,p)$ contains a Hamilton cycle.  
Both of these results are best possible in the sense that, for $p = o(n^{-1}\log n)$, \aas $G(n,p)$ has a vertex contained in no edges, and thus does not contain a perfect matching or Hamilton cycle.  That is, $n^{-1}\log n$ is a \textit{threshold function} for $G(n,p)$ to contain a Hamilton cycle.  In 2019, Montgomery~\cite{Mo19} proved that for every fixed $\Delta$, if $p = \omega(n^{-1}\log n)$, then \aas $G(n,p)$ in fact contains a copy of every spanning tree of maximum degree at most $\Delta$.  In this proof, Montgomery introduced the ``distributive absorption'' or ``template absorption'' method, which has found several important applications since.

Erd\H{o}s and R\'{e}nyi's result can be generalized in at least two additional ways. 
First, Shamir's Problem from the 1970s asked for the threshold for the binomial random $k$-uniform hypergraph to contain a perfect matching, and a closely related problem is to determine the threshold for $G(rn, p)$ to contain a $K_r$-factor.  Both problems were settled by Johansson, Kahn, and Vu~\cite{JKV08} in 2008 and also follow from 
recent breakthroughs of Frankston, Kahn, Narayanan, and Park~\cite{FKNP21} and Park and Pham~\cite{PP23} on the Kahn--Kalai Conjecture~\cite{KK07}.
Notably, both proofs are non-constructive, and there appears to be a theoretical barrier to obtaining optimal threshold bounds proved with known absorption-based approaches.  Prior to Johansson, Kahn, and Vu's proof, Krivelevich~\cite{Kr97} proved an upper bound on the triangle-factor threshold with a technique that could be considered the first use of the absorption method. Recently Ferber and Kwan~\cite{FK22} used the distributive absorption method to prove a ``resilience version'' of Shamir's Problem which provides an upper bound on the perfect matching threshold.  However, these bounds do not match the known threshold functions.

Shamir's Problem can be understood as asking for the threshold for the binomial random $qn$-vertex $q$-uniform hypergraph to contain a $1$-$(qn, q, 1)$ design.  
It also may be understood as a problem of decomposing \textit{vertices} rather than \textit{edges}.  It is natural to ask about (edge) decompositions of random graphs; however, due to the divisibility requirements, it is unlikely that $G(n,p)$ admits an $F$-decomposition for most choices of $F$.
There are two natural ways to amend this question.  First, we can ask for an $F$-packing of $G(n,p)$ with small \textit{leave}, or we can ask whether the uniformly random $n$-vertex $d$-regular graph $G_{n,d}$ contains a $K_q$-packing, provided $n$ and $d$ are chosen to satisfy the divisibility conditions.
(In a companion paper~\cite{DKPIV}, we study the threshold for the binomial random $q$-uniform $n$-vertex hypergraph to contain a $2$-$(n, q, 1)$ design.)
The case for when $F$ is a Hamilton cycle, which was first considered by Bollob{\'a}s and Frieze~\cite{BF83} in the 1980s, has attracted considerable attention.  
Frieze and Krivelevich~\cite{FK08} conjectured that for \textit{every} $p = p(n)$, \aas $G\sim G(n,p)$ contains $\lfloor \delta(G) / 2\rfloor$ pairwise edge-disjoint Hamilton cycles, plus a perfect matching if $\delta(G)$ is odd.
Here, $\delta(G)$ denotes the minimum degree of a vertex of $G$, so every graph $G$ has at most $\lfloor \delta(G) / 2\rfloor$ pairwise edge-disjoint Hamilton cycles.
This conjecture was settled by the combined work of several sets of authors \cite{BF83, KS12, KKO15, KO13}.
Notably, the ``dense'' range of $p$ for this problem, $p \geq n^{-1}\log^{50} n$, was settled by the first two papers to use the Iterative Absorption method \cite{KKO15, KO13}.  
The analogue of Frieze and Krivelevich's conjecture for random regular graphs was posed by K\"uhn and Osthus~\cite{KO14} and remains open.
Other packing problems in random graphs include packing spanning trees, or packing perfect matchings and Hamilton cycles in directed graphs and hypergraphs \cite{FLMN17, FS19, FV18, FKL17}.

In a 2007 survey, Yuster~\cite{Yu07} posed conjectures about both $K_q$-packings in $G(n,p)$ with small leave and $K_q$-decompositions of random regular graphs.  Essentially, he conjectured that if the expected density of the random structure is large enough so that \aas every edge is in a copy of $K_q$, then \aas it is possible to optimally pack copies of $K_q$ in the random structure (see Conjectures~\ref{conj:YusterGnp} and \ref{conj:YusterGnd}).
Despite the many advances with the absorption method in recent years, very little progress has been made on these conjectures.  
The Randomized Algebraic approach of Keevash~\cite{K14} and the Iterative Absorption approach of Glock, K\"uhn, Lo, and Osthus~\cite{GKLO16} for constructing designs have limited applicability for packing cliques in $G(n,p)$ or $G_{n,d}$ when $p$ or $d$ is small.  
Recently, the first and third authors~\cite{DPI} introduced a technique called ``Refined Absorption'' and provided a new proof of the Existence Conjecture.  In this paper, we use this technique to significantly improve the bounds for Yuster's conjectures.

\subsection{Clique Packings and Decompositions of Random Graphs}

For triangle packings in $G(n,p)$, Yuster conjectured that $\sqrt{(\log n) / n}$ is a \textit{sharp threshold} function for the property that $G(n,p)$ has a $K_3$-packing with leave of size at most $3n$, as follows.

\begin{conj}[Yuster \cite{Yu07}]\label{conj:YusterGnp}
    If $p \geq (1 + \eps)\sqrt{(\log n) / n}$, then asymptotically almost surely $G(n, p)$ has a $K_3$-packing containing all but at most $3n$ edges.
\end{conj}

A $d$-regular $n$-vertex graph is $K_3$-divisible provided $d$ is even and at least one of $d$ and $n$ is divisible by 3. Here then is Yuster's conjecture for triangle decompositions of random regular graphs.

\begin{conj}[Yuster \cite{Yu07}]\label{conj:YusterGnd}
    If $d \gg \sqrt{n \log n}$ and $d$ is even, then asymptotically almost surely $G_{n,d}$ has a $K_3$-decomposition provided $3 \mid d \cdot n$.
\end{conj}

In both conjectures, the value of $d$ or $p$ is essentially the smallest one that still guarantees that every edge is in a triangle. 
Yuster also generalized these conjectures to $K_q$-decompositions.
Namely Yuster asked for the threshold for the random $d$-regular graph $G_{n,d}$ to have a $K_q$-decomposition (assuming here that $n$ and $d$ satisfy the necessary divisibility conditions, namely $(q-1)~|~d$ and $q~|~d\cdot n$), and it seems like this threshold should coincide with the threshold for every edge being in a $K_q$, namely $d \gg \left(\log^{1/\left(\binom{q}{2}-1\right)}n\right)\cdot n^{(q-1)/(q+1)}$.  Similarly, Yuster asked for the threshold for $G(n, p)$ to have a $K_q$-packing containing all but at most $qn$ edges, which we expect coincides with the threshold for every edge to be in a $K_q$, so $p \gg \left(\log^{1/\left(\binom{q}{2}-1\right)}n\right) / n^{2/(q+1)}$.  Here we are using that $(q - 2)/\left(\binom{q}{2} - 1\right) = 2/(q + 1)$.

Now one might wonder about the value of $3n$ (or more generally $qn$) in Conjecture~\ref{conj:YusterGnp}. 
Since $G(n,p)$ \aas contains $\Omega(n)$ vertices with degree not divisible by $q - 1$, the size of the leave must at least be linear in $n$.
It seems that Yuster conjectured $3n$ as a simple natural bound that was easy to state or generalize. 
We discuss what the optimal bounds could be in the conclusion. 

Yuster's conjectures are still wide open.  When Keevash~\cite{K14} proved the Existence Conjecture, he also proved that $G(n, 1/2)$ \aas has a triangle packing containing all but at most $n/4 + o(n)$ edges, and his work on decomposing quasirandom / typical graphs implies that $G_{n,d}$ has a $K_q$-decomposition for $d \geq n^{1 - \eps}$ for some unspecified $\eps > 0$.  It is conceivable his methods could also be used to provide a similar result in the $G(n,p)$ setting for $p \geq n^{-\eps}$.  However, the value of $\eps$ is likely far from optimal, at best, say $\eps \leq 2^{-\Omega(q)}$.

We make significant progress on the conjectures of Yuster using the method of refined absorbers, introduced by the first and third authors in~\cite{DPI}, proving a $p$ value that seemed out of reach when using previous methods.
Here is our first result. 

\begin{thm}\label{thm:YusterTriangle}
If $p\ge n^{-\frac{1}{3}+\beta}$ for some $\beta > 0$, then asymptotically almost surely $G(n,p)$ has a $K_3$-packing containing all but at most $n+O(1)$ edges.
\end{thm}

We prove a similar theorem for random regular graphs as follows. 

\begin{thm}\label{thm:YusterTriangleRegular}
If $d \ge n^{\frac{2}{3}+\beta}$ for some $\beta > 0$ and $d$ is even, then asymptotically almost surely $G_{n,d}$ has a triangle decomposition provided $3 \mid d \cdot n$.
\end{thm}

As for general $q$, we also make significant progress on these questions. Here the numbers are slightly worse than the triangle case (even comparatively); this is because the known absorber constructions are comparatively less sparse.

\begin{thm}\label{thm:YusterRandom}
Let $q \ge 4$ be an integer. If $p\ge n^{-\frac{1}{q+0.5}+\beta}$ for some $\beta > 0$, then asymptotically almost surely $G(n,p)$ has a $K_q$-packing containing all but at most $(q-2)n+O(1)$ edges. 
\end{thm}

Similarly, here is our result for general $q$ for random regular graphs.

\begin{thm}\label{thm:YusterRandomRegular}
Let $q \ge 4$ be an integer. If $d\ge n^{1-\frac{1}{q+0.5}+\beta}$ for some $\beta > 0$ and $(q-1)~|~d$, then asymptotically almost surely $G_{n,d}$ has a $K_q$-decomposition provided $q \mid d \cdot n$. 
\end{thm}

It is worth noting that Theorem~\ref{thm:YusterRandomRegular} in the case of $d=n-1$ implies Wilson's~\cite{WiI, WiII, WiIII} result that the Existence Conjecture holds for $r = 2$.


\subsection{Absorber Density Thresholds}

In fact, we derive these theorems as corollaries of more general results where the thresholds are in terms of the maximum density of absorber constructions. To that end, we recall the definition of a $K_q$-absorber as follows.

\begin{definition}[Absorber]
Let $L$ be a $K_q$-divisible hypergraph. A hypergraph $A$ is a \emph{$K_q$-absorber} for $L$ if $V(L)\subseteq V(A)$ is independent in $A$ and both $A$ and $L\cup A$ admit $K_q$-decompositions.
\end{definition}

Next we need a notion of density for absorbers. Here it is natural to define a rooted density since the absorber $A$ will be `rooted at' $V(L)$. In particular, when we seek to embed private absorbers, the roots $V(L)$ will already be fixed. 

\begin{definition}[Maximum rooted density]\label{def:MaxRootedDensity}
Let $H$ be a graph and $R\subseteq V(H)$ such that $R$ is an independent set of $H$. We define the \emph{maximum rooted density} of $H$ at $R$ as
$$m(H,R):= \max\left\{ \frac{e(H')}{|V(H')\setminus R|}:~H' \subseteq H,~V(H')\setminus R \ne \emptyset\right\}.$$
\end{definition}

An important component of our proof is an embedding theorem (Theorem~\ref{thm:Embed}), in which we embed $K_q$-absorbers in $G(n,p)$.  For this result, we need to control the codegrees, that is the number of absorbers containing some fixed set of edges. For this, it is also necessary to have an upper bound on the maximum $2$-density which is defined as follows.

\begin{definition}[Maximum $2$-density]
 Let $H$ be a graph. The \emph{maximum $2$-density} of $H$ is
 $$m_2(H):= \max \left\{\frac{e(H')-1}{v(H')-2}:~H'\subseteq H,~v(H')\ge 3\right\}.$$
 Let $R\subseteq V(H)$ such that $R$ is an independent set of $H$. 
We define the \emph{maximum rooted $2$-density} of $H$ at $R$ as
$$m_2(H,R):= \max\left\{~m(H,R),~m_2(H)~\right\}.$$
\end{definition}

Note then that we defined a maximum rooted $2$-density as the maximum of these required densities. Moreover, it turns out this is the natural rooted notion of $2$-density as it in fact concatenates in a rooted sense (see Lemma~\ref{lem:2DensityConcatenate}). Now we may define a density threshold for absorbers as follows.

\begin{definition}[Absorber rooted $2$-density threshold]
The \emph{$K_q$-absorber rooted $2$-density threshold}, denoted $d_{{\rm abs}}(K_q)$, is the infimum over all positive real numbers $r$ such that for all positive integers $C_1$, there exists a positive integer $C_2$ such that for every $K_q$-divisible graph $L$ on at most $C_1$ vertices, there exists a $K_q$-absorber $A$ of $L$ on at most $C_2$ vertices where $m_2(A,V(L)) \le r$.  
\end{definition}

Here then are our general theorems.

\begin{thm}\label{thm:YusterRandomGeneral}
Let $q \ge 3$ be an integer, and let $a:= \max\left\{ d_{{\rm abs}}(K_q),~ \frac{q+1}{2}\right\}$. If $p\ge n^{-\frac{1}{a}+\beta}$ for some $\beta > 0$, then asymptotically almost surely $G(n,p)$ has a $K_q$-packing containing all but at most $(q-2)n+O(1)$ edges.
\end{thm}

\begin{thm}\label{thm:YusterRandomRegularGeneral}
Let $q \ge 3$ be an integer, and let $a:= \max\left\{ d_{{\rm abs}}(K_q),~ \frac{q+1}{2}\right\}$. If $d\ge n^{1-\frac{1}{a}+\beta}$ for some $\beta > 0$, and $q-1~|~d$, then asymptotically almost surely $G_{n,d}$ has a $K_q$-decomposition provided $q \mid d \cdot n$. 
\end{thm}

In light of Theorems~\ref{thm:YusterRandomGeneral} and~\ref{thm:YusterRandomRegularGeneral}, it is natural then to wonder what the value of $d_{\rm abs}(K_q)$ is for each each integer $q\ge 3$. One avenue to approach this is to upper bound $2$-density by the notion of rooted degeneracy which we now recall. 

\begin{definition}[Rooted degeneracy]
For a graph $H$ and $U\subseteq V(H)$, the \emph{degeneracy of $H$ rooted at $U$} is the smallest nonnegative integer $d$ such that there exists an ordering $v_1,\ldots, v_{v(H)-|U|}$ of the vertices $V(H)\setminus U$ such that for all $i\in [v(H)-|U|]$, we have $|N_H(v_i)\cap (U\cup \{v_j: 1\le j < i\})| \le d$.  
\end{definition}

The two notions are related as one upper bounds the other as follows: If the degeneracy of $H$ rooted at $R$ is at most $d$, then $m_2(H,R)\le d$. Thus one might study the related \emph{$K_q$-absorber rooted degeneracy threshold} wherein rooted degeneracy replaces maximum $2$-density. Finding $K_q$-absorbers with small rooted degeneracy has already been studied by Barber, K\"uhn, Lo, and Osthus \cite{BKLO16} for use in proving that $K_q$-decompositions of $K_q$-divisible graphs of large minimum degree exist. They (\cite[Corollary 8.10, Lemma 8.11, Lemma 12.3]{BKLO16}) proved that $K_3$-absorbers of rooted degeneracy at most $4$ exist, which thus implies that $d_{\rm abs}(K_3)\le 4$; similarly for each integer $q\ge 4$, they proved that $K_q$-absorbers of rooted degeneracy at most $3q-3$ exist, which implies that $d_{\rm abs}(K_q)\le 3q-3$. 

Rooted degeneracy is the natural parameter to consider for an absorption-based approach to finding $K_q$-decompositions of graphs of large minimum degree.  In a similar way, maximum rooted $2$-density seems to be the natural parameter to consider for an absorption-based approach to finding $K_q$-decompositions of random graphs.

In this paper (see Section~\ref{s:Construction}), we also provide new constructions of $K_q$-absorbers which have smaller maximum rooted $2$-density as follows.

\begin{thm}\label{thm:GeneralTriangleAbsorberThreshold}
$d_{{\rm abs}}(K_3) \le 3$.    
\end{thm}

\begin{thm}\label{thm:GeneralQAbsorberThreshold}
For $q\ge 4$, $d_{{\rm abs}}(K_q) \le q+ 0.5$.    
\end{thm}

Then Theorems~\ref{thm:YusterTriangle},~\ref{thm:YusterTriangleRegular},~\ref{thm:YusterRandom} and~\ref{thm:YusterRandomRegular} follow immediately from Theorems~\ref{thm:YusterRandomGeneral} and~\ref{thm:YusterRandomRegularGeneral} when combined with the above results. In the conclusion, we provide a lower bound for $d_{\rm abs}(K_q)$ of $q-1$ and conjecture the value is $q$; hence our results in Theorems~\ref{thm:GeneralTriangleAbsorberThreshold} and~\ref{thm:GeneralQAbsorberThreshold} are nearly optimal.

\subsection{Fractional Decompositions}

Yuster~\cite[Problem 4.10]{Yu07} also posed the natural question of determining the sharp threshold for $G(n,p)$ to have a \textit{fractional} triangle decomposition (see Definition~\ref{def:fractional-decomposition}). 
Again $p$ needs to be at least $\sqrt{(\log n)/n}$, the sharp threshold for every edge being in a triangle. This question can be viewed as a precursor question to those in Conjectures~\ref{conj:YusterGnp} and~\ref{conj:YusterGnd} since the existence of a fractional decomposition is a necessary precursor to the existence of an integral decomposition. In addition, for fractional decompositions there is no need for divisibility assumptions or the existence of a leave. That means however though that the corresponding fractional version does not follow from Theorems~\ref{thm:YusterTriangle} and~\ref{thm:YusterTriangleRegular}.

With a small bit of extra work, the proof of Theorem~\ref{thm:YusterRandomGeneral} can be adapted to the fractional setting as follows.

\begin{thm}\label{thm:YusterRandomFractionalGeneral}
Let $q \ge 3$ be an integer, and let $a:= \max\left\{ d_{{\rm abs}}(K_q),~ \frac{q+1}{2}\right\}$. If $p\ge n^{-\frac{1}{a}+\beta}$ for some $\beta > 0$, then asymptotically almost surely $G(n,p)$ has a fractional $K_q$-decomposition.
\end{thm}

When combined with Theorems~\ref{thm:GeneralTriangleAbsorberThreshold} and~\ref{thm:GeneralQAbsorberThreshold}, this yields the following theorems.

\begin{thm}\label{thm:YusterTriangleFractional}
If $p:=n^{-\frac{1}{3}+\beta}$ for some $\beta > 0$, then asymptotically almost surely $G(n,p)$ has a fractional $K_3$-decomposition.
\end{thm}

\begin{thm}\label{thm:YusterFractional}
Let $q \ge 4$ be an integer. If $p\ge n^{-\frac{1}{q+0.5}+\beta}$ for some $\beta > 0$, then asymptotically almost surely $G(n,p)$ has a fractional $K_q$-decomposition. 
\end{thm}

In the next section, we overview the proofs of Theorems~\ref{thm:YusterRandomGeneral} and~\ref{thm:YusterRandomRegularGeneral} and then outline the remainder of the paper.  See Section~\ref{s:notation} for notation and definitions used throughout the paper.

\section{Proof Overview}\label{s:Overview}

Theorems~\ref{thm:YusterRandom} and~\ref{thm:YusterRandomGeneral} assert the existence of $K_q$-packings in $G(n,p)$ with only linear leave for quite small values of $p$. Similarly, Theorems~\ref{thm:YusterRandomRegular} and~\ref{thm:YusterRandomRegularGeneral} assert the existence of $K_q$-decompositions of $G_{n,d}$ for quite small $d$ (provided $d$ and $n$ satisfy the necessary divisibility conditions). It is natural to inquire how the existential proofs for $(n,q,2)$-Steiner systems fare in these settings.
R\"odl's nibble method \cite{AKS97, Vu00} yields a $K_q$-packing with leave
$\max\left\{(pn^2)(np^2)^{-1/\left(\binom{q}{2}-2\right)},~n^{2 - 2/(q + 1)}\right\}\cdot \log^{O(q)}n$
and \cite{KKMO23} does slightly better, but these bounds are far from linear in $n$.
Wilson's proof~\cite{WiI,WiII,WiIII} uses a recursive algebraic construction and hence is not amenable due to its lack of randomness. 
Keevash's proof~\cite{K14} utilizes ``random algebraic constructions''; indeed, he proved a typicality version of the Existence Conjecture and since random graphs are also typical, his work yields a $K_q$-decomposition of $G_{n,d}$ with high probability (with $n$ and $d$ satisfying the necessary divisibility conditions) but only when $d\ge n^{1-\varepsilon}$ for some $\varepsilon > 0$ which is far from optimal (namely $\varepsilon$ is at most $2^{-\Omega(q)}$ say). The iterative absorption proof of Glock, K\"uhn, Lo, and Osthus~\cite{GKLO16} also yields a typicality version but only for $d$ linear in $n$.

Thus the key to the proof of our main results is to invoke the method of ``refined absorption'' developed by the first and third authors~\cite{DPI} in their new proof of the Existence Conjecture. While the proof does not immediately carry over to our random settings, we are fortunately able to use the key technical theorem from~\cite{DPI} (see Theorem~\ref{thm:omni-absorber} below) as a black box. The novel (and certainly still non-trivial) work of this paper then is how to utilize the template provided by refined absorption and apply it in these random settings. Important to that is that the parameters in the black box theorem work for our small values of $p$ and $d$.  

\subsection{Proof of Existence via Refined Absorption}

To that end, we first outline the proof of Existence from~\cite{DPI} as well as the key technical definition and black box theorem before discussing the modifications that allow us to extract the true power of this new proof. The first key concept from~\cite{DPI} is that of an `omni-absorber', an object that absorbs \emph{all} possible leftovers.

\begin{definition}[Omni-Absorber]
Let $q \geq 3$ be an integer. Let $X$ be a graph. We say a graph $A$ is a \textit{$K_q$-omni-absorber} for $X$ with \emph{decomposition family} $\mathcal{H}$ and \emph{decomposition function} $\mathcal{Q}_A$ if $V(X)=V(A)$, $X$ and $A$ are edge-disjoint, $\mathcal{H}$ is a family of subgraphs of $X\cup A$ each isomorphic to $K_q$ such that $|E(H)\cap E(X)|\le 1$ for all $H\in\mathcal{H}$, and for every $K_q$-divisible subgraph $L$ of $X$, there exists $\mathcal{Q}_A(L)\subseteq \mathcal{H}$ that are pairwise edge-disjoint and such that $\bigcup \mathcal{Q}_A(L)=L\cup A$. 
\end{definition}

The next key concept is that of \emph{refinement} of an omni-absorber, that is that every edge of $X\cup A$ is in only constantly many cliques of the decomposition family as follows.

\begin{definition}[Refined Omni-Absorber]
Let $C\ge 1$ be real. We say a $K_q$-omni-absorber $A$ for a graph $X$ with decomposition family $\mathcal{H}$ is \emph{$C$-refined} if $|\{H\in \mathcal{H} : e\in E(H) \}| \le C$ for every edge $e\in X\cup A$.
\end{definition}

Here is the graph version of the main omni-absorber theorem from~\cite{DPI} (which was more generally proved for hypergraphs).

\begin{thm}\label{thm:omni-absorber}
    For each integer $q \ge 3$, there exists a real $C\ge 1$ such that the following holds: If $X$ is a spanning subgraph of $K_n$ with $\Delta(X) \le \frac{n}{C}$ and we let $\Delta:= \max\left\{\Delta(X),\sqrt{n}\cdot \log n\right\}$, then there exists a $C$-refined $K_q$-omni-absorber $A \subseteq K_n$ for $X$ such that $\Delta(A)\le C \cdot \Delta$.
\end{thm}

These linearly efficient omni-absorbers permitted a more streamlined proof of the Existence Conjecture in~\cite{DPI} whose outline we now recall but restricted to the graph case for simplicity:

\begin{enumerate}
    \item[(1)] `Reserve' a random subset $X$ of $E(K_n)$.
    \item[(2)] Construct an omni-absorber $A$ of $X$.
    \item[(3)] ``Regularity boost'' $K_n\setminus (A\cup X)$.
    \item[(4)] Apply ``nibble with reserves'' theorem to find a $K_q$-packing of $K_n\setminus A$ covering $K_n\setminus (A\cup X)$ and then extend this to a $K_q$-decomposition of $K_n$ by definition of omni-absorber.  
\end{enumerate}


We now discuss the difficulties in adapting this proof to our random settings.  
Indeed, there are numerous difficulties in adapting this proof to both the $G(n,p)$ and $G_{n,d}$ setting. Probably the most significant difficulty lies in Step (2); namely, we need to embed an omni-absorber inside $G(n,p)$ or $G_{n,d}$. Step (1) is more challenging than before as the random subset is taken inside $G(n,p)$. For Step (3), we actually regularity boost in $K_n$ and then show this boosting carries down to $G(n,p)$.  This step is quite intricate, involving a more general version of the Boost Lemma from~\cite{GKLO16} and concentrating how this carries down to $G(n,p)$. We still finish by applying Step (4), nibble with reserves, as before to the set of cliques given by regularity boosting.
However, in $G(n,p)$ we also have to fix the divisibility in a manner that deletes few edges, and in $G_{n,d}$, although we can perform the aforementioned steps inside some $G(n,p_*) \subseteq G_{n,d}$ for $p_* \geq (1-o(1))d/n$ using results on the Sandwich Conjecture \cite{KV04, GIM20, GIM22} (see Conjecture~\ref{conj:SandwichConj} and Theorem~\ref{thm:SandwichResult}), we still need to cover $G_{n,d}\setminus G(n,p_*)$.  Since the most complicated issue is how to find an omni-absorber as in Theorem~\ref{thm:omni-absorber} \emph{inside} $G(n,p)$, we discuss this first.  

\subsection{Finding an Omni-Absorber inside $G(n,p)$}

Here is our main theorem for finding an omni-absorber inside $G(n,p)$. 

\begin{thm}\label{thm:RandomOmniAbsorber}
    For each integer $q \ge 3$ there exists a real $C\ge 1$ such that the following holds for all $\beta > 0$.  Let $a \coloneqq \max\{d_{\rm abs}(K_q), \frac{q + 1}{2}\}$, and let $X$ be a spanning subgraph of $K_n$ with $\Delta(X)\le \frac{n}{C}$. If $p\ge n^{-\frac{1}{a} + \beta}$ and $\Delta(X) \le \frac{pn}{C}$, then asymptotically almost surely there exists a $C$-refined $K_q$-omni-absorber $A$ for $X$ such that $A \subseteq (K_n\setminus X)_p$. 
\end{thm}

With the ideas from this paper, it may be possible to amend every step of the proof of Theorem~\ref{thm:omni-absorber} to be done inside $G(n,p)$ so as to prove Theorem~\ref{thm:RandomOmniAbsorber}.
However, part of the power of Theorem~\ref{thm:omni-absorber} is that it may be used as a black box.
For our proof of Theorem~\ref{thm:RandomOmniAbsorber}, we will view the omni-absorber provided by Theorem~\ref{thm:omni-absorber} as a template wherein we replace each element of its decomposition family with gadgets that we embed inside $G(n,p)$. To that end, we follow a similar strategy as developed in~\cite{DPI} for embeddings. We start by, for each edge of $A$, embedding a gadget we call a `fake edge' (see Definition~\ref{def:BasicGadget}) which has the same divisibility properties as an actual edge. We only use `fake edges' that lie completely inside $G(n,p)$, and we make the choice randomly so as to ensure there is a choice of disjoint fake edges via the Lov\'asz Local Lemma. 
In fact, we actually only use ones inside a small slice $G(n,p_1)$ of $G(n,p)$ for some well-chosen $p_1 < p$.  
This ensures the overall degree is still small (roughly at most $p_1 \cdot n$).  
At this point, the collection of fake edges can be viewed as a \textit{refiner}, a key concept in the proof of Theorem~\ref{thm:omni-absorber} from \cite{DPI}.  An omni-absorber is itself already special type of refiner, but the difference here is that the fake edges form a refiner which is embedded in $G(n,p)$, while the omni-absorber from Theorem~\ref{thm:omni-absorber} is not.
With that accomplished, we then embed private absorbers for each member of the new decomposition family (where edges are replaced by their corresponding fake edges). 

For both of these embedding problems, we use a general embedding theorem we developed (Theorem~\ref{thm:Embed}). The needed $p$ value for said theorem is related (up to polylogarithmic terms) to the maximum rooted 2-density (see Definition~\ref{def:MaxRootedDensity}) of the objects to embed. For fake edges, this density is $\frac{q+1}{2}$ and hence requires $p_1 > n^{-2/(q+1)}$, which is also (up to polylogarithmic terms) the threshold from Yuster for a given edge to be in a clique. For absorbers, the optimal density is unclear, but we show via new absorber constructions that it is at most $3$ for $K_3$-absorbers and $q+\frac{1}{2}$ for $K_q$-absorbers, yielding values of $p$ of $n^{-1/3}$ and $n^{-1/\left(q+\frac{1}{2}\right)}$ (with some small additional error terms). 

Here it is crucial that the omni-absorber is $C$-refined so as to manage how many private absorbers and fake edges we have to embed. It is also important how efficient the omni-absorber is in terms of maximum degree. 
Note that if $\cH$ is the decomposition family of a $K_q$-omni-absorber $A$ for $X \subseteq K_n$ where $\Delta(X \cup A) \leq \Delta$, then each vertex is in at most $\Delta \cdot C$ elements of $\cH$ and $|\cH| \leq n\cdot \Delta\cdot C$.
Any decrease in efficiency will result in a corresponding increase in the value of $p$ (since the value of $p$ is tied to the maximum degree of $A$ rather than to $X$). The omni-absorbers of Theorem~\ref{thm:omni-absorber} are linearly efficient, that is $\Delta(A)\le O(\Delta(X))$, but only when $\Delta(X) \ge \sqrt{n}\cdot \log n$. This suffices for our purpose since we will have $\Delta(X) \sim pn$ and $pn \ge n^{1-\frac{2}{q+1} + \beta} \gg \sqrt{n}\cdot \log n$.
Note that in Theorem~\ref{thm:RandomOmniAbsorber}, we do not include a bound on the maximum degree of $A$ is in Theorem~\ref{thm:omni-absorber}.  We could prove such a bound by using a slightly more involved Lov\'asz Local Lemma-based strategy from \cite{DKPIV}; however, we do not need it here because we can simply use the simple fact that \aas $\Delta((K_n\setminus X)_p) \leq 2pn$.

Again, previous proof methods would only yield values of $p=n^{-\varepsilon}$ for some small undetermined $\varepsilon > 0$; even with high efficiency omni-absorbers (which were previously not known to exist other than for triangle decompositions), iterative absorption would only yield a value depending on that efficiency. With our extremely efficient omni-absorbers, our values are now much closer to the conjectured ones. Refined absorption though was key here to actually embed these new efficient omni-absorbers inside $G(n,p)$. In fact, the new bottleneck for this approach now only comes from embedding absorbers into $G(n,p)$. 

\subsection{Regularity Boosting}

Technically, there is a subtlety in the regularity boosting in Step (3). 
To prove Theorem~\ref{thm:YusterRandomGeneral}, we first reveal a small slice $S \sim G(n,p') \subseteq G(n,p)$ for a carefully chosen $p' < p$.  Within $S$, we can find the reserve $X$ and embed an omni-absorber $A$ for $X$ in $S$ using Theorem~\ref{thm:RandomOmniAbsorber}.
Finally, we reveal the remainder $G(n,p)\setminus S$ which is distributed as $(K_n \setminus S)_{p''}$ for $p'' \coloneqq (p - p')/(1 - p')$. 
To apply the ``nibble with reserves'' theorem (see Theorem~\ref{thm:NibbleReserves} and Definitions~\ref{def:design-hypergraph} and \ref{def:reserve-design}), we need to find a very regular set of cliques of $G(n,p)\setminus (X \cup A)$.  
If we only needed to find a regular set of cliques of $(K_n\setminus S)_{p''}$, this would follow by finding a very regular set of cliques of $K_n\setminus S$ (which is possible by the Boost Lemma from \cite{GKLO16}) and showing that the subset contained in $(K_n\setminus S)_{p''}$ is still very regular (which would follow from standard probabilistic arguments).
However, it is necessary to include $S':=S\setminus (X \cup A)$ with the last slice $(K_n\setminus S)_{p''}$ and find a very regular set of cliques of $S'\cup (K_n\setminus S)_{p''}$ so as to apply nibble.

The difficulty for us lies in the asymmetry between $S'$ and $(K_n\setminus S)_{p''}$; that is, $S'$ has been revealed and hence is considered fixed while the latter is random. If one builds a very regular set of cliques of $(K_n\setminus S)\cup S'$, then this asymmetry causes the set to become too irregular after revealing the last slice. One possible solution is to apply the ideas of the Boost Lemma from~\cite{GKLO16} directly inside $(K_n\setminus S)_{p''}\cup S'$; however this seemed to require that $p\geq n^{-\frac{1}{q+.5}+\beta}$ which would be acceptable for the proof of Theorem~\ref{thm:YusterRandom} but not for Theorem~\ref{thm:YusterTriangle}. 

Our solution then is to use the ideas of the Boost Lemma proof to find a fractional $K_q$-packing of $(K_n\setminus S)\cup S'$ with a special property: instead of finding a fractional $K_q$-decomposition (i.e. where all edges have the same weight), we find a $K_q$-packing such that the \emph{expected} weight of each edge in $(K_n\setminus S)_{p''}\cup S'$ is the same. Then we show the actual weight in $(K_n\setminus S)_{p''}\cup S'$ while not the exact same is highly concentrated. Finally we sample from the associated distribution to find our desired very regular set of cliques. 

\subsection{The Leave and Random Regular Graphs}

Technically, we also have to be more careful in Step (4) for $G(n,p)$ since $G(n,p)$ is highly unlikely to be $K_q$-divisible. To fix this, we first set aside a \emph{$K_q$-divisibility fixer} (see Definition~\ref{def:DivFixer}) which has the property that for any graph $G$ for which $F$ is a spanning subgraph, there exists $F'\subseteq F$ such that $G-F'$ is $K_q$-divisible (see Proposition~\ref{prop:DivFixer}). We show in Lemma~\ref{lem:DivFixer} that such a divisibility fixer $F$ exists with at most $(q - 2)n + O(1)$ edges and in addition may be found inside $G(n,p_0)$ for some small $p_0 \ll p$. To decompose whatever edges of $F$ are not deleted to fix divisibility, we in fact take $A$ to be an omni-absorber of $G(n,p_0)\cup X$. 

Similarly, random regular graphs require more care in Step (4) despite being $K_q$-divisible. While the results on the Sandwich Conjecture allow us to find $G(n,p_*)$ inside $G_{n,d}$ with $p_* = (1-o(1))\cdot \frac{d}{n}$, the sandwich is not tight enough to guarantee that the regularity boosted cliques are still regular enough to apply nibble if we include the edges of $G_{n,d}\setminus G(n,p_*)$. Instead, we divide the reserve $X$ randomly into two sets $X_1$ and $X_2$. We use $X_1$ as the reserve for nibble as before but we now use $X_2$ to decompose the edges of the `lower filling' of the sandwich, that is $G_{n,d} - G(n,p_*)$. This step may be viewed as ``covering down'' $G_{n,d}\setminus G(n,p_*)$ to $X_2$, the remainder of which can then be absorbed.  Note this is only possible since we may choose $X$ to be a constant proportion larger than said filling, which is only possible since the omni-absorbers of Theorem~\ref{thm:omni-absorber} have linear maximum degree.

\subsection{Outline of Paper}\label{ss:Outline}


In Section~\ref{s:Previous}, we recall various necessary previous results for our main proofs. In Section~\ref{s:Embed}, we prove an Embedding Theorem (Theorem~\ref{thm:Embed}). In Section~\ref{s:RandomOmniAbsorber}, we prove that refined omni-absorbers exist inside $G(n,p)$ using said Embedding Theorem. In Section~\ref{s:DivFixer}, we prove the existence of divisibility fixers and that they exist inside $G(n,p)$ for large enough $p$ (which also requires the Embedding Theorem). In Section~\ref{s:Regularity}, we prove the required regularity boosting theorem. In Section~\ref{s:Random-Reserve}, we prove a random reserve lemma for Step (1) above.
In Section~\ref{s:main-proof}, we prove Theorems~\ref{thm:YusterRandomGeneral} and \ref{thm:YusterRandomRegularGeneral} about $K_q$-packings and decompositions of $G(n,p)$ and $G_{n,d}$, respectively, modulo sparse absorber constructions; these constructions are provided in Section~\ref{s:Construction}.

In Section~\ref{s:Conclusion}, we provide concluding remarks, in particular we discuss the optimal leave for $G(n,p)$ and mention some future directions.

\section{Preliminaries}\label{s:Previous}

In this section, we introduce some definitions and collect the other previous results we need.

\subsection{Notation and Definitions}\label{s:notation}

A hypergraph $\cH$  consists of a set $V(\cH)$ whose elements are called \textit{vertices} and a set $E(\cH)$ of subsets of $V(\cH)$ called \textit{edges}; for brevity, we also sometimes write $\cH$ for its set of edges $E(\cH)$ and for a set $U \subseteq E(\cH)$, we sometimes write $V(U)$ for the set of vertices $\bigcup_{e\in U}e$. Similarly, we write
$v(\cH)$ for the number of vertices of $\cH$ and either $e(\cH)$ or alternatively $|\cH|$ for the number of edges of $\cH$.
A \textit{multi-hypergraph} consists of a set of vertices $V(\cH)$ and a multi-set $E(\cH)$ of subsets of $V$. 
For an integer $r \geq 1$, a \textit{(multi-)hypergraph} $\cH$ is said to be \textit{$r$-bounded} if every edge of $\cH$ has size at most $r$ and \textit{$r$-uniform} if every edge has size exactly $r$.
For a set $S \subseteq V(\cH)$, we let $\cH(S) \coloneqq \{e \in E(\cH) : e \supseteq S\}$.  Note that $\cH(S)$ is commonly used to refer to the \textit{link hypergraph}, which would be $\{e \setminus S : e \in E(\cH), e \supseteq S\}$, but we found our notation to be more convenient.  
We denote the \textit{degree} of a set $S \subseteq V(\cH)$ as $d_\cH(S) \coloneqq |\cH(S)|$, and for $i \in \mathbb N$, we let $\Delta_i(\cH) \coloneqq \max_{S \in \binom{V(\cH)}{i}}d_\cH(S)$.
For a graph $G$, we let $\Delta(G) \coloneqq \Delta_1(G)$.
We also denote the \textit{degree} of a vertex $v \in V(\cH)$ as $d_\cH(v) \coloneqq d_\cH(\{v\})$.  A hypergraph $\cH$ is \textit{$D$-regular} if $d_\cH(v) = D$ for all $v \in V(\cH)$ and \textit{$(1 \pm \eps)D$-regular} if $d_\cH(v) = (1 \pm \eps)D$ for all $v \in V(\cH)$ (for reals $a,b,c$, we write $a = b\pm c$ to mean $a \in [b - c, b + c]$).
We let $\cH[S]$ denote the hypergraph of $\cH$ \textit{induced by} $S$ where $V(\cH[S]) \coloneqq S$ and $E(\cH[S]) \coloneqq \{e \in \cH : e \subseteq S\}$.

An $r$-uniform hypergraph $\cH$ is \textit{complete} if $E(\cH) = \binom{V(\cH)}{r}$, and a \textit{clique} in $\cH$ is a set $S \subseteq V(\cH)$ such that $\cH[S]$ is complete.  We call a clique $S$ of $\cH$ a $q$-clique if $|S| = q$.
We let $K^r_q$ denote the complete $r$-uniform hypergraph with vertex set $[q]$.  

We assume $G(n,p)$ has vertex set $[n]$.  For a graph $G$, we let $G_p$ denote the subgraph of $G$ obtained by choosing the edges of $G$ independently with probability $p$, so $G(n,p) = (K_n)_p$.

\subsection{Hypergraph Matchings}

A \textit{matching} in a hypergraph $\cH$ is a set of pairwise disjoint edges of $\cH$.  Here we collect several results guaranteeing certain types of matchings in hypergraphs.  
First we need the following definition.

\begin{definition}[Bipartite hypergraph]
We say a hypergraph $\cG=(A,B)$ is \emph{bipartite with parts $A$ and $B$} if $V(\cG)=A\cup B$ and  every edge of $\cG$ contains exactly one vertex from $A$. We say a matching of $\cG$ is \emph{$A$-perfect} if every vertex of $A$ is in an edge of the matching.
\end{definition}

For Step (4), we require the ``nibble with reserves'' theorem from~\cite{DP22}, as follows.

\begin{thm}\label{thm:NibbleReserves}
For each integer $r \ge 2$ and real $\beta \in (0,1)$, there exist an integer $D_{\beta}\ge 0$ and real $\alpha > 0$ such that following holds for all $D\ge D_{\beta}$: 
\vskip.05in
Let $\cG$ be an $r$-uniform (multi)-hypergraph with $\Delta_2(\cG) \leq D^{1-\beta}$ such that $\cG$ is the edge-disjoint union of $\cG_1$ and $\cG_2$ where $\cG_2=(A,B)$ is a bipartite hypergraph such that every vertex of $B$ has degree at most $D$ in $\cG_2$ and every vertex of $A$ has degree at least $D^{1-\alpha}$ in $\cG_2$, and $\cG_1$ is a hypergraph with $V(\cG_1)\cap V(\cG_2) = A$ such that every vertex of $\cG_1$ has degree at most $D$ in $\cG_1$ and every vertex of $A$ has degree at least $D\left(1- D^{-\beta}\right)$ in $\cG_1$.

Then there exists an $A$-perfect matching of $\cG$. 
\end{thm}

We remark that Theorem~\ref{thm:NibbleReserves} is only proved in the latest version of \cite{DP22}, where it is more generally proved  with forbidden submatchings (see Definition~\ref{def:configuration-hypergraph}).  However, the version without forbidden submatchings, as stated in Theorem~\ref{thm:NibbleReserves}, follows from previously known results on ``pseudorandom hypergraph matchings'' (see \cite{KKKMOnibble, EGJ20}) assuming $D \geq v(\cG)^\beta$, which is the case for all of our applications in this paper and in \cite{DKPIV}.

For the proof of our embedding theorem (Theorem~\ref{thm:Embed}), we also need a simpler version which requires a larger factor between the degrees of vertices in $A$ and those in $B$ but that does not require any codegree assumption. We also use this version to deal with the edges of $G_{n,d}$ not contained in the subgraph $G(n,p_*)$ coming from the Sandwich Conjecture in the proof of Theorem~\ref{thm:YusterRandomRegularGeneral}.

The following is a consequence of a more general theorem from~\cite{DP22}, which follows fairly easily from the Lov\'asz Local Lemma.

\begin{thm}\label{thm:EasyPerfectMatching}
Let $\cG=(A,B)$ be a bipartite $r$-bounded multi-hypergraph. If there exists a real $D\ge 1$ such that every vertex of $A$ has degree at least $8rD$ in $\cG$ and every vertex of $B$ has degree at most $D$ in $\cG$, then there exists an $A$-perfect matching of $\cG$.
\end{thm}

We require a simple but useful generalization of the above theorem (at least for large $D$), which is in fact itself a corollary of the above as follows.

\begin{cor}\label{cor:EasyPerfectMatching}
For every integer $r \geq 1$, there exists $D_0$ such that the following holds.
Let $\cG=(A,B)$ be a bipartite $r$-bounded multi-hypergraph. Let $A_1,\ldots, A_k$ be a partition of $A$, and let $\cG_i=\cG[A_i\cup B]$ for all $i\in [k]$. If there exist reals $D_i\ge D_0$ for all $i\in [k]$ such that every vertex of $A_i$ has degree at least $8rkD_i$ in $\cG_i$ and every vertex of $B$ has degree at most $D_i$ in $\cG_i$, then there exists an $A$-perfect matching of $\cG$.
\end{cor}
\begin{proof}
We assume without loss of generality that all the $D_i$ are integer (say by setting $D_i$ to be its floor). Let $D:= \prod_{i\in [k]} D_i$. Let $\cG'$ be obtained from $\cG$ by for each $i\in [k]$ and edge $e\in \cG_i$, adding $\frac{D}{D_i} - 1$ copies of $e$. Now $\cG'$ is a bipartite $r$-uniform multi-hypergraph where every vertex of $A$ has degree at least $8rkD$ in $\cG'$ and every vertex of $B$ has degree at most $kD$ in $\cG'$. 
Hence by Theorem~\ref{thm:EasyPerfectMatching}, there exists an $A$-perfect matching $M$ of $\cG'$. But then $M$ corresponds to an $A$-perfect matching of $\cG$ as desired (by replacing any copy edges with their original).   
\end{proof}

We also need the definition of configuration hypergraph as follows.

\begin{definition}[Configuration hypergraph]\label{def:configuration-hypergraph}
    Let $\cG$ be a (multi)-hypergraph. We say a hypergraph $\cC$ is
a \textit{configuration hypergraph} for $\cG$ if $V(\cC) = E(\cG)$ and $E(\cC)$ consists of a set of matchings of $\cG$ of size at
least two. We say a matching of $\cG$ is \textit{$\cC$-avoiding} if it spans no edge of $\cC$.
\end{definition}

\begin{definition}[$D$-weighted degree]
    Let $\cH$ be a hypergraph and let $D \geq 1$ be a real number. We define the \textit{$D$-weighted degree} of a vertex $v \in V(\cH)$, denoted 
    \begin{equation*}
        w_D(\cH, v) \coloneqq \sum_{i\geq 2}\frac{i-1}{D^{i-1}}\cdot |\{e \in \cH(\{v\}) : |e| = i\}|.
    \end{equation*}
We define the \textit{maximum $D$-weighted degree of $\cH$}, denoted
    \begin{equation*}
        w_D(\cH) \coloneqq \max_{v\in V(\cH)}w_D(\cH, v).
    \end{equation*}
\end{definition}
We will also need \cite[Lemma 3.4]{DP22} which generalizes Theorem~\ref{thm:EasyPerfectMatching} to include a configuration hypergraph of small maximum weighted degree.  We note that the version stated here is stronger in that we have $r$-bounded instead of $r$-uniform and we allow $\cG$ to be a multi-hypergraph, but the same proof also yields this version.  Note we may assume without loss of generality that $\cG$ is $r$-uniform by adding dummy vertices to $B$.
\begin{thm}\label{thm:BipartitePerfectMatching}
Let $r \geq 1$ be an integer. Let $\cG=(A,B)$ be a bipartite $r$-bounded multi-hypergraph such that $d_\cG(a) \geq 300rD$ for every $a \in A$ and $d_\cG(b) \leq D$ for every $b \in B$. 
If $\cC$ is a configuration hypergraph of $\cG$ with $w_D(\cC) \leq 1$, then there exists a $\cC$-avoiding $A$-perfect matching of $\cG$.
\end{thm}

\subsection{Kim-Vu Corollaries}

To prove our embedding theorem (Theorem~\ref{thm:Embed}) that will be used to embed fake edges and private absorbers in $G(n,p)$, we need a certain concentration inequality, namely a corollary of the Kim-Vu~\cite{KV00} polynomial concentration inequality, but first we need a few definitions.

\begin{definition}[Random sparsifications] 
Let $p \in [0,1]$.  
For a set $X$, we let $X_p$ denote the random subset of $X$ obtained by choosing each element of $X$ independently with probability $p$.
For a graph $G$, we let $G_p$ denote the subgraph of $G$ obtained by choosing the edges of $G$ independently with probability $p$. 
\end{definition}

\begin{definition}[Weighted hypergraph]
A \textit{weighted hypergraph} is a pair $(\cF, w)$ where $\cF$ is a hypergraph and $w : E(\cF) \rightarrow \mathbb R_{\geq0}$ is a weight function.  We denote the \textit{weighted degree} of a set $S \subseteq V(\cF)$ in $(\cF, w)$ by $d_{\cF,w}(S) \coloneqq \sum_{e \in \cF(S)}w(e)$, and for $i \in \mathbb N$, we let $\Delta_i(\cF, w) \coloneqq \max_{S \in \binom{V(\cF)}{i}}d_{\cF,w}(S)$.  We let $w(\cF) \coloneqq \sum_{e \in \cF}w(e)$.
\end{definition}

The following theorem follows immediately from the Kim--Vu~\cite{KV00} polynomial concentration inequality.

\begin{thm}\label{thm:KimVu}
    For all $k \geq 1$, there exists $C \geq 1$ such that the following holds for all $p \in (0, 1]$ and $n, \lambda > 1$.  
    If $(\cF, w)$ is a weighted hypergraph such that $\cF$ is $k$-uniform and has at most $n$ vertices and $\cF' \coloneqq \cF[V(\cF)_p]$, then
    \begin{equation*}
        \Prob{|w(\cF') - \Expect{w(\cF')}| > C\lambda^k\sqrt{E_0E_1}} < C\exp(-\lambda + (k - 1)\log n),
    \end{equation*}
    where $E_0 \coloneqq \max_{i\geq 0}\Delta_i(\cF, w) p^{k - i}$ and $E_1 \coloneqq \max_{i\geq 1}\Delta_i(\cF, w) p^{k-i}$.
\end{thm}\COMMENT{%
    We apply Kim--Vu~\cite[Main Theorem]{KV00} to the hypergraph $\cF$, where $t_v$ for each $v\in V(\cF)$ is a $\{0,1\}$ random variable with expected value $p$, and $w(\cF')$ plays the role of $Y_\cF$. For each $A \subseteq \cF$, the (weighted) link hypergraph $\cF(A)$ plays the role of the ``truncated hypergraph'' $\cF_A$, so we have $\Expect{Y_{\cF_A}} = p^{k-i}d_{\cF,w}(A) \leq p^{k-i}\Delta_i(\cF,w)$ for $A$ of size $i$.  Hence, $E'(\cF) = E_1 \coloneqq \max_{i\geq1}\Delta_i(\cF, w)p^{k-i}$ and $E(\cF) = \max_{i\geq0}\Delta_i(\cF, w)p^{k-i}$.
}  
We will use the following two corollaries of Theorem~\ref{thm:KimVu}, the first of which ensures the weight of $\cF'$ is within a polynomial factor of its expectation with high probability, and the second of which ensures the weight of $\cF'$ is within a logarithmic or constant factor of its expectation with high probability, under more relaxed assumptions.

\begin{cor}\label{cor:KimVu-poly}
    For all $k \geq 1$ and $\beta > 0$, there exists $n_0$ such that the following holds for all $p \in (0, 1]$ and $n\ge n_0$.
    Let $(\cF, w)$ be a weighted hypergraph such that $\cF$ is $k$-uniform and has at most $n$ vertices, and let $\cF' \coloneqq \cF[V(\cF)_p]$.
    If $K > 0$ satisfies
    \begin{equation*}
        \Delta_i(\cF, w) \leq {K\cdot p^i}\cdot n^{-\beta}
    \end{equation*}
    for all $i \in [k]$, then with probability at least $1 - n^{-\log n}$, we have
    \begin{equation*}
        w(\cF') = w(\cF)\cdot p^k\left(1 \pm n^{-\frac{\beta}{3}}\cdot \max\left\{\sqrt{\frac{K}{w(\cF)}}, \frac{K}{w(\cF)}\right\}\right).
    \end{equation*}
\end{cor}
\begin{proof}
    Apply Theorem~\ref{thm:KimVu} with $\lambda = 2\log^2 n$.  For sufficiently large $n$ we have
    \begin{equation*}
        C\exp(-\lambda + (k - 1)\log n) < \exp(-\log^2 n) = n^{-\log n}.
    \end{equation*}
    We have
    \begin{equation*}
        E_1 \coloneqq \max_{i\geq 1}\Delta_i(\cF, w)p^{k - i} \leq Kp^kn^{-\beta}
    \end{equation*}
    and 
    \begin{equation*}
        E_0 \coloneqq \max_{i\geq0}\Delta_i(\cF, w)p^{k - i} \leq \max\{\Expect{X}, Kp^kn^{-\beta}\},
    \end{equation*}
    so
    \begin{equation*}
        C\lambda^k\sqrt{E_0E_1} \leq C\lambda^k\Expect{w(\cF')}\max\left\{\sqrt{\frac{Kp^kn^{-\beta}}{\Expect{w(\cF')}}}, \frac{Kp^kn^{-\beta}}{\Expect{w(\cF')}}\right\}
        \leq \Expect{w(\cF')}n^{-\frac{\beta}{3}}\max\left\{\sqrt{\frac{K}{w(\cF)}}, \frac{K}{w(\cF)}\right\}.
    \end{equation*}
    The result follows since $\Expect{w(\cF')} = \sum_{e\in \cF}p^kw(e) = p^kw(\cF)$.
\end{proof}

Note that a multi-hypergraph $\cF$ can be viewed as a weighted hypergraph $(\tilde\cF, w)$ where $w : E(\tilde\cF)\rightarrow \mathbb N$ is defined such that $w(e)$ is the multiplicity of $e$ in $\cF$.     
In this case, $\Delta_i(\cF) = \Delta_i(\tilde\cF, w)$ and $e(\cF) = w(\tilde\cF)$, so we obtain the following.

\begin{cor}\label{cor:KimVu} 
For all $k \geq 1$, there exists $n_0$ such that the following holds for all $p \in (0, 1]$ and $n\ge n_0$.
Let $\cF$ be a $k$-uniform multi-hypergraph on at most $n$ vertices, and let $\cF' \coloneqq \cF[V(\cF)_p]$.
If $K > 0$ satisfies
\begin{equation*}
    \Delta_i(\cF) \leq \frac{K\cdot p^i}{\log^{4k+2}n}
\end{equation*}
for all $i \in [k]$, then with probability at least $1 - n^{-\log n}$, we have
\begin{equation*}
    e(\cF') = e(\cF)\cdot p^k\left(1 \pm \frac{1}{\sqrt{\log n}}\cdot \max\left\{\sqrt{\frac{K}{e(\cF)}},~\frac{K}{e(\cF)}\right\}\right).
\end{equation*}
In particular, if $e(\cF)\ge K$, then with probability at least $1-n^{-\log n}$,
$$e(\cF')\ \ge p^k \cdot \frac{K}{2},$$
and if $e(\cF)\le K$, then with probability at least $1-n^{-\log n}$,
$$e(\cF')\le p^k \cdot 2K.$$
\end{cor}
\begin{proof}
    Apply Theorem~\ref{thm:KimVu} with $\lambda = 2\log^2 n$ where $w(e)$ is the multiplicity of the edge $e$.  For sufficiently large $n$ we have
    \begin{equation*}
        C\exp(-\lambda + (k - 1)\log n) < \exp(-\log^2 n) = n^{-\log n}.
    \end{equation*}
    We have
    \begin{equation*}
        E_1 \coloneqq \max_{i\geq 1}\Delta_i(\cF, w)p^{k - i} \leq Kp^k / \log^{4k + 2}n
    \end{equation*}
    and 
    \begin{equation*}
        E_0 \coloneqq \max_{i\geq0}\Delta_i(\cF, w)p^{k - i} \leq \max\{e(\cF)p^k, Kp^k / \log^{4k+2} n\},
    \end{equation*}
    so
    \begin{align*}
        C\lambda^k\sqrt{E_0E_1} &\leq C 2^k p^k\log^{2k}n\sqrt{\max\left\{\frac{K^2}{\log^{8k+2}n}, \frac{Ke(\cF)}{\log^{4k+2}n}\right\}} \\
        &\leq p^k\max\left\{\frac{K}{\log^{2k}n}, \sqrt{\frac{Ke(\cF)}{\log n}}\right\} \leq \Expect{e(\cF')}\cdot\frac{1}{\sqrt{\log n}}\cdot\max\left\{\frac{K}{e(\cF)}, \sqrt{\frac{K}{e(\cF)}}\right\}.
    \end{align*}
 To prove the upper bound on $e(\cF_p)$ when $e(\cF) \leq K$, we have 
 \begin{equation*}
     e(\cF)p^k\left(1 + \frac{1}{\sqrt{\log n}}\max\left\{\sqrt{\frac{K}{e(\cF)}}, \frac{K}{e(\cF)}\right\}\right) \leq p^k\left(e(\cF) + \frac{K}{\sqrt{\log n}}\right) \leq p^k K\left(1 + \frac{1}{\sqrt{\log n}}\right) \leq p^k 2K,
 \end{equation*} as desired.    
 To prove the lower bound on $e(\cF_p)$ when $e(\cF) \geq K$, we have 
 \begin{equation*}
     e(\cF)p^k\left(1 - \frac{1}{\sqrt{\log n}}\max\left\{\sqrt{\frac{K}{e(\cF)}}, \frac{K}{ e(\cF)}\right\}\right) \geq e(\cF)p^k\left(1 - \frac{1}{\sqrt{\log n}}\right) \geq p^k \frac{K}{2},
 \end{equation*} as desired.  
\end{proof}



\section{An Embedding Theorem}\label{s:Embed}

In this section, we prove a general embedding theorem (Theorem~\ref{thm:Embed}). 
The setting of the Embedding Theorem is mapping some rooted graph $H$ into a random subgraph of another graph $G$ while preserving edges and the roots. To that end, we need the following definitions.

\begin{definition}
    An \textit{embedding} $\phi : H \hookrightarrow G$ of a graph $H$ into a graph $G$ is an injective map $\phi : V(H) \rightarrow V(G)$ that preserves adjacency.  We let $\phi(H)$ denote the graph with vertex set $\phi(V(H))$ where $uv\in E(\phi(H))$ if and only if $\phi^{-1}(u)\phi^{-1}(v) \in E(H)$.
\end{definition}
\begin{definition}
Let $H$ be a graph and $R\subseteq V(H)$ be an independent set in $H$. 
An \emph{$R$-bridge of $H$} consists of a component $J$ of $H\setminus R$ together with all edges from $V(J)$ to $R$ and all those edges' ends.  We call $V(S) \cap R$ the \textit{roots} of an $R$-bridge $S$.
We let $\cB(H, R)$ denote the set of $R$-bridges of $H$.  
Given another graph $G$ and an injective map $\psi : R \rightarrow V(G)$, we say an embedding $\phi : S \hookrightarrow G$ of an $R$-bridge $S \in \cB(H, R)$ is \textit{$\psi$-respecting} if $\phi|_{V(S)\cap R} = \psi|_{V(S) \cap R}$.
\end{definition}

The maximum rooted $2$-density of the bridges will control the value of $p$ we are allowed.  For $p$ above this threshold, under some natural assumptions about $H$, we can embed every $R$-bridge $S \in \cB(H, R)$ of $H$ while preserving the mapping of the roots and ensuring that embeddings of two distinct bridges do not have a pair of non-root vertices in their overlap.

\begin{thm}\label{thm:Embed}
For every integer $h\ge 1$, there exist $C\ge 1$ and $\varepsilon > 0$ such that the following holds.  Let $H$ be a connected graph and $R\subseteq V(H)$ be an independent set in $H$ such that every $R$-bridge of $H$ has at most $h$ edges and every two distinct vertices of $R$ are in at most $h$ $R$-bridges of $H$. Let $G$ be an $n$-vertex graph with $\delta(G) \ge (1-\varepsilon)n$, and let $\psi : R \rightarrow V(G)$ be an injective map.
If $p \geq n^{-1/m_2(S,V(S)\cap R)} \cdot \log^{8h+3} n$ for every $S \in \cB(H, R)$ and $\Delta(H)\le {pn}/{C}$, then \aas there are $\psi$-respecting embeddings $\phi_S : S \hookrightarrow G_p$ for every $R$-bridge $S \in \cB(H, R)$ of $H$ such that $\binom{\phi_{S_1}(V(S_1)) \cap \phi_{S_2}(V(S_2))}{2} \subseteq \binom{\psi(V(S_1 \cap R) \cap \psi(V(S_2) \cap R)}{2}$ for every pair of distinct $R$-bridges $S_1,S_2 \in \cB(H,R)$ of $H$.

Moreover, for $m,b\in [h]$, if $e(S) = m$ and $|V(S) \setminus R| = b$ for all $S \in \cB(H, R)$ and $\cH$ is a hypergraph with $V(\cH) = R$ such that $\Delta_2(\cH) \leq h$, the roots of every $R$-bridge of $H$ are in at most $h$ edges of $\cH$, and each edge of $\cH$ contains the roots of at most $h$ $R$-bridges, then
\aas there exist such embeddings such that for every pair of distinct $u_1,u_2\in V(G)$, there is at most one edge of $\cH$ containing the roots of $R$-bridges $S_1,S_2$ of $H$ such that $u_i \in \phi_{S_i}(V(S_i))$ for $i \in \{1,2\}$ and $u_i \notin \psi(V(S_i)\cap R)$ for some $i \in \{1,2\}$.

\end{thm}

We remark that  it is conceivable that the ``moreover'' part of Theorem~\ref{thm:Embed} could be strengthened by dropping the assumption that $e(H) = m$ and $|V(S) \setminus R| = b$ for every $S \in \cB(H,R)$.  However, we do not need this stronger version, so we opted to prove the stated version as it is easier to prove.  
We could also improve the $\log^{8h+3}n$ factor in the bound on $p$ in Theorem~\ref{thm:Embed} using a stronger polynomial concentration inequality of Vu~\cite{V02} instead of Theorem~\ref{thm:KimVu}; however, we chose to use Theorem~\ref{thm:KimVu} as it is simpler to apply.

We also remark that each of the papers in this series have some form of an embedding theorem analogous to Theorem~\ref{thm:Embed}. The proof of Theorem~\ref{thm:omni-absorber} requires a ``Partial Clique Embedding Lemma'' \cite[Lemma 5.2]{DPI} to first prove the Refiner Theorem via embedding fake edges and then to prove Theorem~\ref{thm:omni-absorber} from the Refiner Theorem via embedding private absorbers, much like the way we use Theorem~\ref{thm:Embed} to prove Theorem~\ref{thm:RandomOmniAbsorber} to execute these steps in $G(n,p)$.  To obtain our results on the threshold for $(n,q,2)$-Steiner systems, we require a ``spread version'' of the Partial Clique Embedding Lemma \cite[Lemma 4.5]{DKPIV} to prove a ``spread version'' of the Omni-Absorber Theorem via (randomly) embedding ``spread boosters''.  To prove the High Girth Existence Conjecture, a High Girth Omni-Absorber Theorem~\cite[Theorem 2.22]{DPII} is needed. The proof of that theorem proceeds by randomly embedding ``girth boosters''; if it sufficed just to show that the boosters collectively have high girth, then that would follow from a version of Theorem~\ref{thm:Embed} together with an upper bound on the degree of the embedding (as in the Lov\'asz Local Lemma argument in the proof of \cite[Lemma 4.5]{DKPIV}); however, the boosters also need to have high girth with the rest of the decomposition. Thus, we instead proved \cite[Theorem 4.9]{DPII} that a sufficiently sparsified set of the possible embeddings has the desired properties with high probability. 

To prove Theorem~\ref{thm:Embed}, we will set up a bipartite hypergraph matching problem and apply Corollary~\ref{cor:EasyPerfectMatching}.
(To obtain the ``moreover'' part, we will apply Theorem~\ref{thm:BipartitePerfectMatching}.)
We define an auxiliary bipartite multi-hypergraph $\cG=(A,B)$ where $A \coloneqq \cB(H,R)$ is the set of $R$-bridges of $H$ and $B \coloneqq \binom{V(G)}{2}$.  
For every $S \in A$ and $\psi$-respecting $\phi : S \hookrightarrow G_p$, we include a copy of $\{S\} \cup \binom{\phi(V(S))}{2}\setminus \binom{\psi(V(S)\cap R)}{2}$ in $E(\cG)$.  
We use the following lemma, Lemma~\ref{lem:NumberOfEmbeddings}, to lower bound the number of $\psi$-respecting embeddings of each $R$-bridge, which provides a lower bound on the degrees of the vertices of $A$.  We use Lemma~\ref{lem:B-degree-bound} to provide an upper bound on the degrees of the vertices in $B$.  
Both proofs use Corollary~\ref{cor:KimVu}.

\begin{lem}\label{lem:NumberOfEmbeddings}
    For every fixed graph $H$ and $\eps > 0$, there exists $\eps' > 0$ such that the following holds for sufficiently large $n$.  Let $R$ be an independent set in $H$, let $G$ be an $n$-vertex graph with $\delta(G) \geq (1 -\eps')\cdot n$, and let $\psi : R \rightarrow V(G)$ be an injective map.  If $p \geq n^{-1/m(H,R)}\cdot \log^{4\cdot e(H)+3}n$, then with probability at least $1 - n^{-\log n}$, the number of embeddings $\phi : H \hookrightarrow G_p$ satisfying $\phi|_{R} = \psi$ is $(1 \pm \eps)\cdot n^{v(H) - |R|}\cdot p^{e(H)}$.
\end{lem}
\begin{proof}
Let $m \coloneqq e(H)$ and $b \coloneqq v(H) - |R|$.
Choose $\varepsilon' > 0$ small enough and $n$ large enough as needed throughout the proof.  We apply Corollary~\ref{cor:KimVu} to the multi-hypergraph $\cF$  whose vertex set is $E(G)$ where for every embedding $\phi : H \hookrightarrow G$ satisfying $\phi|_R = \psi$, we include a copy of $E(\phi(H))$ as an edge in $\cF$. 
Let $\cF' \coloneqq \cF[E(G_p)]$, and let $K \coloneqq n^b$.
Note that $\cF$ is $m$-uniform and $v(\cF) = e(G) \leq n^2$.
Note also that there is a bijection between the set of embeddings $\phi : H \hookrightarrow G_p$ satisfying $\phi|_R = \psi$ and $E(\cF')$.  Thus, it suffices to show that with probability at least $1 - n^{-\log n}$ we have $e(\cF') = n^b p^m (1 \pm \eps)$.  

First, we claim that
\begin{equation*}
    \left((1 - \eps'\cdot v(H))\cdot n\right)^b \leq e(\cF) \leq n^b.
\end{equation*}
Indeed, the lower bound follows from the fact that $\delta(G) \geq (1 - \eps')n$.  Every set of at most $v(H) - 1$ vertices of $G$ have at least $(1 - \eps'\cdot v(H))\cdot n$ common neighbors in $G$, so $e(\cF) \geq \left((1 - \eps'\cdot v(H))\cdot n\right)^b$, as desired.
For the upper bound, the number of edges of $\cF$ is at most the number of ways to choose a sequence of $b$ distinct vertices of $V(G) \setminus \psi(R)$. The number of such choices is at most $n^b$, as desired.
Hence, $\max\{\sqrt{K / e(\cF)}, K / e(\cF)\} \leq (1 - \eps'\cdot v(H))^{-b}$.

We apply Corollary~\ref{cor:KimVu} to $\cF$ with $n^2$ playing the role of $n$, so we need to show that $\Delta_i(\cF) / K \leq p^i /  \log^{4m+2}(n^2)$ for every $i \in [m]$.  
To that end, let $U \subseteq G$ have $i$ edges, and note that there are at most $b^{|V(U)\setminus \psi(R)|}n^{b - |V(U)\setminus \psi(R)|}$ edges of $\cF$ containing $E(U)$.  Moreover, by the definition of maximum rooted density, if there is at least one such edge, then
\begin{equation*}
    \frac{i}{|V(U)\setminus \psi(R)|} = \frac{e(U)}{|V(U)\setminus \psi(R)|} \le m(H,R).
\end{equation*}
Hence, for every $i \in [m]$, we have
\begin{equation*}
    \frac{\Delta_i(\cF)}{K} \leq \frac{b^{2i}n^{b - i/m(H,R)}}{n^b} = b^{2i} n^{-\frac{i}{m(H,R)}} \leq \frac{p^i}{\log^{4m+2}(n^2)},
\end{equation*}
since $p \geq n^{-1/m(H,R)} \cdot \log^{4m+3} n$.  
By Corollary~\ref{cor:KimVu}, we thus have that with probability at least $1-n^{-4\log n}$ that
\begin{equation*}
    ((1 - \eps'\cdot v(H))\cdot n)^b p^b\left(1 - \frac{(1 - \eps'\cdot v(H))^{-b}}{\sqrt{\log n}}\right) \leq e(\cF') \leq n^bp^b\left(1 + \frac{(1 - \eps'\cdot v(H))^{-b}}{\sqrt{\log n}}\right).
\end{equation*}
Therefore, with probability at least $1 - n^{-\log n}$, since $\eps'$ is sufficiently small and $n$ is sufficiently large, we have 
\begin{equation*}
    e(\cF') = n^bp^b(1 \pm \eps),
\end{equation*}
as desired.
\end{proof}

The proof of the next lemma is similar to the proof of Lemma~\ref{lem:NumberOfEmbeddings} using Corollary~\ref{cor:KimVu}; however, we need the maximum rooted 2-density here to bound the codegrees of the auxiliary hypergraph in Corollary~\ref{cor:KimVu}.

\begin{lem}\label{lem:B-degree-bound}
    For every $h, C'\ge 1$, there exist $C\ge 1$ and $\varepsilon > 0$ such that the following holds for sufficiently large $n$. Let $H$ be a connected graph and $R\subseteq V(H)$ be an independent set in $H$ such that every $R$-bridge of $H$ has at most $h$ edges and every two distinct vertices of $R$ are in at most $h$ $R$-bridges of $H$. Let $G$ be an $n$-vertex graph with $\delta(G) \ge (1-\varepsilon)n$, and let $\psi : R \rightarrow V(G)$ be an injective map.  If $p \geq n^{-1/m_2(S,V(S)\cap R)} \log^{4h+3} n$ for every $S \in \cB(H, R)$ and $\Delta(H)\le {pn}/{C}$, then \aas the following holds.
    For every pair of distinct vertices $u, v \in V(G)$ and every $m,b \in [h]$, the number of $R$-bridges $S$ of $H$ satisfying $e(S) = m$ and $|V(S)\setminus R| = b$ and $\psi$-respecting embeddings $\phi : S \hookrightarrow G_p + uv$ satisfying $\{u,v\} \in \binom{\phi(V(S))}{2}\setminus \binom{\psi(V(S)\cap R)}{2}$ is at most $n^b p^m / C'$.
\end{lem}
\begin{proof}
Choose $\varepsilon > 0$ small enough and $C$ and $n$ large enough as needed throughout the proof.  Let $m_2 \coloneqq \max\{m_2(S, V(S) \cap R) : S \in \cB(H, R)\}$.
We assume without loss of generality that every $R$-bridge $S$ of $H$ satisfies $e(S) = m$ and $|V(S) \setminus R| = b$.
We apply Corollary~\ref{cor:KimVu} to the following hypergraphs.
\begin{itemize}
    \item For distinct $u,v \in V(G)$, let $\cF_{1,u,v}$ be the multi-hypergraph whose vertex set is $E(G) - uv$ where for every $R$-bridge $S$ of $H$ and $\psi$-respecting embedding $\phi : S \hookrightarrow G + uv$ satisfying $uv \in E(\phi(S))$, we include a copy of $E(\phi(S))-uv$ as an edge in $\cF_{1,u,v}$.  
    Let $\cF'_{1,u,v} \coloneqq \cF_{1,u,v}[E(G_p)]$, and let $K_1 \coloneqq 3pb^2n^{b}/C$.
    Note that $\cF_{1,u,v}$ is $(m-1)$-uniform and $v(\cF_{1,u,v}) \leq n^2$.
    \item For distinct $u,v \in V(G)$, let $\cF_{2,u,v}$ be the multi-hypergraph whose vertex set is $E(G) - uv$ where for every $R$-bridge $S$ of $H$ and $\psi$-respecting embedding $\phi : S \hookrightarrow G - uv$ satisfying  $\{u,v\} \in \binom{\phi(V(S))}{2}\setminus \binom{\psi(V(S)\cap R)}{2}$, we include a copy of $E(\phi(S))$ as an edge in $\cF_{2,u,v}$.
    Let $\cF'_{2,u,v} \coloneqq \cF_{2,u,v}[E(G_p)]$, and let $K_2 \coloneqq 3b^2n^{b} / C$.
    Note that $\cF_{2,u,v}$ is $m$-uniform and $v(\cF_{2,u,v}) \leq n^2$.
\end{itemize}
Note that there is a bijection between $E(\cF'_{1,u,v}) \cup E(\cF'_{2,u,v})$ and the set of $(S, \phi)$ where $S$ is an $R$-bridge of $H$ and $\phi$ is a $\psi$-respecting embedding $\phi : S \hookrightarrow G_p + uv$ satisfying $\{u,v\} \in \binom{\phi(V(S))}{2}\setminus \binom{\psi(V(S)\cap R)}{2}$.
Since there are at most $n^2$ pairs of vertices $u$ and $v$, it suffices to show that for every distinct $u,v \in V(G)$, with probability at least $1 - 2n^{-\log n}$, we have that $e(\cF'_{1,u,v}) + e(\cF'_{2,u,v}) \leq n^b p^m / C'$.

First, we claim that $e(\cF_{i,u,v}) \leq K_1 \leq K_2$ for $i \in \{1,2\}$.
Indeed, the number of $R$-bridges $S$ of $H$ satisfying $\{u,v\}\cap \psi(V(S)\cap R) = \emptyset$ is at most $n\Delta(H)$, and the number of $R$-bridges $S$ of $H$ satisfying $\{u,v\}\cap \psi(V(S)\cap R)\neq \emptyset$ is at most $2\Delta(H)$.  For every $R$-bridge $S$ of $H$ satisfying $\{u,v\}\cap \psi(V(S)\cap R) = \emptyset$, there are at most $b^2n^{b-2}$ 
$\psi$-respecting embeddings $\phi : S \hookrightarrow G$ satisfying $u,v \in \phi(V(S))$.
For every $R$-bridge $S$ of $H$ satisfying $|\{u,v\}\cap \psi(V(S)\cap R)| = 1$, there are at most $bn^{b-1}$ $\psi$-respecting embeddings $\phi : S \hookrightarrow G$ satisfying $u,v \in \phi(V(S))$. 
Therefore, $e(\cF_{1,u,v}), e(\cF_{2,u,v}) \leq 3\Delta(H)b^2 n^{b-1} \leq 3pb^2n^{b} / C = K_1 \leq K_2$, as claimed.

We apply Corollary~\ref{cor:KimVu} with $n^2$ playing the role of $n$, so we need to show that $\Delta_i(\cF_{1,u,v}) / K_1 \leq p^i /  \log^{4m-2}(n^2)$ for every $i \in [m - 1]$ and $\Delta_i(\cF_{2,u,v}) / K_2 \leq p^i /  \log^{4m+2}(n^2)$ for every $i \in [m]$.  
To that end, let $U \subseteq G - uv$ such that $u,v\in V(U)$ have $i$ edges, and let $U' \coloneqq U + uv$.  Note that $v(U) \geq 3$.  We calculate the number of edges of $\cF_{1,u,v}$ and $\cF_{2,u,v}$ containing $U$ based on the intersection of $\psi(V(S) \cap R)$ and $V(U)$, where $S$ is an $R$-bridge of $H$ satisfying $\{u,v\}\setminus \psi(V(S)\cap R) \neq \emptyset$.
\begin{itemize}
    \item There are at most $n\Delta(H)$ $R$-bridges $S$ of $H$ satisfying $\psi(V(S) \cap R) \cap V(U) = \emptyset$, and for each such $S$, there are at most $b^{v(U)}n^{b - v(U)}$ $\psi$-respecting embeddings $\phi : S \hookrightarrow G+uv$ satisfying $U \subseteq \phi(S)$.
    \item There are at most $h\Delta(H)$ $R$-bridges $S$ of $H$ satisfying $|\psi(V(S) \cap R) \cap V(U)| = 1$, and for each such $S$, there are at most $b^{v(U)-1}n^{b + 1 - v(U)}$ $\psi$-respecting embeddings $\phi : S \hookrightarrow G+uv$ satisfying $U \subseteq \phi(S)$. 
    \item Since every pair of vertices of $H$ is in at most $h$ $R$-bridges, there are at most $h^3$ $R$-bridges $S$ of $H$ satisfying $|\psi(V(S) \cap R) \cap V(U)| \geq 2$, and for each such $S$, there are at most $b^{b}n^{b - |V(U)\setminus \psi(V(S)\cap R)|}$ $\psi$-respecting embeddings $\phi : S \hookrightarrow G+uv$ satisfying $U \subseteq \phi(S)$.
\end{itemize}

By the definition of maximum rooted 2-density, if there exists an $R$-bridge $S$ of $H$ and a $\psi$-respecting embedding $\phi : S \hookrightarrow G + uv$ satisfying $U \subseteq \phi(S)$, then
\begin{equation*}
    \frac{i - 1}{v(U) - 2} = \frac{e(U) - 1}{v(U) - 2} \leq m_2(S) \leq m_2
\end{equation*}
and 
\begin{equation*}
    \frac{i}{|V(U)\setminus \psi(V(S)\cap R)|} = \frac{e(U)}{|V(U)\setminus \psi(V(S)\cap R)|} \leq m(S, V(S) \cap R) \leq m_2.
\end{equation*}
Hence, for every $i \in [m]\setminus\{1\}$, since $\Delta(H) \leq pn / C$, we have
\begin{align*}
    \frac{\Delta_i(\cF_{2,u,v})}{K_2} &\leq 
    \frac{n\Delta(H)b^bn^{b - 2 - (i-1)/m_2}}{3b^2n^b/C} + 
    \frac{h\Delta(H)b^bn^{b - 1 - (i-1)/m_2}}{3b^2n^b/C} + 
    \frac{h^3b^b n^{b - i/m_2}}{3b^2n^b/C} \\
    &\leq hb^{b-2}pn^{-\frac{i-1}{m_2}} + {Ch^3b^{b-2}}n^{-\frac{i}{m_2}}
     \leq \frac{p^i}{\log^{4m+2}(n^2)},
\end{align*}
since $p \geq n^{-1/m_2} \cdot \log^{4m+3}n $, and for $i = 1$, since $v(U) \geq 3$, we have
\begin{align*}
    \frac{\Delta_1(\cF_{2,u,v})}{K_2} &\leq 
    \frac{n\Delta(H)b^bn^{b - 3}}{3b^2n^b/C} + 
    \frac{h\Delta(H)b^bn^{b - 2}}{3b^2n^b/C} + 
    \frac{h^3b^b n^{b - i/m_2}}{3b^2n^b/C} \\
    &\leq hb^{b-2}pn^{-1} + {Ch^3b^{b-2}}n^{-\frac{i}{m_2}}
     \leq \frac{p}{\log^{4m+2}(n^2)},
\end{align*}
since $p \geq n^{-1/m_2} \cdot \log^{4m+3}n$.
By Corollary~\ref{cor:KimVu}, we thus have with probability at least $1-n^{-4\log n}$ that $e(\cF'_{2,u,v})\leq 2K_2p^{m} \leq 6b^2n^b p^m / C$. 

Similarly, by the definition of maximum rooted $2$-density, if there exists an $R$-bridge $S$ of $H$ and a $\psi$-respecting embedding $\phi : S \hookrightarrow G + uv$ satisfying $U' \subseteq \phi(S)$, then we have
\begin{equation*}
    \frac{i}{v(U) - 2} = \frac{e(U')-1}{v(U') - 2} \leq m_2(S) \leq m_2
\end{equation*}
and 
\begin{equation*}
    \frac{i + 1}{|V(U)\setminus \psi(V(S)\cap R)|} = \frac{e(U')}{|V(U')\setminus \psi(V(S)\cap R)|} \leq m(S, V(S) \cap R) \leq m_2.
\end{equation*}
Hence, for every $i \in [m - 1]$, since $\Delta(H) \leq pn / C$, we have
\begin{align*}
    \frac{\Delta_i(\cF_{1,u,v})}{K_1} &\leq 
    \frac{n\Delta(H)b^bn^{b - 2 - i/m_2}}{3pb^2n^b/C} + 
    \frac{h\Delta(H)b^bn^{b - 1 - i/m_2}}{3pb^2n^b/C} + 
    \frac{h^3 b^bn^{b - (i+1)/m_2}}{3pb^2n^b/C} \\
    &\leq hb^{b-2}n^{-\frac{i}{m_2}} + \frac{C h^3b^{b-2}}{p}n^{-\frac{i+1}{m_2}}
     \leq \frac{p^i}{\log^{4m+2}(n^2)},
\end{align*}
since $p \geq n^{-1/m_2} \cdot \log^{4m+3}n$.  
By Corollary~\ref{cor:KimVu}, we thus have with probability at least $1-n^{-4\log n}$ that $e(\cF_{1,u,v}) \leq 2K_1p^{m-1} \leq 6b^2 n^b p^m / C$. 

Therefore, for every distinct $u,v\in V(G)$, with probability at least $1 - 2n^{-4\log n}$, we have that $e(\cF'_{1,u,v}) + e(\cF'_{2,u,v}) \leq 12 b^2 n^b p^m / C \leq n^b p^m / C'$, and the result follows by taking a union bound over all pairs of $u,v \in V(G)$.
\end{proof}

To prove the ``moreover'' part of Theorem~\ref{thm:Embed}, we will use Theorem~\ref{thm:BipartitePerfectMatching} with a configuration hypergraph having edges of size four.  To bound the $D$-weighted degree of this configuration hypergraph, we need the following lemma.

\begin{lem}\label{lem:configuration-degree-bound}
    For every $h, C'\ge 1$, there exist $C\ge 1$ and $\varepsilon > 0$ such that the following holds for sufficiently large $n$. Let $H$ be a connected graph and $R\subseteq V(H)$ be an independent set in $H$ such that every $R$-bridge of $H$ has at most $h$ edges and every two distinct vertices of $R$ are in at most $h$ $R$-bridges of $H$. 
    Let $\cH$ be a hypergraph with $V(\cH) = R$ such that $\Delta_2(\cH) \leq h$, the roots of every $R$-bridge of $H$ are in at most $h$ edges of $\cH$, and each edge of $\cH$ contains the roots of at most $h$ edges. 
    Let $G$ be an $n$-vertex graph with $\delta(G) \ge (1-\varepsilon)n$, and let $\psi : R \rightarrow V(G)$ be an injective map.     
    If $p \geq n^{-1/m_2(S,V(S) \cap R)} \cdot \log^{8h+3} n$ for every $S \in \cB(H, R)$ and $\Delta(H)\le {pn}/{C}$, then \aas the following holds.
    For every pair of distinct vertices $u_1, u_2 \in V(G)$ and every $m,b \in [h]$, the number of pairs of distinct $R$-bridges $S_1$ and $S_2$ of $H$ with roots together in an edge of $\cH$ satisfying $e(S_1) = e(S_2) = m$ and $|V(S_1) \setminus R| = |V(S_2) \setminus R| = b$ and $\psi$-respecting embeddings $\phi_1 : S_1 \hookrightarrow G_p$ and $\phi_2 : S_2 \hookrightarrow G_p$ satisfying 
    $E(\phi(S_1)) \cap E(\phi(S_2)) = \emptyset$, $u_i \in \phi_i(V(S_i))$ for $i \in\{1,2\}$, and $u_i \notin \psi (V(S_i)\cap R)$ for some $i \in \{1,2\}$ is at most $n^{2b}p^{2m} / C'$.
\end{lem}
\begin{proof}
Choose $\varepsilon > 0$ small enough and $C$ and $n$ large enough as needed throughout the proof.  
Let $m_2 \coloneqq \max\{m_2(S, V(S) \cap R) : S \in \cB(H, R)\}$.
We assume without loss of generality that every $R$-bridge $S$ of $H$ satisfies $e(S) = m$ and $|V(S) \setminus R| = b$.
Let $u_1, u_2 \in V(G)$ be distinct vertices.
Let $\cF$ be the multi-hypergraph whose vertex set is $E(G)$ where for every pair of distinct $R$-bridges $S_1, S_2$ of $H$ with roots together in an edge of $\cH$ and $\psi$-respecting embeddings $\phi_i : S_i \hookrightarrow G$ satisfying $E(\phi(S_1)) \cap E(\phi(S_2)) = \emptyset$, $u_i \in \phi_i(V(S_i))$ for $i \in\{1,2\}$, and $u_i \notin \psi (V(S_i)\cap R)$ for some $i \in \{1,2\}$, we include a copy of $E(\phi(S_1)) \cup E(\phi(S_2))$ as an edge in $\cF$.

We claim that $e(\cF) \leq 3h^2pb^2n^{2b} / C$.
Indeed, the number of $R$-bridges $S_i$ of $H$ satisfying $u_i \notin \psi(V(S_i)\cap R)$ is at most $n\Delta(H)$, and the number of $R$-bridges $S_i$ satisfying $u_i \in \psi(V(S_i) \cap R)$ is at most $\Delta(H)$ for $i \in \{1,2\}$.  For each $R$-bridge $S$ of $H$, there are at most $h^2$ $R$-bridges $S'$ of $H$ such that the roots of $S$ and $S'$ are together in an edge of $\cH$.  
For every pair of distinct $R$-bridges $S_1, S_2$ of $H$ satisfying $u_i \notin \psi(V(S_i)\cap R)$ for $i \in \{1,2\}$, there are at most $b^2n^{2b-2}$ choices of $\psi$-respecting embeddings $\phi_i : S_i \hookrightarrow G$ with $u_i \in \phi(V(S_i))$ for $i \in \{1,2\}$.
For $j \in \{1,2\}$ and $R$-bridges $S_1, S_2$ of $H$ satisfying $u_j \in \psi(V(S_j) \cap R)$ and $u_{3-j} \notin \psi(V(S_{3-j}) \cap R)$, there are at most $bn^{2b-1}$ choices of $\psi$-respecting embeddings $\phi_i : S_i \hookrightarrow G$ with $u_i \in \phi(V(S_i))$ for $i \in \{1,2\}$.
Therefore, $e(\cF) \leq n\Delta(H)h^2b^2n^{2b-2} + 2\Delta(H)h^2bn^{2b-1} \leq 3h^2\Delta(H)n^{2b-1} \leq 3h^2pb^2n^{2b} / C$, as claimed.

Let $\cF' \coloneqq \cF[E(G_p)]$, and note that there is a bijection between $E(\cF')$ and the set of $(S_1, S_2, \phi_1, \phi_2)$ where $S_1$ and $S_2$ are $R$-bridges of $H$ with roots together in an edge of $\cH$ and $\phi_i$ is a $\psi$-respecting embedding $\phi_i : S_i \hookrightarrow G_p$ for $i \in \{1,2\}$ such that $E(\phi(S_1)) \cap E(\phi(S_2)) = \emptyset$, $u_i \in \phi(V(S_i))$ for $i \in \{1,2\}$, and $u_i \notin \psi(V(S_i) \cap R)$ for some $i \in \{1,2\}$. 
Thus, we apply Corollary~\ref{cor:KimVu} to concentrate $e(\cF')$.  
To that end, let $K \coloneqq 3h^2b^2n^{2b}/C$.
Note that $\cF$ is $(2m)$-uniform and $v(\cF) \leq n^2$.

We apply Corollary~\ref{cor:KimVu} with $n^2$ playing the role of $n$, so we need to show that $\Delta_i(\cF) / K \leq p^i /  \log^{8m+2}(n^2)$ for every $i \in [2m]$.  
To that end, let $U_1,U_2 \subseteq G$ such that $u_i \in V(U_i)$ for $i\in\{1,2\}$ have $i_1,i_2$ edges, respectively, where $i_1 + i_2 \in [2m]$.  For $i \in \{1,2\}$, we assume that $V(U_i)$ is minimal subject to $V(U_i) \ni u_i$ and $\binom{V(U_i)}{2}\supseteq E(U_i)$.
We claim that if there exists an $R$-bridge $S_j$ of $H$ and a $\psi$-respecting embedding $\phi : S_j \hookrightarrow G$ satisfying $U_j \subseteq \phi(S_j)$ for $j \in \{1,2\}$, then
\begin{equation}\label{eqn:v(U_j)-bound-rooted-density}
    |V(U_j)\setminus \psi(V(S_j)\cap R)| \geq \frac{i_j}{m_2},
\end{equation}
and 
\begin{equation}\label{eqn:v(U_j)-bound-1-density}
    v(U_j) \geq \frac{i_j}{m_2} + 1.
\end{equation}
Indeed, \eqref{eqn:v(U_j)-bound-rooted-density} holds if $i_j = 0$, and if $i_j \geq 1$, then by the definition of maximum rooted 2-density, 
\begin{equation*}
     \frac{i_j}{|V(U_j)\setminus \psi(V(S_j)\cap R)|} \leq m(S_j, V(S_j) \cap R) \leq m_2,
\end{equation*}
as required for \eqref{eqn:v(U_j)-bound-rooted-density}.
If $v(U_j) = 1$, then \eqref{eqn:v(U_j)-bound-1-density} holds, and if $v(U_j) \geq 2$, then by the definition of maximum rooted 2-density, we also have
\begin{equation*}
     \frac{i_j}{v(U_j) - 1} \leq \frac{\max\{e(U_j), v(U_j)-1\}-1}{v(U_j) - 2} \leq m_2(S_j) \leq m_2,
\end{equation*}
as required for \eqref{eqn:v(U_j)-bound-1-density}.  
Now we calculate the number of edges of $\cF$ containing $E(U_1) \cup E(U_2)$ based on the intersection of $\psi(V(S_i) \cap R)$ and $V(U_i)$, where $S_1$ and $S_2$ are $R$-bridges of $H$ with roots together in an edge of $\cH$ satisfying $u_i \notin \psi(V(S_i) \cap R)$ for some $i \in \{1,2\}$.  

\begin{itemize}
    \item There are at most $n\Delta(H)h^2$ choices for $S_1$ and $S_2$ with roots together in an edge of $\cH$ such that $|\psi(V(S_i) \cap R) \cap V(U_i)| = 0$ for $i \in \{1,2\}$, and for each such choice, there are at most $b^{v(U_1) + v(U_2)}n^{2b - v(U_1) - v(U_2)}$ choices for $\psi$-respecting embeddings $\phi_i : S_i \hookrightarrow G$ satisfying $U_i \subseteq \phi(S_i)$.  By \eqref{eqn:v(U_j)-bound-1-density}, we have at most 
    $h^2b^{2b}\Delta(H)n^{2b - (i_1 + i_2)/m_2 - 1}$ such choices in total.
    \item For $j \in \{1,2\}$, there are at most $h\Delta(H)h^2$ choices for $S_1$ and $S_2$ with roots together in an edge of $\cH$ such that $|\psi(V(S_j) \cap R) \cap V(U_j)| = 1$ and $|\psi(V(S_{3-j}) \cap R) \cap V(U_{3-j})| = 0$, and for each such choice, there are at most $b^{v(U_1) + v(U_2) - 1}n^{2b + 1 - v(U_1) - v(U_2)}$ choices for $\psi$-respecting embeddings $\phi_i : S_i \hookrightarrow G$ satisfying $U_i \subseteq \phi(S_i)$. By \eqref{eqn:v(U_j)-bound-1-density}, we have at most 
    $h^3b^{2b}\Delta(H)n^{2b - (i_1 + i_2)/m_2 - 1}$ such choices in total.
    \item Since $\Delta_2(\cH) \leq h$ and each edge of $\cH$ contains the roots of at most $h$ edges, there are at most $h^5$ choices for $S_1$ and $S_2$ with roots together in an edge of $\cH$ 
    such that $|\psi(V(S_i) \cap R) \cap V(U_i)| = 1$ for $i \in \{1,2\}$, and for each such choice, there are at most $b^{v(U_1) + v(U_2) - 2}n^{2b + 2 - v(U_1) - v(U_2)}$ choices for $\psi$-respecting embeddings $\phi_i : S_i \hookrightarrow G$ satisfying $U_i \subseteq \phi(S_i)$. 
    By \eqref{eqn:v(U_j)-bound-1-density}, we have at most
    $h^5b^{2b}n^{2b - (i_1 + i_2)/m_2}$ such choices in total.
    \item For $j \in \{1,2\}$, there are at most $h^5$ choices for $S_1$ and $S_2$ with roots together in an edge of $\cH$  such that $|\psi(V(S_j) \cap R) \cap V(U_j)| \geq 2$ and $|\psi(V(S_{3-j}) \cap R) \cap V(U_{3-j})| \leq 1$, and for each such choice, there are at most $b^{2b}n^{2b + 1 - |V(U_j)\setminus \psi(V(S_j)\cap R)| - v(U_{3-j})}$ $\psi$-respecting embeddings $\phi_i : S_i \hookrightarrow G$ satisfying $U_i \subseteq \phi(S_i)$.
    By \eqref{eqn:v(U_j)-bound-rooted-density} and \eqref{eqn:v(U_j)-bound-1-density}, we have at most 
    $h^5 b^{2b}n^{2b - (i_1 + i_2)/m_2}$ such choices in total.
    \item There are at most $h^5$ choices for $S_1$ and $S_2$ with roots together in an edge of $\cH$ such that $|\psi(v(S_i) \cap R) \cap V(U_i)| \geq 2$ for $i \in \{1,2\}$, and for each such choice, there are at most $b^{2b}n^{2b - |V(U_1)\setminus \psi(V(S_1)\cap R)| - |V(U_1)\setminus \psi(V(S_1)\cap R)|}$ $\psi$-respecting embeddings $\phi_i : S_i \hookrightarrow G$ satisfying $U_i \subseteq \phi(S_i)$.
    By \eqref{eqn:v(U_j)-bound-rooted-density}, we have at most $h^5b^{2b}n^{2b - (i_1 + i_2)/m_2}$ such choices in total.
\end{itemize}
Hence, for every $i \in [2m]$, since $\Delta(H) \leq pn / C$, we have 
\begin{align*} 
    \Delta_i(\cF) &\leq \sum_{j = 0}^i\binom{i}{j}\left(
    3h^3b^{2b}\Delta(H)n^{2b - 1 - \frac{j + i - j}{m_2} } + 
    4h^5b^{2b}n^{2b - \frac{j + i - j}{m_2}}\right)\\
    & \leq 2^m h^5 b^{2b}\left(3\left(\frac{p}{C}\right)n^{2b-\frac{i}{m_2}} + 4 n^{2b-\frac{i}{m_2}}\right)
     \leq \frac{p^i}{\log^{8m+2}(n^2)}K,
\end{align*}
since $p \geq n^{-1/m_2} \cdot \log^{8m+3}n$.
By Corollary~\ref{cor:KimVu}, we thus have with probability at least $1-n^{-4\log n}$ that $e(\cF) \leq 2Kp^{2m} \leq 6h^2n^{2b} p^{2m} / C$, and the result follows by taking a union bound over all pairs of $u_1,u_2 \in V(G)$.
\end{proof}

Now we can prove Theorem~\ref{thm:Embed}.
\begin{proof}[Proof of Theorem~\ref{thm:Embed}]
Let $C'_1 \coloneqq 600h^4$, and let $C'_2 \coloneqq 3600\cdot 600^2\cdot h^{10}$.
Choose $\eps$ small enough to play the role of  $\eps'$ in Lemma~\ref{lem:NumberOfEmbeddings} with $1/2$ playing the role of $\eps$ and any connected graph with at most $h$ edges playing the role of $H$.
Choose $\eps$ small enough and $C$ large enough to apply Lemma~\ref{lem:B-degree-bound} with $C'_1$ playing the role of $C'$ and to apply Lemma~\ref{lem:configuration-degree-bound} with $C'_2$ playing the role of $C'$.

We define an auxiliary bipartite multi-hypergraph $\cG=(A,B)$ where $A \coloneqq \cB(H,R)$ is the set of $R$-bridges of $H$ and $B \coloneqq \binom{V(G)}{2}$.  
For every $S \in A$ and $\psi$-respecting $\phi : S \hookrightarrow G_p$, we include a copy of $\{S\} \cup \binom{\phi(V(S))}{2}\setminus \binom{\psi(V(S)\cap R)}{2}$ in $E(\cG)$.  
To find the desired embeddings $\phi_S : S \hookrightarrow G_p$ for every $S \in \cB(H, R)$, it suffices to show that \aas $\cG$ has an $A$-perfect matching.
For each $m, b\in [h]$, let $A_{m,b} := \{S\in A : e(S)=m, |V(S)\setminus R|=b\}$, let $\cG_{m,b} := \cG[A_{m,b}\cup B]$, and let $D_{m,b} \coloneqq n^b p^m / (16h^4)$.
We will apply Corollary~\ref{cor:EasyPerfectMatching} to $\cG$ with $\{A_{m,b} : m , b\in [h]\}$ as the partition of $A$ and $h^2$ as $r$ to show that \aas there is an $A$-perfect matching of $\cG$.  
To prove the ``moreover'' part, we will apply Theorem~\ref{thm:BipartitePerfectMatching} with $D \coloneqq n^b p^m / (600 h^2)$.

First, we claim that \aas the following hold.
\begin{enumerate}[label=(\alph*)]
    \item\label{embedding-theorem:A-degree-lower} every $S \in A$ satisfies $d_{\cG}(S) \geq n^{|V(S)\setminus R|}p^{e(S)} / 2$,
    \item\label{embedding-theorem:A-degree-upper} every $S \in A$ satisfies $d_{\cG}(S) \leq 2n^{|V(S)\setminus R|}p^{e(S)}$,
    \item\label{embedding-theorem:B-degree} every $\{u,v\}\in B$ satisfies $d_{\cG_{m,b}}(\{u,v\}) \leq n^bp^m / C'_1$, and
    \item\label{embedding-theorem:configuration-degree} for every distinct $u_1, u_2 \in V(G)$, the number of distinct $S_1,S_2 \in \cB(H,R)$ with roots together in an edge of $\cH$ and $\psi$-respecting embeddings $\phi_i : S_i \hookrightarrow G_p$ satisfying $u_i \in \phi_i(V(S_i))$ for $i \in \{1,2\}$, $E(\phi(S_1)) \cap E(\phi(S_2)) = \emptyset$, and $u_i \notin \psi(V(S_i) \cap R)$ for some $i \in \{1,2\}$ is at most $n^{2b}p^{2m} / C'_2$.
\end{enumerate}

Indeed, for every $R$-bridge $S \in \cB(H, R)$ of $H$, by Lemma~\ref{lem:NumberOfEmbeddings} applied with $S$ playing the role of $H$ and $1/2$ playing the role of $\eps$, with probability at least $1 - n^{\log n}$, we have $d_{\cG_{m,b}}(S) = (1\pm1/2)n^{|V(S)\setminus R|}p^{e(S)} / 2$.  Hence, by a union bound over the at most $n\Delta(H)$ $R$-bridges of $H$, \aas \ref{embedding-theorem:A-degree-lower} and \ref{embedding-theorem:A-degree-upper} hold. 
We have that \ref{embedding-theorem:B-degree} holds \aas by Lemma~\ref{lem:B-degree-bound} and \ref{embedding-theorem:configuration-degree} holds \aas by Lemma~\ref{lem:configuration-degree-bound}.

Assuming \ref{embedding-theorem:A-degree-lower}--\ref{embedding-theorem:configuration-degree}, we find the desired embeddings $\phi_S : S \hookrightarrow G_p$ for every $S \in \cB(H, R)$.  By \ref{embedding-theorem:A-degree-lower}, every $S \in A_{m,b}$ satisfies $d_{\cG_{m,b}}(S) \geq n^b p^m / 2 \geq 8h^4 D_{m,b}$, and by \ref{embedding-theorem:B-degree}, every $\{u,v\}\in B$ satisfies $d_{\cG_{m,b}}(\{u,v\}) \leq n^bp^m / C'_1 \leq D_{m,b}$.  
Therefore, by Corollary~\ref{cor:EasyPerfectMatching}, $\cG$ has an $A$-perfect matching, as desired.

Now suppose $e(S) = m$ and $|V(S)\setminus R| = b$ for all $S \in \cB(H,R)$ and
$\cH$ is a hypergraph with $V(\cH) = R$ such that $\Delta_2(\cH) \leq h$, the roots of every $R$-bridge of $H$ are in at most $h$ edges of $\cH$, and each edge of $\cH$ contains the roots of at most $h$ edges. 
By \ref{embedding-theorem:A-degree-lower}, every $S \in A$ satisfies $d_{\cG}(S) \geq n^b p^m / 2 \geq 300 h^2 D$, and by \ref{embedding-theorem:B-degree}, every $\{u,v\}\in B$ satisfies $d_{\cG}(\{u,v\}) \leq n^bp^m / C'_1 \leq D$. 
To prove the ``moreover'' part, we let $\cC$ be the configuration hypergraph for $\cG$ defined as follows.  For distinct vertices $u_1, u_2\in V(G)$, distinct $R$-bridges $S_1,S_2 \in \cB(H,R)$ with roots together in an edge of $\cH$ and $\psi$-respecting embeddings $\phi_i : S_i \hookrightarrow G_p$ satisfying $u_i \in \phi_i(V(S_i))$ for $i \in \{1,2\}$, $E(\phi_1(S_1)) \cap E(\phi_2(S_2)) = \emptyset$, and $u_i \notin \psi(V(S_i) \cap R)$ for some $i \in \{1,2\}$, and distinct $R$-bridges $S'_1,S'_2 \in \cB(H,R)\setminus\{S_1,S_2\}$ with roots together in an edge of $\cH$ and $\psi$-respecting embeddings $\phi'_i : S'_i \hookrightarrow G_p$ satisfying $E(\phi'_1(S'_1)) \cap E(\phi'_2(S'_2)) = \emptyset$, $u_i \in \phi_i(V(S'_i))$ for $i \in \{1,2\}$, and $u_i \notin \psi(V(S'_i) \cap R)$ for some $i \in \{1,2\}$, we include the size-4 edge consisting of the four edges corresponding to $(S_1, \phi_1)$, $(S_2, \phi_2)$, $(S'_1, \phi'_1)$, and $(S'_2, \phi'_2)$ in $\cC$.

Now we bound the maximum $D$-weighted degree of $\cC$, as follows.
Given $(S_1, \phi_1)$, there are at most $h$ edges of $\cH$ containing the roots of $S_1$.  For each such edge of $\cH$, there are most $h$ choices of an $S_2 \in \cB(H, R)$ with roots in that edge together with $S_1$.  For each such choice of $S_2$, by \ref{embedding-theorem:A-degree-upper}, there are at most $2n^bp^m = 1200h^2 D$ choices for the embedding $\phi_2$. 
For each such choice of $\phi_2$, there are at most $h^2$ choices for $u_1, u_2$.  For each such choice of $u_1,u_2$, by \ref{embedding-theorem:configuration-degree} there are at most $n^{2b}p^{2m} / C'_2$ choices for $(S'_1, \phi'_1)$ and $(S'_2, \phi'_2)$.  
Therefore, we have 
\begin{equation*}
    w_D(\cC) \leq \frac{3}{D^3}\cdot h\cdot h\cdot 1200h^2 D \cdot h^2 \cdot \frac{n^{2b}p^{2m}}{C'_2} = \frac{3600\cdot h^6\cdot (600h^2)^2}{C'_2}\leq 1,
\end{equation*}
as required.  Therefore, by Theorem~\ref{thm:BipartitePerfectMatching}, there exists a $\cC$-avoiding $A$-perfect matching of $\cG$.  By the choice of $\cG$, for every $A$-perfect matching of $\cG$, the corresponding embeddings $\phi_S : S\hookrightarrow G_p$ for $S \in \cB(H,R)$ satisfy $E(\phi_S(S)) \cap E(\phi_{S'}(S')) = \emptyset$ for distinct $S,S' \in \cB(H, R)$, since $R$ is independent in $H$.  Thus, by the choice of $\cC$, a $\cC$-avoiding $A$-perfect matching of $\cG$ yields the desired set of embeddings.
\end{proof}

\section{A Random Omni-Absorber Theorem}\label{s:RandomOmniAbsorber}

In this section, we prove Theorem~\ref{thm:RandomOmniAbsorber}. First we introduce anti edges and fake edges and determine their maximum rooted $2$-density. 

Before that, we need the following lemma that shows the maximum rooted density and maximum rooted $2$-density concatenate as follows. The proof proceeds via simple case analysis and computation, but we include its proof for completeness.

\begin{lem}\label{lem:2DensityConcatenate}
Let $H$ be a graph and $R\subseteq V(H)$ be an independent set of $H$. If $H_1$ is a proper induced subgraph of $H$ containing $R$ and we let $H_2:= H\setminus E(H_1)$, then 
$$m(H,R) \le \max~\big\{~m(H_1,R),~m(H_2,V(H_1))~\big\}$$
and
$$m_2(H,R) \le \max~\big\{~m_2(H_1,R),~m_2(H_2,V(H_1))~\big\}.$$
\end{lem}
\begin{proof}
First we prove the first inequality. To that end, let $H_0$ be a subgraph of $H$ such that $V(H_0)\setminus R\ne \emptyset$. Let $H_1' := H_0\cap H_1$ and let $H_2':= H_0 \cap H_2$. Let $e_1 :=e(H_1')$, $e_2:=e(H_2')$, $v_1 := |V(H_1')\setminus R|$, and $v_2 := |V(H_2')\setminus V(H_1)|$. Then $e(H_0)=e_1+e_2$ and $|V(H_0)\setminus R| = v_1+v_2$. 
Thus
$$\frac{e(H_0)}{|V(H_0)\setminus R|} = \frac{e_1+e_2}{v_1+v_2}.$$
If $e_1=0$, then we have 
$$\frac{e(H_0)}{|V(H_0)\setminus R|} \le \frac{e_2}{v_2} \le m(H_2,V(H_1))$$
as desired. Similarly if $e_2=0$, then 
$$\frac{e(H_0)}{|V(H_0)\setminus R|} \le \frac{e_1}{v_1} \le m(H_1,R)$$
as desired. So we assume that $e_1,e_2\ne 0$. It follows that $v_1,v_2\ne 0$. Thus
{\small
\begin{eqnarray*}
\frac{e(H_0)}{|V(H_0)\setminus R|} &=& \frac{e_1+e_2}{v_1+v_2} = \frac{e_1}{v_1}\cdot \frac{v_1}{v_1+v_2} + \frac{e_2}{v_2}\cdot\frac{v_2}{v_1+v_2}\\ &\le& m(H_1,R) \cdot \frac{v_1}{v_1+v_2} + m(H_2,V(H_1))\cdot \frac{v_2}{v_1+v_2}\\ &\le& \max \{m(H_1,R),~m(H_2,V(H_1))\}
\end{eqnarray*}}
as desired. This proves the first inequality.

Finally we prove the second inequality. Given the first inequality, it suffices to prove that $m_2(H)\le \max~\big\{~m_2(H_1,R),~m_2(H_2,V(H_1))~\big\}$. To that end, let $H_0$ be a subgraph of $H$ such that $v(H_0)\ge 3$. Again let $H_1' := H_0\cap H_1$ and let $H_2':= H_0\cap H_2$. Let $e_1 :=e(H_1')$, $e_2:=e(H_2')$, $v_1 := |V(H_1')\setminus R|$, and $v_2 := |V(H_2')\setminus V(H_1)|$. Then $e(H_0)=e_1+e_2$ and $|V(H_0)\setminus R| = v_1+v_2$.
Thus
$$\frac{e(H_0)-1}{v(H_0)-2} = \frac{e_1+e_2-1}{v(H_1')+v_2-2}.$$

First suppose $v(H_1')\ge 3$. If $e_2=0$, then 
$$\frac{e(H_0)-1}{v(H_0)-2} \le \frac{e_1-1}{v(H_1')-2} \le m_2(H_1,R)$$
as desired. So we assume $e_2\ne 0$. It follows that $v_2\ne 0$. 
Thus
{\small
\begin{eqnarray*}
\frac{e(H_0)-1}{v(H_0)-2} &=& \frac{e_1+e_2-1}{v(H_1')+v_2-2} = \frac{e_1-1}{v(H_1')-2}\cdot \frac{v(H_1')-2}{v(H_1')+v_2-2} + \frac{e_2}{v_2}\cdot\frac{v_2}{v(H_1')+v_2-2}\\ &\le& m_2(H_1) \cdot \frac{v(H_1') - 2}{v(H_1')+v_2-2} + m(H_2,V(H_1))\cdot \frac{v_2}{v(H_1')+v_2-2}\\ &\le& \max \{m_2(H_1,R),~m_2(H_2,V(H_1))\}
\end{eqnarray*}}
as desired.

So we assume that $v(H_1')\le 2$. It follows that $v_2 \ne 0$. First suppose $v(H_1')\le 1$. Then $e_1=0$ and we find that
$$\frac{e(H_0)-1}{v(H_0)-2} \le \frac{e_2-1}{v(H_2')-2} \le m_2(H_2,V(H_1))$$
as desired. So we assume that $v(H_1')=2$. Hence $e_1\le 1$. Thus
$$\frac{e(H_0)-1}{v(H_0)-2} \le \frac{e_2}{v_2} \le m(H_2,V(H_1))$$
as desired.
\end{proof}

\subsection{Fake Edge and Private Absorbers}

Given our embedding theorem, we turn to proving our random omni-absorber theorem via embedding fake edges and absorbers. To that end, we recall the following gadgets from~\cite{DPI} restricted to the graph case.

\begin{definition}\label{def:BasicGadget}
Let $q\ge 3$ be an integer. Let $S$ be a set of vertices of size $2$.
\begin{itemize}
    \item An \emph{anti-edge on $S$}, denoted ${\rm AntiEdge}_q(S)$, is a set of new vertices $x_1,\ldots, x_{q-2}$ together with edges $\binom{S\cup \{x_i:~i\in[q-2]\} }{2} \setminus \{S\}$.
    \item A \emph{fake edge on $S$}, denoted ${\rm FakeEdge}_q(S)$, is a set of new vertices $x_1,\ldots, x_{q-2}$ together with  $\{ {\rm AntiEdge}_q(T): T\in \binom{S\cup \{x_i:~i\in[q-2]\} }{2} \setminus \{S\} \}$.
\end{itemize}    

\end{definition}

Next we collect basic facts about the divisibility of the degrees in these basic gadgets as follows.

\begin{fact}\label{fact:Div}
Let $q\ge 3$ be an integer. The following hold:
\begin{itemize}
    \item[(1)] If $F={\rm AntiEdge}_q(S)$, then $e(F) \equiv -1 \mod \binom{q}{2}$, $d_F(v)\equiv -1 \mod (q-1)$ for every $v\in S$ and $d_F(v)\equiv 0 \mod (q-1)$ for every $v\in V(F)\setminus S$.
    \item[(2)] If $F={\rm FakeEdge}_q(S)$, then $e(F) \equiv +1 \mod \binom{q}{2}$, $d_F(v)\equiv +1 \mod (q-1)$ for every $v\in S$ and $d_F(v)\equiv 0 \mod (q-1)$ for every $v\in V(F)\setminus S$.
\end{itemize}
\end{fact}

Thus Fact~\ref{fact:Div}(1) says that an anti-edge $F$ on $S$ has the negative divisibility properties of an actual edge on $S$. Similarly Fact~\ref{fact:Div}(2) says that a fake edge $F$ on $S$ has the same divisibility properties as an actual edge on $S$.  Furthermore, we note all of the above basic gadgets defined in Definition~\ref{def:BasicGadget} have at most $(q-2)\cdot \binom{q}{2} \le q^{3}$ vertices and hence at most $q^{6}$ edges.

Next we compute the density of fake edges and anti-edges.
Here then is our proposition computing the maximum rooted $2$-density of anti-edges and fake edges.

\begin{proposition}\label{prop:FakeEdge}
If $q\ge 3$ is an integer, then 
$$m_2({\rm Fake Edge}_q(S),S) = m_2({\rm AntiEdge}_q(S),S) = m({\rm Fake Edge}_q(S),S) = m({\rm AntiEdge}_q(S),S) = \frac{q+1}{2}.$$
\end{proposition}
\begin{proof}
By Lemma~\ref{lem:2DensityConcatenate} and the definition of ${\rm FakeEdge}_q(S)$, it follows that it suffices to prove that $$m_2({\rm AntiEdge}_q(S),S) = m({\rm AntiEdge}_q(S),S) = \frac{q+1}{2}.$$ Let $H'\subseteq {\rm AntiEdge}_q(S)$ such that $H'\setminus S\ne \emptyset$, and let $m := |V(H')\setminus S|$. Note $1\le m\le q-2$. Then
$$\frac{e(H')}{m} \le \frac{\binom{m+2}{2}-1}{m} = \frac{m+3}{2} \le \frac{q+1}{2}.$$
Hence $m({\rm AntiEdge}_q(S),S) = \frac{q+1}{2}$ (since equality is attained when $m=q-2$).

Similarly let $H''\subseteq {\rm AntiEdge}_q(S)$ such that $v(H'')\ge 3$. First suppose $v(H'')<q$. Then
$$\frac{e(H'')-1}{v(H'')-2} \le \frac{\binom{v(H'')}{2}-1}{v(H'')-2} = \frac{v(H'')+1}{2} < \frac{q+1}{2}.$$
Next suppose $v(H'')=q$. In that case, we have that
$$\frac{e(H'')-1}{v(H'')-2} \le \frac{\binom{q}{2}-2}{q-2} < \frac{q+1}{2}.$$
Hence $m_2({\rm AntiEdge}_q(S)) \le \frac{q+1}{2}$. Combined with the first equality proved above, this yields that \\
$m_2({\rm AntiEdge}_q(S),S) = m({\rm AntiEdge}_q(S),S)$ as desired.
\end{proof}

\subsection{Proof of Random Omni-Absorber Theorem}

Now we prove Theorem~\ref{thm:RandomOmniAbsorber}.

\begin{proof}[Proof of Theorem~\ref{thm:RandomOmniAbsorber}]
By the definition of $d_{{\rm abs}}(K_q)$, there exists $C_1$ such that for all $K_q$-divisible graphs $L$ on at most $q^5$ vertices, there exists a $K_q$-absorber $A_L$ of $L$ on at most $C_1$ vertices such that $m_2(A_L,V(L))\le d_{{\rm abs}}(K_q)+\frac{\beta}{10}$. 
Let $C'$ be as in Theorem~\ref{thm:omni-absorber} and $C''$ and $\eps$ as in Theorem~\ref{thm:Embed} with $\max\{q^6, C' + 1, C_1^2\}$ playing the role of $h$. Let $C \coloneqq \max\{\eps^{-1}8q^2 C' (C'')^2, 4C_1C'(C'+1)C'', 2\eps^{-1}\}$. By Theorem~\ref{thm:omni-absorber}, there exists a $C'$-refined $K_q$-omni-absorber $A_0\subseteq K_n$ for $X$ with decomposition family $\mathcal{H}_0$ and decomposition function $\mathcal{Q}_{A_0}$ such that $\Delta(A_0)\le C'\cdot \Delta$ where $\Delta := \max\{\Delta(X), \sqrt{n}\cdot \log n\}$.  Note that since $\Delta(X) \leq pn / C$ and $p \geq n^{-\frac{2}{q + 1} + \beta} \geq Cn^{-1/2}\log n$, we have $\Delta \leq pn / C$.

Let $p_1 \coloneqq  \eps p / (8C'')$, let $p_2 \coloneqq p/4$, and let $G_1 \sim G(n,p_1)$ and $G_2\sim G(n,p_2)$ where $V(G_1) = V(G_2) = V(X)$.  Let $G'_1 \coloneqq G_1\setminus E(X)$ and $G'_2 \coloneqq G_2 \setminus (E(G_1) \cup X)$.  We can couple $G_1 \cup G_2$ with $G(n,p_1 + p_2 - p_1p_2)$ and $p_1 + p_2 - p_1p_2 \leq p / 2$, so it suffices to show that \aas $G_1 \cup G_2 \setminus X$ contains a $C$-refined $K_q$-omni-absorber $A$ for $X$. 
We will also assume $\Delta(G_1) \leq 2p_1n$ as this event also holds asymptotically almost surely by the Chernoff bounds.

For each $e \in E(A_0)$, let $H_e \cong {\rm FakeEdge}_q(e)$ so that $H_e$ and $H_{e'}$ are edge-disjoint for every distinct $e,e' \in E(A_0)$ and $V(H_e)$ intersects $V(X)$ only in the vertices of $e$.  Let $H \coloneqq \bigcup_{e\in E(A_0)}H_e$, and let $\psi : V(X) \rightarrow V(G_1)$ be the identity. 
We will embed $H$ in $G'_1$ using Theorem~\ref{thm:Embed} with  $\max\{q^{6}, C'\}$, $H$, $V(X)$, $K_n \setminus E(X)$, $\psi$, and $p_1$ playing the roles of $h$, $H$, $R$, $G$, $\psi$, and $p$, respectively.

Note that every $V(X)$-bridge of $H$ is a ${\rm FakeEdge}_q(e)$ for some $e$, so every $V(X)$-bridge of $H$ has the same number, say $h$, of edges.  Recall that $v({\rm FakeEdge}_q(e))\le q^3$, so $h \leq q^6$.  By Proposition~\ref{prop:FakeEdge}, we have that $m_2(S, V(X)) = \frac{q+1}{2}$ for every $S \in \cB(H, V(X))$.  
Since $n$ is large enough, we have that $p_1 \geq n^{-\frac{2}{q+1}} \cdot 
\log^{8q^{6}+3} n$.  
By the choice of $p_1$ and $C$, since $\Delta(A_0) \leq C'\Delta$ and $\Delta \leq pn / C$, we have $\Delta(H) \leq q^2 \Delta(A_0) \leq q^2 C'\Delta \leq p_1n/C''$.
Also, $\delta(K_n \setminus X) = n - 1 - \Delta(X) \geq (1 - \eps)n$ since $\Delta(X) \leq pn / C \leq \eps n / 2$. Finally, note that every two vertices of $X$ are in at most one $V(X)$-bridge of $H$.
Let $\cH$ be the hypergraph with $V(\cH) = V(X)$ and edge set $\{V(J) : J \in \cH_0\}$.  Since $A_0$ is $C'$-refined, we have $\Delta_2(\cH) \leq C'$, and by the construction of $H$, the roots of every $V(X)$-bridge of $H$ are in at most $C'$ edges of $\cH$.  Moreover, each edge of $\cH$ contains the roots of $\binom{q}{2} \leq q^6$ bridges.  
Hence, 
by Theorem~\ref{thm:Embed}, there exist $\psi$-respecting embeddings $\phi_e : H_e \hookrightarrow G'_1$ such that $\binom{\phi_e(V(H_e)) \cap \phi_{e'}(V(H_{e'}))}{2} = \emptyset$ (since $\binom{e \cap e'}{2} = \emptyset$) for every pair of distinct $e, e' \in E(A_0)$ and moreover, for every pair of distinct $u_1, u_2 \in V(G)$, there is at most one $J \in \cH_0$ containing edges $e_1$ and $e_2$ such that $u_i \in \phi_{e_i}(V(H_{e_i}))$ for $i \in \{1,2\}$ and $u_i \notin e_i$ for some $i \in \{1,2\}$.

We let $X'\coloneqq X\cup \bigcup_{e\in E(A_0)}\phi_e(H_e)$, and let
$$\mathcal{H}_1 \coloneqq \Bigg\{ (J\cap X) \cup \bigcup_{e\in J\setminus X} \phi_e(H_e) : J\in \mathcal{H}_0\Bigg\}.$$
Note that every element of $\mathcal{H}_1$ is a graph with at most $q^5$ vertices which is $K_q$-divisible by Fact~\ref{fact:Div}. 
Hence, by the choice of $C_1$, for every $J\in \mathcal{H}_1$, there exists a $K_q$-absorber $A_J$ of $J$ such that $v(A_J)\le C_1$ and $m_2(A_J,V(J)) \le d_{{\rm abs}}(K_q)+\frac{\beta}{10}$.  We assume that $A_J$ and $A_{J'}$ are edge-disjoint for distinct $J, J' \in \cH_1$ and $V(A_J)$ intersects $V(X')$ only in $V(J)$.  Let $H' \coloneqq \bigcup_{J\in\cH_1}A_J$.
We will embed $H'$ in $G'_2$ using Theorem~\ref{thm:Embed} with $\max\{C_1^2, C'+1\}$, $H'$, $V(X')$, $K_n \setminus (X \cup E(G_1))$, $\psi$, and $p_2$ playing the roles of $h$, $H$, $R$, $G$, $\psi$, and $p$, respectively.

Note that every $V(X')$-bridge $S$ of $H'$ has at most $C_1$ vertices and thus at most $C_1^2$ edges and by construction satisfies $m_2(S, V(S) \cap V(X')) \leq d_{{\rm abs}}(K_q) + \frac{\beta}{10}$.
Since $\beta > 0$ and $d_{{\rm abs}}(K_q)\ge 1$, we have $p_2 \geq n^{-\frac{1}{d_{{\rm abs}}(K_q)} + \beta}/4 \geq n^{-\frac{1}{d_{{\rm abs}}(K_q) + \beta/10} + \frac{\beta}{2}} \geq n^{-\frac{1}{m_2(S, V(S) \cap V(X'))}}\log^{8\max\{C_1^2, C'+1\}+3}n$ for every $S \in \cB(H', V(X'))$.
Since $A_0$ is $C'$-refined, we have\COMMENT{
    every vertex is in $\Delta(A_0 \cup X)$ edges of $A_0 \cup X$, and each of these are in at most $C'$ of the $J \in \cH_0$, and we have degree at most degree $C_1$ in the absorber for $J$}
    $\Delta(H') \leq C_1 C' \Delta(A_0 \cup X) \leq C_1C'(C' + 1)\Delta \leq p_2 n / C''$.
Also, $\delta(K_n \setminus (X \cup E(G_1))) \geq n - 1 - \Delta(X) - \Delta(G_1) \geq (1 - \eps)n$ since we are assuming $\Delta(G_1) \leq 2p_1 n$.
Finally, note that every two distinct vertices of $X'$ are in at most $C' + 1$ $V(X')$-bridges of $H'$\COMMENT{
    If the two vertices $u_1, u_2$ are together in some $\phi_e(H_e)$ for $e \in E(A_0)$ (either because $e = u_1u_2$ or otherwise), then they are not in any other $\phi_{e'}(H_{e'})$ by the choice of the embeddings, and $e$ is in at most $C'$ of the $J \in \cH_0$ since $A_0$ is $C'$-refined.  
    There is at most one other bridge of $H'$ containing $u_1$ and $u_2$ since there is at most one $J \in \cH_0$ containing edges $e_1$ and $e_2$ such that $u_i \in \phi_{e_i}(V(H_{e_i}))$ for $i \in \{1,2\}$ and $u_i \notin e_i$ for some $i \in \{1,2\}$, and if there is a $J \in \cH_0$ containing edges $e_1$ and $e_2$ such that $u_i \in \phi_{e_i}(V(H_{e_i}))$ and $u_i \in e_i$ for $i \in \{1,2\}$, then $u_1u_2 \in E(A_0)$. 
}.
Hence, 
by Theorem~\ref{thm:Embed}, there exist $\psi$-respecting embeddings $\phi_J : A_J \hookrightarrow G'_2$ for all $J \in \cH_1$ such that $\phi_J(A_J)$ and $\phi_{J'}(A_{J'})$ are edge-disjoint for every pair of distinct $J, J' \in \cH_1$.  

We let $A \coloneqq \bigcup_{e\in E(A_0)}\phi_e(H_e) \cup \bigcup_{J \in \cH_1}\phi_J(A_J)$.
We define a decomposition family $\cH_A$ for $A$ where for each $J \in \cH_0$, we include both the $K_q$-decomposition of $\phi_{J'}(A_{J'})$ and of $J' \cup \phi_{J'}(A_{J'})$ in $\cH_A$, where $J' \coloneqq (J \cap X) \cup \bigcup_{e \in J \setminus X} \phi_e(H_e)$ (these decompositions exist because $A_{J'}$ is an absorber).  We define a decomposition function $\cQ_A$ for $A$ as follows.  Given a $K_q$-divisible $L \subseteq X$, for every $J\in \cH_0$, we let $\cQ_{A}(L)$ contain the $K_q$-decomposition of $J' \cup \phi_{J'}(A_{J'})$ if $J \in \cQ_{A_0}(L)$ and of $\phi_{J'}(A_{J'})$ if $J \in \cH_0\setminus \cQ_{A_0}(L)$, where $J' = (J \cap X) \cup \bigcup_{e \in J \setminus L}\phi_e(H_e)$.  
By construction, $A$ is a $C$-refined $K_q$-omni-absorber for $X$ with decomposition family $\cH_A$ and decomposition function $\cQ_A$, and $A \subseteq G'_1 \cup G'_2$, as desired.
\end{proof}

\section{The Divisibility Fixer}\label{s:DivFixer}

In this section, we introduce the notion of a $K_q$-divisibility fixer which we will use to remove a `leave' such that the rest of the graph satisfies the necessary divisibility conditions for having a $K_q$-decomposition before finding such a decomposition. Crucially, we prove that for our range of $p$, $G(n,p)$ \aas contains such a $K_q$-divisibility fixer with at most $(q-2)n+O(1)$ edges. 
Here is our definition of divisibility fixer.

\begin{definition}\label{def:DivFixer}
Let $q \ge 3$ be an integer.  We say a graph $F$ is a \textit{$K_q$-divisibility fixer} if for every $m \in \{0, \dots, \binom{q}{2}-1\}$ and every $(d_v \in \{0, \dots, q - 2\} : v \in V(F))$ such that $\sum_{v\in V(F)}d_v \equiv 2m \mod q - 1$, there exists a spanning subgraph $F' \subseteq F$ such that $e(F') \equiv m \mod \binom{q}{2}$ and $d_{F'}(v) \equiv d_v \mod q - 1$ for every $v \in V(F)$.
\end{definition}

The following proposition shows how divisibility fixers `fix the divisibility' of a host graph.

\begin{proposition}\label{prop:DivFixer}
Let $q\ge 3$ be an integer. If $G$ is a graph and $F$ is a spanning subgraph of $G$ that is also a $K_q$-divisibility fixer, then there exists a subgraph $F'$ of $F$ such that $G-E(F')$ is $K_q$-divisible.     
\end{proposition}
\begin{proof}
    Choose $m \in \{0, \dots, \binom{q}{2}-1\}$ so that $m + e(G) - e(F) \equiv 0 \mod \binom{q}{2}$, and for each $v \in V(F)$, choose $d_v \in \{0, \dots, q - 2\}$ so that $d_v + d_{G-E(F)}(v) \equiv 0\mod q - 1$.
    Since $F$ is a $K_q$-divisibility fixer, there exists a spanning subgraph $F' \subseteq F$ such that $e(F') \equiv m \mod \binom{q}{2}$ and $d_{F'}(v) \equiv d_v \mod q - 1$ for every $v \in V(F)$.  Hence,  $G - (E(F) \setminus E(F'))$ is $K_q$-divisible, as desired.
\end{proof}

To prove Theorem~\ref{thm:YusterRandomGeneral}, we will embed a $K_q$-divisibility fixer $F$ in $G \sim G(n,p)$ and use Proposition~\ref{prop:DivFixer} to find $F' \subseteq F$ such that $G - F'$ is $K_q$-divisible.  We will then find a $K_q$-decomposition of $G - E(F')$.  
In order to address Yuster's question, we need to construct a $K_q$-divisibility fixer $F$ with at most $qn$ edges, as follows.

\begin{lem}\label{lem:DivFixer}
For each integer $q\ge 3$, there exists a real $C\ge 1$ such that the following holds for all $\beta > 0$: If $p \ge n^{-\frac{2}{q+1}+\beta}$, then \aas $G(n,p)$ contains a $K_q$-divisibility fixer $F$ with $v(F)=n$ such that $e(F) \le (q-2)n + C$.
\end{lem}

To prove the divisibility fixer theorem, we will first show the existence of ``sparse'' divisibility fixers.  For inductive purposes, we will first show the existence of sparse divisibility fixers that are multi-graphs before using fake edges to construct new divisibility fixers which are simple graphs.  Moreover, it will be helpful to have a stronger property for fixing edge divisibility.

\begin{definition}
Let $q \ge 3$ be an integer.  We say a multi-graph $F$ is a \textit{strong $K_q$-divisibility fixer} if for every $m \in \{0, \dots, q(q-1)-1\}$ and every $(d_v \in \{0, \dots, q - 2\} : v \in V(F))$ such that $\sum_{v\in V(F)}d_v \equiv 2m \mod q - 1$, there exists a spanning subgraph $F' \subseteq F$ such that $e(F') \equiv m \mod q(q-1)$ and $d_{F'}(v) \equiv d_v \mod q - 1$ for every $v \in V(F)$.
\end{definition}

Note that a strong $K_q$-divisibility fixer is a $K_q$-divisibility fixer since $\binom{q}{2}\mid q(q - 1)$.

We first show that a multi-graph fixer exists on $3$ vertices.  Multi-graph fixers exist on one or two vertices as well if we allow loops; however, this approach would require more work to get rid of the loops.

\begin{lem}\label{lem:div-fixer-fat-triangle}
    For each integer $q \geq 3$, a 3-vertex loopless multi-graph with at least $q(q-1)$ parallel edges between every pair of vertices is a strong $K_q$-divisibility fixer. 
\end{lem}
\begin{proof}
    Let $F$ have vertex set $\{x,y,z\}$ with at least $q(q - 1)$ parallel edges between every pair of vertices.  
    Suppose we have $d_x, d_y, d_z \in \{0, \dots, q - 2\}$ and $m \in \{0, \dots, q(q - 1)-1\}$ such that $d_x + d_y + d_z \equiv 2m \mod q - 1$.  To show that $F$ is a strong $K_q$-divisibility fixer, we claim that it suffices to show that there is a solution to the following system of equations with three variables $e_{xy}, e_{xz}, e_{yz} \in \{0, \dots, q(q - 1) - 1\}$:
    \begin{align*}
        e_{xy} + e_{xz} &\equiv d_x \mod q(q - 1)\\
        e_{xy} + e_{yz} &\equiv d_y \mod q - 1\\
        e_{xz} + e_{yz} &\equiv d_z \mod q(q - 1)\\
        e_{xy} + e_{xz} + e_{yz} &\equiv m \mod q(q - 1).
    \end{align*}
    Indeed, if such a solution exists, then we can construct the desired subgraph $F' \subseteq F$ by including $e_{xy}$, $e_{xz}$, and $e_{yz}$ parallel edges between $x$ and $y$, $x$ and $z$, and $y$ and $z$, respectively.  Since $q - 1 \mid q(q - 1)$, the first three congruences ensure each vertex has the desired degree modulo $q - 1$, and the final congruence ensures that $e(F') \equiv m \mod q(q - 1)$, as desired.

    To find such a solution, first choose $e_{yz} \in \{0, \dots, q(q - 1)- 1\}$ such that $e_{yz} \equiv m - d_x \mod q(q - 1)$.  Then, choose $e_{xz} \in \{0, \dots, q(q - 1) - 1\}$ such that $e_{xz} \equiv d_z - e_{yz} \mod q(q - 1)$.  By this choice, the third congruence above holds.  Then, choose $e_{xy} \in \{0, \dots, q(q - 1) - 1\}$ such that $e_{xy} \equiv d_x - e_{xz} \mod q(q - 1)$.  By this choice, the first congruence holds.  Moreover, $e_{xy} + e_{xz} + e_{yz} \equiv d_x - e_{xz} + e_{xz} + m - d_x \equiv m \mod q(q - 1)$, so the fourth congruence holds.  Since $d_x + d_y + d_z \equiv 2m \mod q - 1$, using the three congruences we know, we have $d_y \equiv 2m - d_x - d_z \equiv 2(e_{xy} + e_{xz} + e_{yz}) - (e_{xy} + e_{xz}) - (e_{xz} + e_{yz}) \equiv e_{xy} + e_{yz} \mod q - 1$, as desired.
\end{proof}

Next we show that adding a vertex and $q - 2$ non-loop edges preserves the strong divisibility fixing property.  Here is the inductive step in which we need the strong divisibility fixing property.

\begin{lem}\label{lem:div-fixer-inductive-step}
    Let $q \geq 3$ be an integer.  If a multi-graph $F_1$ is a strong $K_q$-divisibility fixer and $F_2$ is obtained from $F_1$ by adding a vertex and at least $q - 2$ non-loop edges incident to it, then $F_2$ is also a strong $K_q$-divisibility fixer.
\end{lem}
\begin{proof}
    Let $F_2$ be obtained from $F_1$ by adding a vertex $z$ and at least $q - 2$ non-loop edges incident to $z$.  Suppose we have $d_v \in \{0, \dots, q - 2\}$ for each $v\in V(F_2)$ and $m \in \{0, \dots, q(q - 1) - 1\}$ such that $\sum_{v\in V(F_2)}d_v \equiv 2m \mod q - 1$.  Let $X$ be a set of $d_z$ edges incident to $z$.  We will construct a subgraph $F' \subseteq F_2$ satisfying $d_{F'}(v) \equiv d_v \mod q - 1$ for every $v \in V(F_2)$ and $e(F') \equiv m \mod q(q - 1)$ such that $z$ is incident to an edge in $F'$ if and only if it is in $X$.  To that end, let $d'_v \in \{0, \dots, q - 2\}$ for each $v \in V(F_1)$ such that $d'_v \equiv d_v - \mu_X(zv) \mod q - 1$, where $\mu_X(zv)$ denotes the number of edges in $X$ incident to both $z$ and $v$, and let $m' \in \{0, \dots, q(q - 1) - 1\}$ such that $m' \equiv m - d_z \mod q(q - 1)$.  Since $\sum_{v\in V(F_1)}d'_v \equiv \sum_{v\in V(F_2)}d_v - 2d_z \equiv 2m - 2d_z \equiv 2m' \mod q - 1$ and $F_1$ is a $K_q$-divisibility fixer, there exists a spanning subgraph $F'_1 \subseteq F_1$ such that $e(F'_1) \equiv m' \mod q(q - 1)$ and $d_{F'_1}(v) \equiv d'_v \mod q - 1$ for every $v \in V(F_1)$.  By adding the vertex $z$ and the edges $X$ to $F'_1$, we obtain the desired subgraph $F' \subseteq F_2$. 
\end{proof}

With the following lemma, we can construct a $K_q$-divisibility fixer which is a simple graph from a $K_q$-divisibility fixer which is a multi-graph.  By replacing an edge with a fake edge, we can eliminate parallel edges while preserving the divisibility fixing property.  Although this operation does not preserve the strong divisibility fixing property, this lemma is sufficient for our purposes.

\begin{lem}\label{lem:div-fixer-fake-edge}
    Let $q \geq 3$ be an integer. If a multi-graph $F_1$ is a $K_q$-divisibility fixer, $e \in E(F_1)$, and $X \cong {\rm FakeEdge}_q(e)$ is edge-disjoint from $F_1$ with $X \subseteq V(F_1)$, then $F_2 \coloneqq F_1 - e \cup X$ is also a $K_q$-divisibility fixer.
\end{lem}
\begin{proof}
    Suppose we have $d_v \in \{0, \dots, q - 2\}$ for each $v\in V(F_2)$ and $m \in \{0, \dots, \binom{q}{2}-1\}$ such that $\sum_{v\in V(F_2)}d_v \equiv 2m \mod q - 1$.  To show $F_2$ is a $K_q$-divisibility fixer, we show there exists a spanning subgraph $F' \subseteq F_2$ such that $d_{F'}(v) \equiv d_v \mod q - 1$ for every $v \in V(F_2)$ and $e(F') \equiv m \mod \binom{q}{2}$.
    Since $F_1$ is a $K_q$-divisibility fixer, there exists a spanning subgraph $F'_1 \subseteq F_1$ such that $d_{F'_1}(v) \equiv d_v \mod q - 1$ and $e(F'_1) \equiv m \mod \binom{q}{2}$.  If $e \notin E(F'_1)$, then $F'_1 \subseteq F_2$, so letting $F' \coloneqq F'_1$, we obtain the desired subgraph.  Otherwise, let $F' \coloneqq F'_1 - e \cup X$.  Now $F' \subseteq F'_2$, as required, and by Fact~\ref{fact:Div}, we have $e(F') \equiv e(F_1) \equiv m \mod \binom{q}{2}$ and $d_{F'}(v) \equiv d_{F'_1}(v) \equiv d_v \mod q - 1$ for every $v \in V(F_2)$, as desired.
\end{proof}

\begin{proof}[Proof of Lemma~\ref{lem:DivFixer}]
    
Let $p_1, p_2 \coloneqq n^{-\frac{2}{q+1} + \beta}/2$, and let $G_1 \sim G(n,p_1)$ and $G_2\sim G(n,p_2)$.  We can couple $G_1 \cup G_2$ with $G(n,p_1 + p_2 - p_1p_2)$ and $p_1 + p_2 - p_1p_2 \leq n^{-\frac{2}{q+1} + \beta}$, so it suffices to show that \aas $G_1 \cup G_2$ contains a $K_q$-divisibility fixer $F$ with $v(F)=n$ such that $e(F) \le (q-2)n + C$.

Let $F_1$ be the multi-graph obtained from $P_n^{q-2}$ (the $(q - 2)$nd power of an $n$-vertex path) by adding a set $S$ of edges between distinct vertices among the first $\max\{3, q - 2\}$ so that every pair of them has $q(q - 1)$ parallel edges; that is, $V(F_1)$ can be enumerated $\{v_1, \dots, v_n\}$ where $v_i$ and $v_j$ are adjacent if $|i - j| \leq q - 2$ for $i\neq j$, and for distinct $i, j \in \{1, \dots, \max\{3, q - 2\}\}$, there are $q(q - 1)$ parallel edges between $v_i$ and $v_j$.  By Lemma~\ref{lem:div-fixer-fat-triangle}, $F_1[\{v_1, v_2, v_3\}]$ is a strong $K_q$-divisibility fixer.  By induction, using Lemma~\ref{lem:div-fixer-inductive-step}, we have $F_1[\{v_1, \dots, v_i\}]$ is also a strong $K_q$-divisibility fixer for each $i \in \{3, \dots, n\}$.  
For each $e \in S$, let $H_e \cong {\rm FakeEdge}_q(e)$ so that $H_e$ and $H_{e'}$ are edge-disjoint for every distinct $e,e' \in S$ and $V(H_e)$ intersects $V(F_1)$ only in the vertices of $e$.  Let $H \coloneqq \bigcup_{e\in S}H_e$.  We will embed $F_1 - S$ in $G_1$ and embed $H$ in $G_2$ using Theorem~\ref{thm:Embed}.

By the Park--Pham Theorem~\cite{PP23}, $G(n,p_1)$ \aas contains the $(q - 2)$nd power of a Hamilton cycle. (For more information on the precise threshold for $G(n,p)$ to contain a power of a Hamilton cycle, see the paper of Kahn, Narayanan, and Park~\cite{KNP21}.)  Hence, we may assume $F_1 - S$ is a subgraph of $G_1$, and hence there exists a bijection $\psi$ from $R:=V(F_1-S)$ to $V(G_1)$ that preserves edges. 

Now we embed $H$ into $G_2 \setminus F$ using Theorem~\ref{thm:Embed} with $K_n \setminus (F_1 - S)$ playing the role of $G$ and $p_2$ playing the role of $p$.
Note that every $R$-bridge of $H$ is a ${\rm FakeEdge}_q(e)$ for some $e$. Hence by Proposition~\ref{prop:FakeEdge}, we have that that $m_2(T, R) = \frac{q+1}{2}$ for every $T \in \cB(H, R)$. 
Let $h:=q^5$. Since ${\rm FakeEdge}_q(e)$ has at most $q^5$ edges, we have that every $R$-bridge of $H$ has at most $h$ edges.  
Let $C'$ and $\varepsilon$ be as in Theorem~\ref{thm:Embed} for $h$, and let $C \coloneqq \max\{C', e(H) + 2(q-2)\}$.

Note that $\delta(K_n\setminus (F_1 - S)) \ge (1-\varepsilon)n$ since $\Delta(F_1-S) \le 2 (q-2)$ and $n$ is large enough. Also note that $p_2 \geq n^{-1/m_2(T, R)}\cdot \log^{8h+3} n$ for every $T \in \cB(H, R)$ since $n$ is large enough. Note that there are at most $q^4$ $R$-bridges of $H$ and so any two vertices of $R$ are in at most $q^4$ ($\le h$) $R$-bridges of $H$. Thus, we also find that $\Delta(H) \le q^4\cdot q^5 = q^9 \le {p_2\cdot n}/{C}$ since $n$ is large enough. 

Hence by Theorem~\ref{thm:Embed}, \aas there exist $\psi$-respecting embeddings $\phi_e : H_e \hookrightarrow G_2 - (E(F_1)\setminus S)$ such that $\phi_e(H_e)$ and $\phi_{e'}(H_{e'})$ are edge-disjoint for every pair of distinct $e,e' \in S$.  Let $F \coloneqq (F_1 - S) \cup \bigcup_{e\in S}\phi_e(H_e)$.
By Lemma~\ref{lem:div-fixer-fake-edge} (applied repeatedly, once for each edge of $S$), $F$ is a $K_q$-divisibility fixer in $G_1\cup G_2$ as desired.
Moreover, $e(F) \leq e(F_1 - S) + e(H) \leq (q - 2)n + C$, as required.
\end{proof}

\section{Regularity Boosting}\label{s:Regularity}

Rephrasing a decomposition problem in terms 
of a perfect matching of an auxiliary hypergraph is useful. To that end, we have the following definition.

\begin{definition}[Design hypergraph]\label{def:design-hypergraph}
Let $F$ be a hypergraph. If $G$ is a hypergraph, then the \emph{$F$-design hypergraph of $G$}, denoted ${\rm Design}(G,F)$, 
is the hypergraph with $V({\rm Design}(G,F))=E(G)$ and $E({\rm Design}(G,F)) = \{S\subseteq E(G): S \text{ is isomorphic to } F\}$.
\end{definition}

Note that $\mathrm{Design}(K_n, K_q)$ is $\binom{n - 2}{q - 2}$-regular and $\Delta_2(\mathrm{Design}(K_n, K_q)) \leq \binom{n - 3}{q - 3}$.  

In this section, we prove the following theorem which will be used for the regularity boosting of Step (3) of the main proofs.
Recall that we say a hypergraph $\cD$ is $(1 \pm \eps)D$-regular if $d_{\cD}(v) = (1 \pm \eps)D$ for all $v \in V(\cH)$.

\begin{thm}\label{thm:RegularCliques}
For each integer $q \ge 3$, there exists a real $C \ge 1$ such that the following holds for all $\beta > 0$: If $p \ge n^{-\frac{2}{q+1} + \beta}$ and $S'\subseteq S\subseteq K_n$ such that $\Delta(S) \le \frac{pn}{C}$, then \aas there exists a subhypergraph $\mc{D} \subseteq {\rm Design}( (K_n\setminus S)_p\cup S',K_q)$ which is  $\left(1/2\pm n^{-\beta/8}\right)\cdot p^{\binom{q}{2}-1}\cdot~\binom{n}{q-2}$-regular.
\end{thm}

First we recall the notion of \emph{edge-gadgets} from~\cite{GKLO16} as follows.

\begin{proposition}[Proposition 6.2 in~\cite{GKLO16}]\label{prop:EdgeGadget}
Let $q > r \ge 1$ be integers and let $e$ and $J$ be disjoint sets with $|e| = r$ and $|J| = q$. Let $G$ be the complete $r$-uniform hypergraph on vertex set $e\cup J$. There exists a function $\psi: \binom{V(G)}{q} \rightarrow \mathbb{R}$ such that
\begin{itemize}
    \item[(i)] for all $e'\in G$, we have $\sum_{H\in \binom{V(G)}{q}: e'\subseteq H} \psi(H) = {\bf 1}_{e=e'}$, and
    \item[(ii)] for all $H\in \binom{V(G)}{q}$, we have $|\psi(H)| \le \frac{2^{r-j}\cdot (r-j)!}{\binom{q-r+j}{j}} \le 2^r\cdot r!$ where $j:= |e\cap H|$.
\end{itemize}
\end{proposition}

To prove Theorem~\ref{thm:RegularCliques}, we require a more general fractional decomposition version of the Boost Lemma~\cite{GKLO16}; the proof is essentially the same as the second half of that lemma but we include it for completeness, but first a definition. 

\begin{definition}\label{def:fractional-decomposition}
A \emph{$K_q$-weighting} of a graph $G$ is a function $\psi$ from the $q$-cliques of $G$ to $\mathbb R_{\ge 0}$; for an edge $e$ of $G$, we let $\psi(e):= \sum_{Q: e\subseteq Q} \psi(Q)$; similarly for a set of edges $F$ of $G$ we let $\psi(F):= \sum_{Q: V(F)\subseteq Q} \psi(Q)$. The \emph{support} of a $K_q$-weighting $\psi$ is the set of $q$-cliques $Q$ of $G$ such that $\psi(Q)>0$.

A \emph{fractional $K_q$-packing} is a $K_q$-weighting such that $\psi(e)\in [0,1]$ for all $e\in G$, and a \textit{fractional $K_q$-decomposition} is a $K_q$-weighting such that $\psi(e) = 1$ for all $e \in G$.
\end{definition}

We note the same definitions extend naturally to hypergraphs. Here we state our general boost lemma for hypergraphs for potential future uses even though we require it only for graphs.  Throughout this section, if $\cH$ is a family of $q$-cliques of a hypergraph $G$, then we use notation as if $\cH$ is a hypergraph with $V(\cH) = V(G)$ and $E(\cH) = \cH \subseteq \binom{V(G)}{q}$.

\begin{lem}\label{lem:GeneralBoostLemma}
For integers $q > r \ge 2$, there exists $\alpha > 0$ such that the following holds: Let $G$ be an $r$-uniform hypergraph, let $\mathcal{H}$ be a family of $q$-cliques of $G$ and $\mathcal{Q}$ be a family of $(q+r)$-cliques of $G$ such that for all $Q\in \mc{Q}$, we have $\binom{Q}{q} \subseteq \mc{H}$. Let $\phi:E(G)\rightarrow [0,1]$. If there exists a real $d > 0$ such that 
$$\frac{\bigg| |\mc{H}(e)|-d\cdot \phi(e) \bigg|}{|\mathcal{Q}(e)|} \cdot \max_{R\in \binom{V(G)}{q}} |\mathcal{Q}(R)| \le \alpha$$
for all $e \in E(G)$, then there exists a $K_q^r$-weighting $\psi$ of $G$ with $\psi(e)=\phi(e)$ for all $e\in G$ whose support is in $\mc{H}$ and whose weights are in $\left(1\pm \frac{1}{2}\right) \cdot \frac{1}{d}$.
\end{lem}
\begin{proof}
For each $e\in G$ and $Q\in \mc{Q}(e)$, there exists $\psi_{e,Q}:\binom{Q}{q}\rightarrow \mathbb{R}$ as in Proposition~\ref{prop:EdgeGadget}. 
We extend $\psi_{e,q}$ to have domain $\binom{V(G)}{q}$ by letting $\psi_{e,Q}(H) = 0$ if $H$ is not contained in $\binom{Q}{q}$.
For every $e\in G$, define
$$c_e:= \frac{d\cdot \phi(e)-|\mc{H}(e)|}{|\mc{Q}(e)|}.$$
Let $M:= \max_{R\in \binom{V(G)}{q}} |\mathcal{Q}(R)|$. Note that by assumption
$$|c_e| \le \frac{\alpha}{M}.$$

We now define $\psi: \mc{H}\rightarrow \mathbb{R}$ as follows:
$$\psi(H):= \frac{1}{d} + \frac{1}{d} \cdot \sum_{e\in G} c_e \cdot \sum_{Q\in \mc{Q}(e)} \psi_{e,Q}(H).$$
Note that for every $e\in G$, we have
\begin{align*}\sum_{H\in \mc{H}(e)} \psi(H) &= \frac{|\mc{H}(e)|}{d} + \frac{1}{d}\cdot \sum_{e'\in G}c_{e'} \cdot \sum_{Q\in \mc{Q}(e')} \sum_{H\in \mc{H}(e)} \psi_{e',Q}(H)\\
&= \frac{|\mc{H}(e)|}{d} + \frac{1}{d}\cdot  \sum_{e'\in G}c_{e'} \cdot \sum_{Q\in \mc{Q}(e')} {\bf 1}_{e=e'}\\
&= \frac{|\mc{H}(e)| + c_{e} \cdot |\mc{Q}(e)|}{d} = \phi(e).\\
\end{align*}
Let $\gamma:= 2^r\cdot r!$. Furthermore, for every $H\in \mc{H}$, we have 
\begin{align*}
\bigg|d\cdot \psi(H)-1\bigg| &= \bigg|\sum_{e\in G} c_e\cdot \sum_{Q\in \mc{Q}(e)} \psi_{e,Q}(H)\bigg| \le \sum_{e\in G} |c_e|\cdot \sum_{Q\in \mc{Q}(e): H\subseteq Q} |\psi_{e,Q}(H)| \\ 
&\le \frac{\alpha}{M} \cdot \sum_{e\in G}  \sum_{Q\in \mc{Q}(e): H\subseteq Q} 2^r\cdot r! \le \frac{\alpha}{M} \cdot \gamma \cdot \sum_{Q\in \mc{Q}(H)} \sum_{e\in Q} 1\\
&\le \frac{\alpha}{M} \cdot \gamma \cdot \binom{q+r}{r} \cdot |\mc{Q}(H)| \le \frac{\alpha}{M} \cdot \gamma \cdot \binom{q+r}{r} \cdot M \\
&\le \alpha \cdot \gamma \cdot \binom{q+r}{r} \le \frac{1}{2},
\end{align*}
where the last inequality follows since we choose $\alpha \le {1}/({2\cdot \binom{q+r}{r}\cdot \gamma})$. Thus $\psi$ is a $K_q^r$-weighting of $G$ with $\psi(e)=\phi(e)$ for all $e\in G$ whose support is in $\mc{H}$ and whose weights are in $\left(1\pm \frac{1}{2}\right) \cdot \frac{1}{d}$ as desired.
\end{proof}

Next we prove the following lemma which shows that in the set-up of Theorem~\ref{thm:RegularCliques} there exists a fractional $K_q$-packing of $(K_n\setminus S)\cup S'$ whose expected weight in $(K_n\setminus S)_p\cup S'$ is $1$ for every edge. 
In order to control the expected weight of a clique in $(K_n\setminus S)_p\cup S'$, we only weight cliques that induce at most one edge of $S'$.

\begin{lem}\label{lem:YusterFractionalPrelim}
For each integer $q\ge 3$, there exists a real $C \ge 1$ such that the following holds for all $n \geq C$ and $p \in (0,1]$.
Let $S'\subseteq S\subseteq K_n$, let $\mc{H}$ be the set of $q$-cliques in $K_n\setminus S$, and let $\mc{H}'$ be the set of $q$-cliques in $(K_n\setminus S)\cup S'$ containing exactly one edge of $S'$. If $\Delta(S) \le \frac{pn}{C}$, then there exists a fractional $K_q$-packing $\psi$ of $(K_n\setminus S) \cup S'$ such that $\psi(e)=1$ for all $e\in S'$, and for all $e\in K_n\setminus S$ we have 
$$\sum_{H \in \cH(e)} \psi(H) + \frac{1}{p}\cdot \sum_{H \in \cH'(e)} \psi(H)= 1,$$ 
 and each clique in the support of $\psi$ is in $\cH \cup \cH'$ and has weight in $\left.(1 \pm 1/2) \middle/\binom{n}{q-2}\right.$.
\end{lem}
\begin{proof}
We choose $C$ large enough as needed throughout the proof.
For each $e\in S'$ and $H\in \mc{H}'(e)$, let $\psi_1(H) := \frac{1}{|\mc{H}'(e)|}$. Note that $\psi_1(e) = 1$ for all $e\in S'$ by construction (recall that $\psi_1(e):= \sum_{H\in \mc{H}'(e)} \psi_1(H)$). 

For $e\in K_n\setminus S$, we define
$$\phi(e):= 1-\frac{1}{p}\cdot \psi_1(e).$$
We note that for all $e\in S'$, we have that $|\mc{H}'(e)| \ge \frac{1}{2} \cdot \binom{n}{q-2}$ since $\Delta(S)\le \frac{pn}{C}$ and $C$ is large enough. 
Hence, we find that for all $e \in K_n \setminus S$, we have $\psi_1(e) \leq 2|\cH'(e)| / \binom{n}{q - 2} \leq 2q\Delta(S)n^{q - 3} / \binom{n}{q - 2} \leq pq^q / C$\COMMENT{by the same argument in Claim~\ref{claim:fractional-concentration-codegrees-bound}\ref{fractional-concentration-better-degree-upper-bound}, we have $|\cH'(e)| \leq q\Delta(S)n^{q-3}$}
 and thus $\phi(e) \ge 1- {q^q}/{C}$.

Let $d:=\binom{n}{q-2}$ and let $\mc{Q}$ be the set of $(q+2)$-cliques in $K_n\setminus S$. Let $\alpha$ be as in Lemma~\ref{lem:GeneralBoostLemma} for $q$ and $r=2$. Since $C$ is large enough, we find that $|\mc{Q}(e)| \ge \frac{1}{2} \cdot \binom{n}{q}$ for all $e\in K_n\setminus S$. We also have that $\max_{R\in \binom{V(K_n\setminus S)}{q}} |\mc{Q}(R)| \le \binom{n}{2}$. Meanwhile, for all $e\in K_n\setminus S$, we have that $d\ge |\mc{H}(e)|\ge (1-\frac{qp}{C})\cdot d$ since $\Delta(S)\le \frac{pn}{C}$. Combining, we find that
$$\frac{\bigg| |\mc{H}(e)|-d\cdot \phi(e) \bigg|}{|\mathcal{Q}(e)|} \cdot \max_{R\in \binom{V(K_n\setminus S)}{q}} |\mathcal{Q}(R)| \le \frac{ \left(\frac{qp}{C}+\frac{q^q}{C}\right)\cdot d}{\frac{1}{2}\cdot \binom{n}{q}} \cdot \binom{n}{2} \le \frac{4q^q}{C}\binom{n}{q - 2}\binom{n}{2} / \binom{n}{q} \leq \frac{2q^{q+2}}{C} \le \alpha,$$
\COMMENT{
\begin{equation*}
    \binom{n}{q - 2}\binom{n}{2} / \binom{n}{q} = \frac{\frac{n!}{(n - (q - 2))!(q - 2)!}\frac{n(n-1)}{2}}{\frac{n!}{(n - q)!q!}} = \frac{n(n-1)}{(n - q + 2)(n - q + 1)}\frac{q(q-1)}{2} \leq \frac{q^2}{2}
\end{equation*}
}
since $C$ is large enough. Hence the conditions of Lemma~\ref{lem:GeneralBoostLemma} are satisfied for $K_n\setminus S$, $\mc{H}$, $\mc{Q}$, $\phi$ and $d$.

Now by Lemma~\ref{lem:GeneralBoostLemma}, there exists a fractional $K_q$-packing $\psi_2$ of $K_n\setminus S$ such that $\psi_2(e)=\phi(e)$ for all $e\in K_n\setminus S$ and whose non-zero weights are in $(1\pm 1/2)\cdot \frac{1}{d}$. Let $\psi:= \psi_1+\psi_2$. Note that for $e\in S'$, $\psi(e)=\psi_1(e)=1$. Meanwhile for $e\in K_n\setminus S$, we have that 
$$\psi_2(e) +\frac{1}{p}\cdot \psi_1(e) = \phi(e)+\frac{1}{p}\cdot \psi_1(e) = 1$$ 
by definition of $\phi$. Moreover, each clique in the support of $\psi$ has a weight of $(1\pm 1/2)\cdot \frac{1}{d}$ (even the cliques in $\mc{H}'$ which have weight $\frac{1}{|\mc{H}'(e)|}$ and $\frac{d}{2}\le |\mc{H}'(e)|\le d$ as noted before). Thus $\psi$ is as desired.
\end{proof}

Next we show that what survives of $\psi$ in $(K_n\setminus S)_p$ is close to its expectation yielding a $K_q$-weighting of $(K_n\setminus S)\cup S'$ whose weights on edges are very regular.

\begin{lem}\label{lem:YusterFractionalConcentrated}
For each integer $q\ge 3$, there exists a real $C \ge 1$ such that the following holds for all $\beta > 0$: If $S'\subseteq S\subseteq K_n$ and $1\ge p \ge n^{-\frac{2}{q+1}+\beta}$ such that $\Delta(S) \le \frac{pn}{C}$, then \aas there exists a $K_q$-weighting $\psi$ of $(K_n\setminus S)_p \cup S'$ such that all of the following hold:
\begin{enumerate}[label=(\arabic*)]
    \item\label{fractional-concentration-weight} each clique in the support of $\psi$ has weight in $\left.(1 \pm 1/2)\middle/  \left(p^{\binom{q}{2}-1}\cdot \binom{n}{q-2}\right)\right.$ and 
    \item\label{fractional-concentration-degree} $\psi(e) = 1\pm n^{-\beta/7}$ for all $e\in (K_n\setminus S)_p\cup S'$.
\end{enumerate}
\end{lem}
\begin{proof}
Note we assume $n$ is large enough as needed throughout the proof since the outcome need only hold asymptotically almost surely. Similarly we choose $C\ge C'$ where $C'$ is the value in Lemma~\ref{lem:YusterFractionalPrelim}. Let $\psi$ be a fractional $K_q$-packing of $(K_n\setminus S) \cup S'$ as in Lemma~\ref{lem:YusterFractionalPrelim}. Let $\mc{H}_1$ be the set of $q$-cliques in $K_n\setminus S$ with nonzero weight in $\psi$. Let $\mc{H}_2$ be the set of $q$-cliques in $(K_n\setminus S)\cup S'$ that contain one edge of $S'$ and have nonzero weight in $\psi$.
For $i\in \{1,2\}$, define $\psi_i$ as $\psi_i(H)=\psi(H)$ if $H\in \mc{H}_i$ and $0$ otherwise. Recall that by the outcome of Lemma~\ref{lem:YusterFractionalPrelim}, we have that 
\begin{equation}\label{fractional-concentrated-psi-e-in-S'}
    \psi(e)=\psi_2(e)=1 \text{ for all } e\in S',
\end{equation}
\begin{equation}\label{fractional-concentrated-psi-e-notin-S'}
    \psi_1(e)+\frac{1}{p}\cdot \psi_2(e) = 1 \text{ for all } e\in K_n\setminus S,
\end{equation}
and
\begin{equation}\label{fractional-concentration-psi-support}
    \text{$\psi(H) = \left.(1 \pm 1/2) \middle/\binom{n}{q-2}\right.$ if $H \in \cH_1 \cup \cH_2$ and $\psi(H) = 0$ otherwise.}
\end{equation}
We let $\psi_p$ denote the $K_q$-weighting of $(K_n\setminus S)\cup S'$ where 
\begin{equation*}
    \psi_p(H) \coloneqq \left\{\begin{array}{l l}
    \left.\psi(H)\middle/p^{\binom{q}{2}-1}\right. &\text{if } E(H)\subseteq (K_n\setminus S)_p\cup S'\\
    0 &\text{otherwise.}
    \end{array}\right.
\end{equation*}
Similarly for $i\in\{1,2\}$, we define $\psi_{i,p}$ as $\psi_{i,p}(H)=\psi_p(H)$ if $H\in \mc{H}_i$ and $0$ otherwise.

We show that $\psi_p$ satisfies \ref{fractional-concentration-weight} and \ref{fractional-concentration-degree}.  
Indeed, \ref{fractional-concentration-weight} holds for $\psi_p$ by \eqref{fractional-concentration-psi-support}.  To show \ref{fractional-concentration-degree}, we define the following weighted hypergraphs to which we will apply Corollary~\ref{cor:KimVu-poly} (with $n^2$ playing the role of $n$).  
We abuse notation and use $w$ to denote the weight function in each of these weighted hypergraphs, but there is no ambiguity as these hypergraphs are pairwise edge-disjoint.  Let $k \coloneqq \binom{q}{2}-1$.

\begin{itemize}
\item For $e\in S'$, let $(\cF_{0,e}, w)$ be the weighted hypergraph with $V(\cF_{0,e}):= (K_n\setminus S)\setminus \{e\}$ and $E(\cF_{0,e}):=\{E(H)\setminus \{e\}: H\in \mc{H}_2(e)\}$ and edge weights $w(Z):= \psi(V(Z)\cup e)$.
Let $\cF_{0,e}':= \cF[V(\cF_{0,e})_p]$, and let $K_{0,e}\coloneqq 1$. 
Note that $\cF_{0,e}$ is $k$-uniform and $e(\cF_{0,e}) \leq n^2$.

\item For $e\in K_n\setminus S$, let $(\cF_{1,e}, w)$ be the weighted hypergraph with $V(\cF_{1,e}):= (K_n\setminus S)\setminus \{e\}$ and $E(\cF_{1,e}):=\{E(H)\setminus \{e\}: H\in \mc{H}_1(e)\}$ and edge weights $w(Z):= \psi(V(Z)\cup e)$.
Let $\cF_{1,e}':= \cF[V(\cF_{1,e})_p]$, and let $K_{1,e}:= w(\cF_{1,e})=\psi_1(e)$.
Note that $\cF_{1,e}$ is $k$-uniform and $e(\cF_{1,e}) \leq n^2$.

\item For $e\in K_n\setminus S$, let $(\cF_{2,e}, w)$ be the weighted hypergraph with $V(\cF_{2,e}):= (K_n\setminus S)\setminus \{e\}$ and $E(\cF_{2,e}):=\{E(H)\setminus (S\cup \{e\}): H\in \mc{H}_2(e)\}$ and edge weights $w(Z):= \psi(V(Z)\cup e)$.
Let $\cF_{2,e}':= \cF[V(\cF_{2,e})_p]$, and let $K_{2,e}:=\max\left\{\psi_2(e),~p\cdot n^{-\beta/6}\right\}$. 
Note that $\cF_{2,e}$ is $(k - 1)$-uniform and $e(\cF_{2,e}) \leq n^2$.



\end{itemize}

Next we bound the weighted codegrees of these hypergraphs, but first we need the following claim.
\begin{claim}\label{claim:fractional-concentration-codegrees-bound}
    For every $U \subseteq (K_n \setminus S) \cup S'$, the following holds:
\begin{enumerate}[label=(\alph*)]
    \item\label{fractional-concentration-easy-degree-upper-bound} 
    $|\cH_1(V(U))|, |\cH_2(V(U))| \leq n^{q - v(U)}$;
    \item\label{fractional-concentration-better-degree-upper-bound} 
    if $\binom{V(U)}{2} \subseteq K_n \setminus S$,  then $|\cH_2(V(U))| \leq q\cdot \Delta(S)\cdot n^{q - 1 - v(U)} \leq q\cdot  p \cdot n^{q - v(U)} / C$;
    \item\label{fractional-concentration-2-density} if $\cH_1(V(U))$ or $\cH_2(V(U))$ is non-empty, then $e(U) - 1 \leq (v(U) - 2)\cdot(q + 1) / 2$.
\end{enumerate}
\end{claim}
\begin{proofclaim}
Let $U \subseteq (K_n \setminus S) \cup S'$.  First, note that if $\cH_1(V(U))$ or $\cH_2(V(U))$ is non-empty, then $U$ is isomorphic to a subgraph of $K_q$.  
Hence, $v(U) \leq q$, and there are at most $\binom{n - v(U)}{q - v(U)} \leq n^{q - v(U)}$ $q$-cliques of $K_n$ containing $V(U)$, so \ref{fractional-concentration-easy-degree-upper-bound} holds. 
To prove \ref{fractional-concentration-better-degree-upper-bound}, we assume $v(U) < q$, as otherwise $|\cH_2(V(U))| = 0$ given $\binom{V(U)}{2} \subseteq K_n\setminus S$.  
If $\binom{V(U)}{2} \subseteq K_n \setminus S$, then there are at most $v(U)\cdot \Delta(S)\cdot n^{q - 1 - v(U)}$ $q$-cliques of $\cH_2(V(U))$ containing vertices $u,v$ with $u \in V(U)$ and $uv \in S'$ and at most $|S'|\cdot n^{q - 2 - v(U)} \leq \Delta(S)\cdot n^{q - 1 - v(U)}$ $q$-cliques of $\cH_2(V(U))$ containing vertices $u,v$ with $u,v\notin V(U)$ and $uv \in S'$; in total, there are at most $(v(U) + 1)\cdot n^{q - 2 - v(U)} \leq q\cdot \Delta(S)\cdot n^{q - 1 - v(U)}$ $q$-cliques of $\cH_2(V(U))$, so \ref{fractional-concentration-better-degree-upper-bound} holds.
Finally, every subgraph of $K_q$ has 2-density at most $m_2(K_q) = 2 / (q + 1)$, so \ref{fractional-concentration-2-density} holds.
\end{proofclaim}

\begin{claim}\label{cl:codegF}
Let $e \in S'$.  For all $i\in [k]$, we have $\Delta_i(\cF_{0,e}, w)\le K_{0,e}\cdot p^i\cdot n^{-\beta/2}$.
\end{claim}
\begin{proofclaim}
Let $i\in [k]$, and let $U\subseteq (K_n\setminus S)\setminus \{e\}$ with $|U|=i$.  To show $\Delta_i(\cF_{0,e}) \leq K_{0,e} \cdot p^i\cdot n^{-\beta / 2}$, it suffices to show that $d_{\cF_{0,e}, w}(U) \le p^i\cdot n^{-\beta/2}$ since $K_{0,e}=1$.

Let $U':= U\cup \{e\}$. By \eqref{fractional-concentration-psi-support} and \ref{fractional-concentration-easy-degree-upper-bound}, we have
$$d_{\cF_{0,e},w}(U) = \sum_{Z\in \mc{H}_2(V(U'))} \psi(V(Z)) \leq \left.\frac{3}{2}\cdot |\cH_2(V(U'))|\middle/\binom{n}{q - 2}\right. \leq q^q\cdot n^{2 - v(U')}.$$
Furthermore, by \ref{fractional-concentration-2-density}, $i=e(U')-1\le (v(U')-2)\cdot (q+1) / 2$. Hence,
$$p^i \ge \frac{n^{\beta\cdot i}}{n^{i\cdot 2/(q+1)}} \geq \frac{n^{\beta\cdot i}}{n^{v(U')-2}} \ge d_{\cF_{0,e},w}(U) \cdot n^{\beta/2},$$
as desired where we used that $n$ is large enough and $i\ge 1$. 
\end{proofclaim}

\begin{claim}\label{cl:codegF1}
Let $e \in K_n \setminus S$.  For all $i\in [k]$, we have $\Delta_{i}(\cF_{1,e}, w)\le K_{1,e}\cdot p^i\cdot n^{-\beta/2}$.
\end{claim}
\begin{proofclaim}
By \ref{fractional-concentration-better-degree-upper-bound} and \eqref{fractional-concentration-psi-support}, we find that 
$$\psi_2(e) \le \left.\frac{3}{2}\cdot |\mc{H}_2(e)|\middle/{\binom{n}{q-2}}\right. \le q^q\cdot \frac{p}{C} \leq \frac{p}{2},$$ 
so by \eqref{fractional-concentrated-psi-e-notin-S'} $K_{1,e} = \psi_1(e) = 1 - \frac{1}{p}\cdot \psi_2(e) \ge \frac{1}{2}$.
We omit the remainder of the proof as it is nearly identical to that of Claim~\ref{cl:codegF} 
\COMMENT{
Let $i\in [k]$, and let $U\subseteq (K_n\setminus S)\setminus \{e\}$ with $|U|=i$. To show $\Delta_i(\cF_{1,e}, w) \leq K_{1,e}\cdot p^i \cdot n^{-\beta/2}$, it suffices to show that $d_{\cF_{1,e}, w}(U) \le p^i\cdot n^{-\beta/2} / 2$ since $K_{1,e} \geq 1/2$.\\
\indent Let $U':= U\cup \{e\}$.  By \eqref{fractional-concentration-psi-support} and \ref{fractional-concentration-easy-degree-upper-bound}, we have
$$d_{\cF_{1,e},w}(U) = \sum_{Z\in \mc{H}_1(V(U'))} \psi(V(Z)) \leq \left.\frac{3}{2}|\cH_2(V(U'))|\middle/\binom{n}{q - 2}\right. \leq q^q\cdot n^{2 - v(U')}.$$
Furthermore, by \ref{fractional-concentration-2-density}, $i=e(U')-1\le (v(U')-2)\cdot (q+1) / 2$. Hence,
$$p^i \ge \frac{n^{\beta\cdot i}}{n^{i\cdot 2/(q+1)}} \geq \frac{n^{\beta\cdot i}}{n^{v(U')-2}} \ge 2 \cdot d_{\cF_{1,e},w}(U) \cdot n^{\beta/2},$$
as desired where we used that $n$ is large enough and $i\ge 1$.\\ 
}
\end{proofclaim}

\begin{claim}\label{cl:codegF2}
Let $e \in K_n\setminus S$.  For all $i\in [k-1]$, we have $\Delta_{i}(\cF_{2,e}, w)\le K_{2,e}\cdot p^i\cdot n^{-\beta/2}$.
\end{claim}
\begin{proofclaim}
Let $i\in [k-1]$, and let $U\subseteq (K_n\setminus S)\setminus \{e\}$ with $|U|=i$. To show $\Delta_i(\cF_{2,e}) \leq K_{2,e}\cdot p^i\cdot n^{-\beta / 2}$, it suffices to show that $d_{\cF_{2,e},w}(U) \le K_{2,e}\cdot p^i\cdot n^{-\beta/2}$.

Let $U':=U\cup \{e\}$. We split into two cases depending on whether $\binom{V(U')}{2} \subseteq K_n\setminus S$.

First suppose that $\binom{V(U')}{2} \subseteq K_n\setminus S$. 
By \eqref{fractional-concentration-psi-support} and \ref{fractional-concentration-better-degree-upper-bound}, we have
\begin{equation*}
    d_{\cF_{2,e},w}(U) = \sum_{Z\in \mc{H}_2(V(U'))} \psi(V(Z)) \leq \left.\frac{3}{2}\cdot |\cH_2(V(U'))|\middle/\binom{n}{q - 2}\right. \leq \frac{q^q\cdot p}{C\cdot n^{v(U') - 2}}.
\end{equation*}
Furthermore, by \ref{fractional-concentration-2-density}, $i=e(U')-1\le (v(U')-2)\cdot (q+1) / 2$.  Hence $p^i \ge \frac{n^{\beta\cdot i}}{n^{v(U')-2}}$.
Substituting this into the inequality above, we find that
$$d_{\cF_{2,e},w}(U) \le \frac{p^{i+1}}{n^{\beta\cdot i}}\cdot \frac{q^q}{C} \le p\cdot n^{-\beta/6} \cdot p^i\cdot n^{-\beta/2} \le K_{2,e}\cdot p^i\cdot n^{-\beta/2}$$
as desired where we used that $n$ is large enough and $K_{2,e} \ge p\cdot n^{-\beta/6}$ by definition of $K_{2,e}$.

So we assume that $\binom{V(U')}{2} \cap S\ne \emptyset$, and in particular we may assume that there exists $f \in S' \cap \binom{V(U')}{2}$.
Let $U'':= U'\cup \{f\}$.
By \eqref{fractional-concentration-psi-support} and \ref{fractional-concentration-easy-degree-upper-bound}, we have
\begin{equation*}
    d_{\cF_{2,e},w}(U) = \sum_{Z\in \mc{H}_2(V(U''))} \psi(V(Z)) \leq \left.\frac{3}{2}\cdot |\cH_2(V(U''))|\middle/\binom{n}{q - 2}\right. \leq q^q \cdot n^{2 - v(U'')}.
\end{equation*}
Furthermore, by \ref{fractional-concentration-2-density}, $i + 1 =e(U'')-1\le (v(U'')-2)\cdot (q+1) / 2$.  Hence $p^{i+1} \ge \frac{n^{\beta(i+1)}}{n^{v(U')-2}}$.
Substituting this into the inequality above, we find that
$$d_{\cF_{2,e},w}(U) \le \frac{p^{i+1}}{n^{\beta}}\cdot q^q \le (p\cdot n^{-\beta/6}) \cdot p^i\cdot n^{-\beta/2} \le K_{2,e}\cdot p^i\cdot n^{-\beta/2}$$
as desired where we used that $n$ is large enough and $K_{2,e} \ge p\cdot n^{-\beta/6}$ by definition of $K_{2,e}$. 
\end{proofclaim}


Recall that \ref{fractional-concentration-weight} holds by construction, so it suffices to show that \ref{fractional-concentration-degree} holds asymptotically almost surely.    
Since $n$ is large enough and there are at most $n^2$ edges of $K_n$, it suffices by the union bound to show that for each $e\in (K_n\setminus S)\cup S'$, we have with probability at least $1-n^{-\log n}$ that $\psi_p(e) = 1\pm n^{-\beta/7}$ if $e \in (K_n\setminus S)_p \cup S'$.

First suppose $e\in S'$. 
It follows from Claim~\ref{cl:codegF} and Corollary~\ref{cor:KimVu-poly} applied to $\cF_{0,e}$ with $n^2$ playing the role of $n$ and $\beta / 4$ playing the role of $\beta$, that with probability at least $1-n^{-4\log n}$, we have
$$w(\cF_{0,e}') = w(\cF_{0,e})\cdot p^k \cdot \left(1\pm n^{-\beta/6}\right).$$ 
By definition, we have that $\psi_p(e) = \left.w(\cF'_{0,e})\middle/p^{\binom{q}{2}-1}\right.$, and 
by \eqref{fractional-concentrated-psi-e-in-S'}, we have $w(\cF_{0,e}) = \psi(e) = 1$.
Hence, $\psi_p(e) = 1\pm n^{-\beta/6}$ as desired. 

Now let $e \in K_n\setminus S$.
By definition, we have for $e \in (K_n \setminus S)_p$ that $\psi_{i,p}(e) = \left.w(\cF_{i,e}')\middle/p^{\binom{q}{2}-1}\right.$ for $i \in \{1,2\}$.  Moreover, $w(\cF_{i,e}) = \psi_i(e)$ for $i \in \{1,2\}$.  Hence, via a similar argument as above using Claim~\ref{cl:codegF1} instead of Claim~\ref{cl:codegF}, with probability at least $1 - n^{-4\log n}$ we have
\COMMENT{It follows from Claim~\ref{cl:codegF1} and Corollary~\ref{cor:KimVu-poly} applied to $\cF_{1,e}$ with $n^2$ playing the role of $n$ and $\beta / 4$ playing the role of $\beta$, that with probability at least $1-n^{-4\log n}$, we have
$$w(\cF_{1,e}') = w(\cF_{1,e})\cdot p^k \cdot \left(1\pm n^{-\beta/6}\right).$$
By definition, we have for $e\in (K_n\setminus S)_p$ that $\psi_{1,p}(e) = \left.w(\cF_{1,e}')\middle / {p^{\binom{q}{2}-1}}\right.$, and recall that $w(\cF_{1,e})=\psi_1(e)$.}
\begin{equation}\label{psi_1-concentration}
    \psi_{1,p}(e) = \psi_{1}(e)\cdot \left(1\pm n^{-\beta/6}\right) \text{ for } e \in (K_n\setminus S)_p.
\end{equation}
We also claim that with probability at least $1 - n^{-4\log n}$,
\begin{equation}\label{psi_2-concentration}
    \psi_{2,p}(e) = \frac{1}{p}\cdot \psi_{2}(e)\cdot \left(1\pm n^{-\beta/6}\right) \pm 2\cdot n^{-\beta/6} \text{ for } e \in (K_n\setminus S)_p.
\end{equation}
Indeed, if $\psi_2(e) \geq K_{2,e}$, then $\psi_{2}(e) = K_{2,e}$ and it follows from Claim~\ref{cl:codegF2} and Corollary~\ref{cor:KimVu-poly} that with probability at least $1-n^{-4\log n}$, we have
$$w(\cF_{2,e}') = w(\cF_{2,e})\cdot p^{k-1} \cdot \left(1\pm n^{-\beta/6}\right)$$
and $\psi_{2,p}(e) = \frac{1}{p}\cdot \psi_2(e) \cdot \left(1\pm n^{-\beta/6}\right)$ for $e \in (K_n\setminus S)_p$, as claimed.  
Otherwise, $\psi_2(e) < K_{2,e}$ and hence $K_{2,e} = p\cdot n^{-\beta/6}$, and
it follows from Claim~\ref{cl:codegF2} and  Corollary~\ref{cor:KimVu-poly} that with probability at least $1-n^{-4\log n}$, we have
$$w(\cF_{2,e}') \le p^{k-1} \cdot 2\cdot K_{2,e} = 2\cdot p^{k}\cdot n^{-\beta/6}$$
and $\psi_{2,p}(e) \leq \frac{1}{p}\cdot\psi_2(e) \pm 2n^{-\beta / 6}$ for $e \in (K_n\setminus S)_p$\COMMENT{using that $\frac{1}{p}\psi_2(e) \leq {K_{2,e}}/{p} = n^{-\beta / 6}$ and $w(\cF_{2,e}')\ge 0$}, as claimed.

Combining \eqref{fractional-concentrated-psi-e-notin-S'}, \eqref{psi_1-concentration}, and \eqref{psi_2-concentration}, we find that with probability at least $1-n^{-\log n}$, for $e\in (K_n\setminus S)_p$,
\begin{align*}
    \psi_p(e) &= \psi_{1,p}(e)+\psi_{2,p}(e)\\
    &= \psi_{1}(e)\cdot \left(1\pm n^{-\beta/6}\right) + \frac{1}{p}\cdot \psi_{2}(e)\cdot \left(1\pm n^{-\beta/6}\right) \pm 2\cdot n^{-\beta/6} = 1\pm 3\cdot n^{-\beta/6}=1\pm n^{-\beta/7}
\end{align*}
as desired since $n$ is large enough so that $3\cdot n^{-\beta/6} \le n^{-\beta/7}$.
\end{proof}

Finally to prove Theorem~\ref{thm:RegularCliques}, we turn the $K_q$-weighting of Lemma~\ref{lem:YusterFractionalConcentrated} into a very regular collection of cliques by reweighting so the weights of clique lie between $1/4$ and $3/4$ and then sampling according to the weights (as was done in~\cite{GKLO16}). The proof is essentially the same as the first half of the Boost Lemma from~\cite{GKLO16} but we include it for completeness.

\begin{proof}[Proof of Theorem~\ref{thm:RegularCliques}]
Let $\psi$ be a $K_q$-weighting of $G:=(K_n\setminus S)_p\cup S'$ as in Lemma~\ref{lem:YusterFractionalConcentrated}. Let $D:= p^{\binom{q}{2}-1}\cdot \binom{n}{q-2}$. Since $q\geq 3$ and $(q - 2)/(\binom{q}{2}-1) = 2/(q + 1)$, we have 
$$D\ge n^{\beta\cdot \left(\binom{q}{2}-1\right)}/q^q \ge n^{\beta}.$$
We choose each clique $Q$ independently with probability $\psi(Q)\cdot \frac{D}{2}$. Recall that $\psi(Q) =\frac{1\pm 1/2}{D}$ and hence this probability is in $1/2 \pm 1/4$. Let $\mc{D}$ be the set of chosen cliques. Note that $\mc{D}$ is also a subhypergraph of ${\rm Design}(G,K_q)$.

For each edge $e$ of $G$, the expected number of elements of $\mc{D}$ containing $e$ is
$$\Expect{|\mc{D}(e)|} = \frac{D}{2}\cdot \psi(e) = \frac{D}{2}\cdot \left(1~\pm~n^{-\beta/7}\right).$$
By the Chernoff bounds, we find that 
\begin{equation*}
    \Prob{ \left| |\mc{D}(e)| - \Expect{|\mc{D}(e)|}\right| \ge n^{-\beta/7}\cdot \frac{D}{2} \cdot \psi(e)} \le 2\cdot e^{-n^{-2\beta/7}\cdot \frac{D}{2} \cdot \psi(e)/ 3} \le 2\cdot e^{-n^{5\beta/7}/12} \le e^{-\log^2 n},
\end{equation*}
where we used that $n$ is large enough, $D \geq n^\beta$, and $\beta > 0$.
Since there are at most $\binom{n}{2}$ edges of $K_n$, we have by the union bound asymptotically almost surely that for all $e\in G$,
\begin{equation*}
    d_{\cD}(e) = |\mc{D}(e)| = \left(1\pm n^{-\beta/7}\right)\frac{D}{2}\cdot \psi(e) = \left(1\pm n^{-\beta/8}\right)\cdot \frac{D}{2}
\end{equation*}
as desired. 
\end{proof}

\section{The Random Reserve}\label{s:Random-Reserve}

The last ingredient for Theorems~\ref{thm:YusterRandomGeneral} and~\ref{thm:YusterRandomRegularGeneral} is a lemma about the existence of the reserve edge set $X$ needed in Step (1) of our proof outline. 
In order to apply Theorem~\ref{thm:NibbleReserves} in Step (4), we will need to reserve a random $X \subseteq G(n,p)$ in Step (1) with the property that \aas every edge $e$ outside of $X$ has many `reserve cliques' in $X$ (that is $Q \cong K_q$ such that $e\in Q$ and $Q\setminus e \subseteq X$) and in addition no edge in $X$ is used by too many reserve cliques (and also such that $\Delta(X)$ is small of course).
To that end, we use the following definition.

\begin{definition}[Reserve design hypergraph]\label{def:reserve-design}
Let $F$ be a hypergraph. If $G$ is a hypergraph and $A,B$ are disjoint subsets of $E(G)$, then the \emph{$F$-design reserve hypergraph of $G$ from $A$ to $B$}, denoted ${\rm Reserve}_F(G,A,B)$, 
is the bipartite hypergraph $\mathcal{D}=(A,B)$ with $V(\mathcal{D}):=A\cup B$ and $$E(\mathcal{D}) := \{S\subseteq A\cup B: S \text{ is isomorphic to } F \text{ and } |S\cap A|=1\}.$$
\end{definition}

We will use this random reserve lemma to find $X\subseteq G(n,p)$ and $\cG_2 \subseteq {\rm Reserve}_{K_q}(G(n,p), G(n,p)\setminus X, X)$ such that every edge in $X$ has degree at least $D / C$ and every edge not in $X$ has degree at most $D$ for some $D$ on the order of $p^{\binom{q}{2} - 1}\binom{n}{q - 2}$ and some constant $C$ depending only on $q$.
We will also use it to show $\Delta_2(\cG_1 \cup \cG_2) \leq \Delta_2({\rm Design}(G(n,p), K_q)) \leq D^{1 - \beta / (7q)}$ where $\cG_1 \subseteq {\rm Design}(G(n,p), K_q)$ is obtained from Theorem~\ref{thm:RegularCliques}.  Then we will be able to apply Theorem~\ref{thm:NibbleReserves}.
For the proof of Theorem~\ref{thm:YusterRandomRegularGeneral}, we will have a coupling of $G_{n,d}$ and $G(n, (1 - o(1))d/n)$ such that \aas $G(n, (1 - o(1))d/n) \subseteq G_{n,d}$, and we will use Theorem~\ref{thm:EasyPerfectMatching} to ``cover down'' $G(n, (1 - o(1))d/n) \setminus G_{n,d}$ with another set of ``reserve cliques''.  We need the random reserve lemma to control degrees of a similar auxiliary hypergraph for this step as well.

Here then is our random reserve lemma.
The proof of this lemma is a simple application of Corollary~\ref{cor:KimVu-poly}.  We note that it also follows from Lemma~\ref{lem:NumberOfEmbeddings}, and naturally there are similarities in the proof.

\begin{lem}\label{lem:RandomReserve}
For each integer $q\ge 3$ and real $\eps > 0$, there exists $\varepsilon' > 0$ such that the following holds for all $\beta > 0$: If $G$ is an $n$-vertex graph with $\delta(G)\ge (1-\varepsilon')n$  and $p\ge n^{-\frac{2}{q+1} +\beta}$, then asymptotically almost surely 
\begin{enumerate}[label=(\arabic*)]
    \item\label{RandomReserve:degree} every edge $e \in E(G)$ is in $(1 \pm \eps)\cdot p^{\binom{q}{2}-1}\cdot \binom{n}{q - 2}$ $q$-cliques of $G_p + e$ and 
    \item\label{RandomReserve:codegree} every distinct $e,f \in E(G)$ are in at most $p^{\binom{q}{2} - 1}\cdot \binom{n}{q -2}\cdot n^{-\beta / 7}$ $q$-cliques of $G_p + e + f$. 
\end{enumerate}
\end{lem}
\begin{proof}
We choose $\varepsilon' > 0$ small enough as needed throughout the proof, and we assume $n$ is large enough as needed throughout the proof since the outcome need only hold asymptotically almost surely. 
We will apply Corollary~\ref{cor:KimVu-poly} (with $n^2$ playing the role of $n$) to the following hypergraphs. 
First, let $\cH$ be the set of $q$-cliques of $G$, and let $k \coloneqq \binom{q}{2} - 1$.  

\begin{itemize}
    \item For $e\in E(G)$, let $\cF_{e}$ be the hypergraph with $V(\cF_{e}):= E(G)\setminus \{e\}$ and $E(\cF_{e}):= \{E(H)\setminus \{e\}: H\in \mc{H}(e)\}$.
    Let $\cF_{e}' := \cF_{e}[V(\cF_{e})_p]$, and let $K_{e}:=\binom{n}{q-2}$.
    Note that $\cF_{e}$ is $k$-uniform.
    \item For distinct $e,f \in K_n$, let $\cF_{e,f}$ be the hypergraph with $V(\cF_{e,f}) \coloneqq E(G)\setminus \{e,f\}$ and $E(\cF_{e,f}) = \{Z \setminus \{f\} : Z \in \cF_{e}(f)\}$.
    Let $\cF_{e,f}':=\cF_{e,f}[V(\cF_{e,f})_p]$, and let $K_{e,f} :=p\cdot K_{e}$. 
    Note that $\cF_{e,f}$ is $(k-1)$-uniform. 
\end{itemize}

Note that since $\delta(G) \geq (1 - \eps')n$ and $e(\cF_e) = |\cH(e)|$, we have
\begin{equation}\label{eqn:random-reserve:K-bounds}
    K_e = \binom{n}{q - 2} \geq e(\cF_{e}) \geq \frac{((1 - \eps' \cdot q)n)^{q - 2}}{(q-2)!} \geq \left(1 - \frac{\eps}{2}\right)\cdot \binom{n}{q - 2}
\end{equation}
for every $e \in E(G)$.
In order to apply Corollary~\ref{cor:KimVu-poly}, we need the following claim.

\begin{claim}\label{cl:F1bounds}
Let $e \in E(G)$.  For all $i\in [k]$, we have $\Delta_i(\cF_{e}) \le K_{e} \cdot p^i \cdot n^{-\beta/2}$, and for all $f \in E(G)\setminus\{e\}$, we have $\Delta_{i-1}(\cF_{e,f}) \leq K_{e,f}\cdot p^{i-1}\cdot n^{-\beta/2}$.
\end{claim}
\begin{proofclaim}
Let $i\in [k]$, and let $U\subseteq G\setminus \{e\}$ with $|U|=i$. It suffices to show that $|\cF_{e}(U)| \le p^i\cdot K_{e}\cdot n^{-\beta/2}$.

Let $U':= U\cup \{e\}$. Notice that $|\cF_{e}(U)| = |\mc{H}(V(U'))|\le n^{q-v(U')}$. Yet $i=e(U')-1 \le (v(U')-2)\cdot m_2(K_q) = (v(U')-2)\cdot \frac{q+1}{2}$. Hence
$$p^i \ge \frac{n^{\beta\cdot i}}{n^{v(U')-2}} \ge n^{\beta} \cdot \frac{n^{q-v(U')}}{n^{q-2}} \ge n^{\beta}\cdot \frac{|\cF_{e}(U)|}{q^q\cdot K_{e}} \ge \frac{n^{\beta/2}\cdot |\cF_{e}(U)|}{K_{e}}$$
as desired where we used that $n^{\beta/2} \ge q^q$ since $n$ is large enough.
Finally, for $f \in E(G)\setminus \{e\}$, we have
\begin{equation*}
    \Delta_{i-1}(\cF_{e,f}) \leq \Delta_i(\cF_{e}) \leq K_{e}\cdot p^{i}\cdot n^{-\beta/2} = K_{e,f}\cdot p^{i-1}\cdot n^{-\beta/2},
\end{equation*}
as claimed.
\end{proofclaim}
\begin{claim}\label{cl:Random-Reserve:degree}
\ref{RandomReserve:degree} holds asymptotically almost surely.
\end{claim}
\begin{proofclaim} 
Since $n$ is large enough and there are at most $n^2$ edges of $K_n$, it suffices by the union bound to show that for each $e\in E(G)$, we have with probability at least $1-n^{-\log n}$ that $e(\cF_{e}') = (1 \pm \eps)p^{\binom{q}{2} - 1}\binom{n}{q-2}$.
By \eqref{eqn:random-reserve:K-bounds}, we have $\max\{\sqrt{K_e / e(\cF_e)}, K_e / e(\cF_e)\} \leq 2$.  Hence, it follows from Claim~\ref{cl:F1bounds} and Corollary~\ref{cor:KimVu-poly} applied to $\cF_{e}$ with $n^2$ playing the role of $n$ and $\beta / 4$ playing the role of $\beta$ (where each edge has weight $1$) that with probability at least $1-n^{-4\log n}$, we have
\begin{equation*}
    e(\cF'_e) = e(\cF)\cdot p^k\left(1 \pm 2n^{-\frac{\beta}{6}}\right).
\end{equation*}
By \eqref{eqn:random-reserve:K-bounds}, have $e(\cF_e) = (1 \pm \eps/2)\binom{n}{q - 2}$, and we thus have $e(\cF'_e) = (1 \pm \eps)p^{\binom{q}{2}-1}\binom{n}{q - 2}$ with probability at least $1 - n^{-4\log n}$, as desired.
\end{proofclaim}

\begin{claim}\label{cl:Random-Reserve:codegree}
\ref{RandomReserve:codegree} holds asymptotically almost surely.
\end{claim}
\begin{proofclaim}
Since $n$ is large enough and there are at most $n^4$ pairs of edges of $G$, it suffices by the union bound to show that for distinct $e,f\in E(G)$, we have with probability at least $1-n^{-\log n}$ that $e(\cF_{e,f}') \le p^{\binom{q}{2} - 1}n^{q - 2}n^{-\beta/7}$.
By Claim~\ref{cl:F1bounds} and Corollary~\ref{cor:KimVu-poly} applied to $\cF_{e,f}$ with $n^2$ playing the role of $n$ and $\beta/4$ playing the role of $\beta$ (where each edge has weight $1$),  with probability at least $1-n^{-4\log n}$ we have that
\begin{align*}
e(\cF_{e,f}') &\le e(\cF_{e,f})\cdot p^{k-1} \left(1+n^{-\beta/6}\cdot \frac{K_{e,f}}{e(\cF_{e,f})}\right)=p^{k-1}\cdot (e(\cF_{e,f})+ K_{e,f}\cdot n^{-\beta/6}) \\
&\le p^k\cdot K_{e}\cdot n^{-\beta/7}= p^{\binom{q}{2}-1}\cdot \binom{n}{q -2} \cdot n^{-\beta/7}
\end{align*}
as desired where we used that $e(\cF_{e,f}) = \Delta_0(\cF_{e,f})\le K_{e,f}\cdot n^{-\beta/2}$ from Claim~\ref{cl:F1bounds} and $K_{e,f} = p \cdot K_{e}$ and $n^{-\beta/2}+n^{-\beta/6}\le n^{-\beta/7}$ since $n$ is large enough.
\end{proofclaim}

\noindent The lemma now follows from Claims~\ref{cl:Random-Reserve:degree} and \ref{cl:Random-Reserve:codegree}.
\end{proof}

\section{Proofs of the Main Theorems}\label{s:main-proof}

In this section, we prove our main theorems, namely Theorems~\ref{thm:YusterRandomGeneral} and~\ref{thm:YusterRandomRegularGeneral}.
First we prove the following lemma, which combines the Omni-Absorber Theorem, Theorem~\ref{thm:RandomOmniAbsorber}, the Divisibility 
Fixer Lemma, Lemma~\ref{lem:DivFixer}, the Regularity Boosting Theorem, Theorem~\ref{thm:RegularCliques}, and the Random Reserve Lemma, Lemma~\ref{lem:RandomReserve}, to show that for $p$ above the threshold in Theorem~\ref{thm:GeneralQAbsorberThreshold}, asymptotically almost surely $G \sim G(n,p)$ contains all of the ingredients we need to execute the refined absorption approach.
In this lemma, we also consider a coupling of $G$ and $G' \sim G(n,(1 +o(1))p)$ such that \aas $G \subseteq G'$.  We need this coupling for the proof of Theorem~\ref{thm:YusterRandomRegularGeneral}, in which we use known cases of the Sandwich Conjecture (Conjecture~\ref{conj:SandwichConj}).  For Theorem~\ref{thm:YusterRandomGeneral}, we do not need this coupling and it suffices to take $G' = G$.  Similarly, we only need \ref{what-appears-in-Gnp:reserves2} below for the proof of Theorem~\ref{thm:YusterRandomRegularGeneral}, and we only need \ref{what-appears-in-Gnp:div-fixer} below for the proof of Theorem~\ref{thm:YusterRandomGeneral}.

\begin{lem}\label{lem:what-appears-in-Gnp}
  For each integer $q \geq 3$, there exist a real $C \geq 1$ and $\gamma > 0$ such that the following holds for all $\beta > 0$. Let $a \coloneqq \max\{d_{\rm abs}(K_q), \frac{q + 1}{2}\}$.  
  If $p \geq n^{-\frac{1}{a} + \beta}$, $p \leq p' \leq (1 + \gamma)p$, and we have a coupling of $G \sim G(n,p)$ and $G' \sim G(n,p')$ such that \aas $G \subseteq G'$, then asymptotically almost surely there exist 
  $F, X, X', A \subseteq G$, 
  hypergraphs $\cG_1 \subseteq {\rm Design}(G \setminus (X' \cup A), K_q)$ and
  $\cG_2 \subseteq {\rm Reserve}_{K_q}(G \setminus (X' \cup A) \cup X, G \setminus (X' \cup A), X)$, and 
  $D,D' \geq n^\beta$ such that the following holds.
    \begin{enumerate}[label=(\ref*{lem:what-appears-in-Gnp}.\arabic*), left=0pt]
    \item\label{what-appears-in-Gnp:div-fixer} $F \subseteq X' \setminus X$ is a $K_q$-divisibility fixer such that $v(F) = n$ and $e(F) \leq (q - 2)n + C$. \label{what-appears-in-Gnp:first}
    \item\label{what-appears-in-Gnp:omni-absorber} $A$ is a $K_q$-omni-absorber for $X'$. 
    \item\label{what-appears-in-Gnp:boosted-design-hypergraph} Every vertex of $\cG_1$ has degree at most $D$ and at least $D\left(1 - D^{-\beta/(17q)}\right)$.
    \item\label{what-appears-in-Gnp:reserves} $X\subseteq X'$, and $\cG_2$ satisfies
      \begin{itemize}
        \item every edge of $G \setminus (X' \cup A)$ has degree at least $D / C$ in $\cG_2$ and
        \item every edge of $X$ has degree at most $D$ in $\cG_2$.
        \end{itemize}
    \item\label{what-appears-in-Gnp:codegrees} $\Delta_2(\cG_1 \cup \cG_2) \leq D^{1 - \beta/(17q)}$.
    \item\label{what-appears-in-Gnp:reserves2} $\cG \coloneqq {\rm Reserve}_{K_q}((G' \setminus G) \cup (X' \setminus X), G' \setminus G, X' \setminus X)$ satisfies
    \begin{itemize}
        \item every edge of $X'\setminus X$ has degree at most $D'$ in $\cG$ and 
        \item every edge of $G' \setminus G$ has degree at least $8\binom{q}{2}D'$ in $\cG$.
    \end{itemize}
    \label{what-appears-in-Gnp:last}
\end{enumerate}
\end{lem}
\begin{proof}
Let $C_1$ be as in Lemma~\ref{lem:DivFixer} for $q$. Let $\varepsilon_2$ be as $\varepsilon'$ in Lemma~\ref{lem:RandomReserve} for $q$ with $1/2$ playing the role of $\eps$. Let $C_2$ be the maximum of ${1}/{\varepsilon_2}$ and the value of $C$ in Theorem~\ref{thm:RegularCliques} for $q$. Let $C_3$ be as in Theorem~\ref{thm:RandomOmniAbsorber} for $q$.
Let $C \coloneqq \max\{C_1, 8\cdot 2^{\binom{q}{2} + 2}\cdot (16\cdot C_2 \cdot C_3)^{\binom{q}{2}-1}\}$, and we let $\gamma > 0$ be sufficiently small as needed throughout the proof. 

Let
$$p_1:= p_2:= \frac{p}{16\cdot C_2\cdot C_3},~~p_3:= \frac{p}{4\cdot C_2},~~p_4:=\frac{p-p_1-p_2(1-p_1)-p_3(1-p_1)(1-p_2)}{(1-p_1)(1-p_2)(1-p_3)}.$$
Note that $p_1,p_2\le \frac{p}{16}$ and $p_3 \le \frac{p}{4}$. Hence $p_1+p_2+p_3 \le \frac{p}{2}$ and it follows that $p_4 \ge p-(p_1+p_2+p_3) \ge \frac{p}{2}$. On the other hand, we see that $p_4 \le p$. 
Note also that
\begin{equation}\label{eqn:pi-lower-bound}
    p_i \geq n^{-\frac{1}{a} + \frac{\beta}{2}} \text{ for each } i \in [4].
\end{equation}
Now we let
$$G_1 := G(n,p_1),~~G_2:= (K_n\setminus G_1)_{p_2},~~G_3:= (K_n\setminus (G_1\cup G_2))_{p_3},~~G_4:= (K_n\setminus (G_1\cup G_2\cup G_3))_{p_4}.$$
By the definition of $p_4$\COMMENT{
    For each $e\in K_n$, choose $X_e \in [0, 1]$ independently and uniformly at random, and let $e \in G_i$ for $i \in [4]$ if $X_e \in [x_{i-1}, x_i)$ for $x_i$ as follows.  Let $x_0 = 0$ and $x_1 = p_1$, so $G_1 \sim G(n, p_1)$, as required.  We choose $x_2 \coloneqq p_2(1 - p_1) + p_1$ so that 
    \begin{equation*}
        p_2 = \ProbCond{e\in G_2}{e\notin G_1} = \frac{x_2 - x_1}{1 - x_1},
    \end{equation*}
    so $G_2 \sim (K_n \setminus G_1)_{p_2}$, as required.
    Note that $1 - x_2 = (1 - p_1)(1 - p_2)$.
    We choose $x_3 \coloneqq p_3(1 - p_1)(1 - p_2) + p_2(1 - p_1) + p_1$ so that 
    \begin{equation*}
        p_3 = \ProbCond{e\in G_3}{e\notin G_1\cup G_2} = \frac{x_3 - x_2}{1 - x_2},
    \end{equation*}
    so $G_3 \sim (K_n \setminus (G_1 \cup G_2))_{p_3}$, as required.
    Note that $1 - x_3 = (1 - p_1)(1 - p_2) - p_3(1 - p_1)(1 - p_2) = (1 - p_1)(1 - p_2)(1 - p_3)$.
    We choose $x_4 \coloneqq p$ so that 
    \begin{equation*}
        p_4 = \ProbCond{e\in G_4}{e\notin G_1\cup G_2 \cup G_3} = \frac{x_4 - x_3}{1 - x_3},
    \end{equation*}
    so $G_4 \sim (K_n \setminus (G_1 \cup G_2 \cup G_3))_{p_4}$, as required.  Since $x_4 = p$, ...
}, we can couple $G'_i \sim G(n,p_i)$ for $i \in [4]$ and $G \sim G(n,p)$ such that 
$$G_i = G'_i \setminus \bigcup_{j<i}G_j \text{ for $i \in [4]$}\qquad\text{and}\qquad G = G'_1\cup G'_2 \cup G'_3\cup G'_4$$ and $G'_i$ and $G'_j$ are independent for distinct $i,j \in [4]$.

\begin{claim}
  Asymptotically almost surely, there exist $F$, $A$, and $\cG_1$ such that the following holds.
    \begin{enumerate}[label=(\alph*)]
    \item\label{what-appears-claim:bounded-max-degree} For all $i\in [4]$, we have $\Delta(G'_i) \leq 2p_in$.\label{what-appears-claim:first}
    \item\label{what-appears-claim:reserves-degree-upper} Every edge $e \in K_n$ is in at most $(1 + \gamma)^{\binom{q}{2}}\cdot p^{\binom{q}{2}-1}\cdot\binom{n}{q-2}$ $q$-cliques of $G' + e$ and in at least $(1 - \gamma)\cdot p^{\binom{q}{2}-1}\cdot \binom{n}{q - 2}$ $q$-cliques of $G + e$.
    In particular, $\Delta_1({\rm Design}(G', K_q)) \leq 2p^{\binom{q}{2}-1}\cdot\binom{n}{q-2}$.
    \item\label{what-appears-claim:reserves-codegree} Every distinct $e,f \in K_n$ are in at most $p^{\binom{q}{2} - 1}\cdot\binom{n}{q - 2}\cdot n^{-\beta / 7}$ $q$-cliques of $G + e + f$.
    In particular, 
    \begin{equation*}
        \Delta_2({\rm Design}(G, K_q)) \leq p^{\binom{q}{2}-1}\cdot\binom{n}{q-2}\cdot n^{-\frac{\beta}{7}}.
    \end{equation*}
    \item\label{what-appears-claim:reserves-degree-lower2} Every edge $e \in K_n$ is in at least $\frac{1}{2}\cdot (p_1)^{\binom{q}{2}-1}\cdot \binom{n}{q - 2}$ $q$-cliques of $G_1 + e$.
    \item\label{what-appears-claim:div-fixer} $F \subseteq G_1$ is a $K_q$-divisibility fixer such that $v(F) = n$, and $e(F) \leq (q - 2)n + C_1$. 
    \item\label{what-appears-claim:reserves-degree-lower} Every edge $e \in K_n \setminus G_1$ is in at least $\frac{1}{2}\cdot (p_2)^{\binom{q}{2}-1}\cdot \binom{n}{q - 2}$ $q$-cliques of $G_2 + e$.
    \item\label{what-appears-claim:omni-absorber} $A \subseteq G_3$ is a $K_q$-omni-absorber for $G_1 \cup G_2$.
    \item\label{what-appears-claim:boosted-design-hypergraph} $\cG_1 \subseteq {\rm Design}(G_3 \cup G_4 \setminus A, K_q)$ is $(1/2 \pm n^{-\beta / 16})\cdot (p_4)^{\binom{q}{2}-1}\cdot \binom{n}{q - 2}$-regular.
    \label{what-appears-claim:last}
\end{enumerate}  
\end{claim}
\begin{proofclaim}
  By the Chernoff bounds, $G(n, p')$ \aas has maximum degree at most $2p'n$ for all $p' \geq \log^2 n / n$, say. Therefore, \aas \ref{what-appears-claim:bounded-max-degree} holds.
  By Lemma~\ref{lem:RandomReserve} applied with both $p$ and $p'$ playing the role of $p$, with $\gamma$ playing the role of $\eps$, and with $K_n$ playing the role of $G$, \aas \ref{what-appears-claim:reserves-degree-upper} and \ref{what-appears-claim:reserves-codegree} hold. 
  By \eqref{eqn:pi-lower-bound} and Lemma~\ref{lem:RandomReserve} (with $\beta / 2$, $p_1$, and $1/2$ playing the roles of $\beta$,  $p$, and $\eps$, respectively, and $K_n$ playing the role of $G$) \aas \ref{what-appears-claim:reserves-degree-lower2} holds.
  We show the remaining events hold conditional on \ref{what-appears-claim:bounded-max-degree}--\ref{what-appears-claim:reserves-degree-lower2}.
  
  By \eqref{eqn:pi-lower-bound} and Lemma~\ref{lem:DivFixer} (with $\beta/2$ and $p_1$ playing the roles of $\beta$ and $p$, respectively), $G_1$ asymptotically almost surely  contains a $K_q$-divisibility fixer $F$ such that $v(F) = n$, and $e(F) \le (q-2)n + C_1$, so \aas \ref{what-appears-claim:div-fixer} holds.

  Since $\Delta(G_1) \leq \Delta(G'_1) \leq 2p_1 n$ (assuming \ref{what-appears-claim:bounded-max-degree}), we have that 
$$\delta(G_2) \ge n-\Delta(G_1) \ge (1-2p_1)\cdot n \ge (1-\varepsilon_2)\cdot n,$$
where we used that $2p_1\le \varepsilon_2$ which follows by the definition of $p_1$ since $C_2 \ge {1}/{\varepsilon_2}$. 
Thus, \aas \ref{what-appears-claim:reserves-degree-lower} holds by \eqref{eqn:pi-lower-bound} and Lemma~\ref{lem:RandomReserve} (with $\beta / 2$ and $p_2$ playing the roles of $\beta$ and $p$, respectively, and $K_n \setminus G_1$ playing the role of $G$).  

Note that assuming \ref{what-appears-claim:bounded-max-degree},
$$\Delta(G_1 \cup G_2)\le \Delta(G'_1)+\Delta(G'_2) \le 2(p_1+p_2)\cdot n = \frac{p_3\cdot n}{C_3}$$
by definition of $p_1$, $p_2$ and $p_3$. Hence by \eqref{eqn:pi-lower-bound} and Theorem~\ref{thm:RandomOmniAbsorber}, \aas there exists a $C_3$-refined $K_q$-omni-absorber $A$ for $G_1 \cup G_2$ such that $A \subseteq G_3$, and thus \ref{what-appears-claim:omni-absorber} holds.

Let $S:= G_1 \cup G_2 \cup G_3$, and let $S':= G_3\setminus A$.  Note that 
$$\Delta(S) \le \Delta(G'_1) + \Delta(G'_2) + \Delta(G'_3) \leq 2(p_1+p_2+p_3)\cdot n\le \frac{pn}{C_2}$$
since $C_3\ge 1$. 
Hence by \eqref{eqn:pi-lower-bound} and Theorem~\ref{thm:RegularCliques} (with $\beta / 2$ and $p_4$ playing the roles of $\beta$ and $p$, respectively), \aas there exists a subhypergraph $\mc{G}_1$ of ${\rm Design}(G_3 \cup G_4 \setminus A,K_q)$ that is $\left(1/2\pm n^{-\beta/16}\right)\cdot (p_4)^{\binom{q}{2}-1}\cdot~\binom{n}{q-2}$-regular, so \ref{what-appears-claim:boosted-design-hypergraph} holds.
\end{proofclaim}

Let $$D:= \left(\frac{1}{2} + n^{-\frac{\beta}{16}}\right)\cdot (p_4)^{\binom{q}{2}-1}\cdot~\binom{n}{q-2}\qquad\text{and}\qquad D' \coloneqq 2\sqrt{\gamma}\cdot p^{\binom{q}{2}-1}\cdot\binom{n}{q-2},$$
and note that since $p_4 \ge \frac{p}{2} \ge \frac{1}{2}\cdot n^{-\frac{2}{q+1} + \beta}$, we have that $D \ge D' \geq \gamma \cdot n^{\beta \cdot \left(\binom{q}{2}-1\right)}\ge n^{\beta},$ as required.
Let
\begin{equation*}
  X \coloneqq G_2 \qquad\text{and}\qquad X' \coloneqq G_1 \cup G_2.
\end{equation*}

Now we show that \ref{what-appears-in-Gnp:first}--\ref{what-appears-in-Gnp:last} hold assuming \ref{what-appears-claim:first}--\ref{what-appears-claim:last} hold and $G \subseteq G'$.
\ref{what-appears-in-Gnp:div-fixer} and \ref{what-appears-in-Gnp:omni-absorber} follow from \ref{what-appears-claim:div-fixer} and \ref{what-appears-claim:omni-absorber} immediately.

Now we show \ref{what-appears-in-Gnp:boosted-design-hypergraph}.
By \ref{what-appears-claim:boosted-design-hypergraph}, every vertex of $\cG_1$ has degree at most $D$, as required.
Similarly every vertex of $\mc{G}_1$ has degree at least 
\begin{equation*}
    \left(1-n^{-\frac{\beta}{16}}\right)\cdot (p_4)^{\binom{q}{2}-1}\cdot~\binom{n}{q-2} = D - 2n^{-\frac{\beta}{16}}\cdot (p_4)^{\binom{q}{2}-1}\cdot \binom{n}{q - 2} \geq D - 2n^{-\frac{\beta}{16}}\cdot D \geq D\left(1 - D^{-\frac{\beta}{17q}}\right),
\end{equation*}
where in the last inequality we used that $D \leq n^q$ and $2n^{-\beta / 16} \leq n^{-\beta / 17}$.

Now we show \ref{what-appears-in-Gnp:reserves}. 
Since ${\rm Reserve}_{K_q}(G \setminus (X' \cup A) \cup X, G \setminus (X' \cup A), X) \subseteq {\rm Design}_{K_q}(G', K_q)$, by \ref{what-appears-claim:reserves-degree-upper}, every edge of $X$ has degree at most $2p^{\binom{q}{2}-1}\cdot\binom{n}{q - 2}$ in ${\rm Reserve}_{K_q}(G \setminus (X' \cup A) \cup X, G \setminus (X' \cup A), X)$.  Since $p_4 \geq p/2$, every edge of $X$ has degree at most $2^{\binom{q}{2}+2}D$ in ${\rm Reserve}_{K_q}(G \setminus (X' \cup A) \cup X, G \setminus (X' \cup A), X)$.  Moreover, since every edge of ${\rm Reserve}_{K_q}(G \setminus (X' \cup A) \cup X, G \setminus (X' \cup A), X)$ contains one element of $G \setminus (X' \cup A)$,  by \ref{what-appears-claim:reserves-degree-lower}, every edge of $G \setminus (X' \cup A)$ has degree at least $\frac{1}{2}\cdot (p_2)^{\binom{q}{2}-1}\cdot\binom{n}{q - 2}$ in ${\rm Reserve}_{K_q}(G \setminus (X' \cup A) \cup X, G \setminus (X' \cup A), X)$.  Since $p_2 \geq p_4 / (16 \cdot C_2 \cdot C_3)$, every edge of $G \setminus (X' \cup A)$ has degree at least $D / (2\cdot (16\cdot C_2\cdot C_3)^{\binom{q}{2}-1})$ in ${\rm Reserve}_{K_q}(G \setminus (X' \cup A) \cup X, G \setminus (X' \cup A), X)$.  By a standard application of the Chernoff bounds and the union bound, by considering a random subhypergraph of ${\rm Reserve}_{K_q}(G \setminus (X' \cup A) \cup X, G \setminus (X' \cup A), X)$ obtained by including each edge independently with probability $1 / 2^{\binom{q}{2} + 3}$, there exists a subhypergraph $\cG_2 \subseteq {\rm Reserve}_{K_q}(G \setminus (X' \cup A) \cup X, G \setminus (X' \cup A), X)$ in which every edge of $X$ has degree at least $D / (4\cdot 2^{\binom{q}{2} + 3}\cdot (16\cdot C_2 \cdot C_3)^{\binom{q}{2} - 1}) \geq D / C$ and every edge of $G \setminus (X' \cup A)$ has degree at most $D$\COMMENT{
Indeed, the expected degree of an edge of $X$ is at least $D / (2\cdot 2^{\binom{q}{2} + 3}\cdot (16\cdot C_2 \cdot C_3)^{\binom{q}{2} - 1}) \geq D / C \geq n^{\beta}/C$ and by the Chernoff bounds, the probability that the degree of an edge of $X$ is less than half this quantity is at most $\exp(-n^{\beta}/(8C))$.  Similarly, the expected degree of an edge of $G \setminus (X' \cup A)$ is at most $D / 2$, and by the Chernoff bounds, the probability that the degree of an edge of $G \setminus (X' \cup A)$ is at most $D$ is at most $\exp(-D / 4) \leq \exp(-n^{\beta} / 4)$.  Since there are at most $n^2$ total edges, it follows from the union bound.
}, as desired.

Now we show \ref{what-appears-in-Gnp:codegrees}.
Since $\cG_1 \cup \cG_2 \subseteq {\rm Design}_{K_q}(G, K_q)$, by \ref{what-appears-claim:reserves-codegree}, we have $\Delta_2(\cG_1 \cup \cG_2) \leq p^{\binom{q}{2} - 1}\cdot \binom{n}{q - 2} \cdot n^{-\beta/7}$.  Since $p_4 \geq p / 2$ and $D \leq n^q$, we have
\begin{equation*}
    \Delta_2(\cG_1 \cup \cG_2) \leq p^{\binom{q}{2} - 1}\cdot \binom{n}{q - 2} \cdot n^{-\frac{\beta}{7}} \leq 2^{\binom{q}{2}+1}\cdot  D \cdot n^{-\frac{\beta}{7}} \leq D^{1 - \frac{\beta}{17q}},
\end{equation*}
as desired.

Finally we show \ref{what-appears-in-Gnp:reserves2}.  
Since every $e \in X' \setminus X$ satisfies\COMMENT{Note that $\cG(\{e\})$ is the set of $S \cong K_q$ containing $e$ with one edge in $G' \setminus G$ and the rest in $X' \setminus X$, while ${\rm Design}(G', K_q)(\{e\})$ is the set of $S \cong K_q$ containing $e$ such that $S\subseteq G'$ and ${\rm Design}(G, K_q)(\{e\})$ is the set of $S \cong K_q$ containing $e$ such that $S\subseteq G$, so (assuming $G \subseteq G'$ and thus $X' \subseteq G'$) every such $S \in \cG(\{e\})$ is in the former but not the latter.}
\begin{equation*}
    \cG(\{e\}) \subseteq {\rm Design}(G', K_q)(\{e\}) \setminus {\rm Design}(G, K_q)(\{e\}),
\end{equation*}
by \ref{what-appears-claim:reserves-degree-upper}, since ${\rm Design}(G, K_q)(\{e\}) \subseteq {\rm Design}(G', K_q)(\{e\})$  (assuming $G \subseteq G'$), we have 
\begin{equation*}
    d_{\cG}(e) \leq (1 + \gamma)^{\binom{q}{2}}\cdot p^{\binom{q}{2}-1}\cdot \binom{n}{q-2} - (1 - \gamma)\cdot p^{\binom{q}{2}-1}\cdot \binom{n}{q - 2} \leq D',
\end{equation*}
as desired.
Since every edge of $\cG$ contains one element of $G' \setminus G$, by \ref{what-appears-claim:reserves-degree-lower2}, every $e \in G' \setminus G$ satisfies
\begin{equation*}
    d_{\cG}(e) \geq \frac{1}{2}\cdot (p_1)^{\binom{q}{2}-1}\cdot\binom{n}{q - 2} \geq \gamma^{\frac{1}{3}}\cdot p^{\binom{q}{2}-1}\cdot\binom{n}{q - 2} \geq 8\binom{q}{2}D',
\end{equation*}
as desired.
\end{proof}

\subsection{Proof of $G(n,p)$ Theorem}

We are now prepared to prove Theorem~\ref{thm:YusterRandomGeneral}.

\begin{proof}[Proof of Theorem~\ref{thm:YusterRandomGeneral}]
  Let $G \sim G(n,p)$, and let $C$ be as in Lemma~\ref{lem:what-appears-in-Gnp}.  By Lemma~\ref{lem:what-appears-in-Gnp} (with $G$ playing the role of both $G$ and $G'$), asymptotically almost surely, there exist $F, X, X', A \subseteq G$, hypergraphs $\cG_1 \subseteq {\rm Design}(G \setminus (X' \cup A), K_q)$ and $\cG_2 \subseteq {\rm Reserve}_{K_q}(G\setminus (X' \cup A) \cup X, G \setminus (X' \cup A), X)$, and $D \geq n^\beta$ satisfying \ref{what-appears-in-Gnp:div-fixer}--\ref{what-appears-in-Gnp:codegrees}.  We prove that $G$ has a $K_q$-packing containing all but at most $(q - 2)n + C$ edges of $G$ assuming these events hold.

  Let $D_0, \alpha > 0$ be $D_\beta$ and $\alpha$ as in Theorem~\ref{thm:NibbleReserves} with $\beta / (17q)$ and $\binom{q}{2}$ playing the roles of $\beta$ and $r$, respectively.  We may assume $n$ is large enough so that $D \geq n^\beta \geq D_0$ and $D / C \geq D^{1 - \alpha}$.  Thus, by \ref{what-appears-in-Gnp:boosted-design-hypergraph}, \ref{what-appears-in-Gnp:reserves}, and \ref{what-appears-in-Gnp:codegrees}, and by Theorem~\ref{thm:NibbleReserves} applied to $\cG_1 \cup \cG_2$, there exists a $(G\setminus (X'\cup A))$-perfect matching of $\mc{G}_1\cup \mc{G}_2$, which is a $K_q$-packing $\mathcal{Q}_1$ of $G \setminus (X' \cup A) \cup X$ such that $G \setminus (X' \cup A) \subseteq \bigcup \mc{Q}_1$. 

By \ref{what-appears-in-Gnp:div-fixer}, since $F$ is a $K_q$-divisibility fixer we have by Proposition~\ref{prop:DivFixer} that there exists $F'\subseteq F$ such that $G\setminus F'$ is $K_q$-divisible and $e(F') \leq e(F) \leq (q - 2)n + C$.  Note that $A$ is also $K_q$-divisible since it admits a $K_q$-decomposition (since the empty graph is a $K_q$-divisible subgraph of $X'$ and $A$ is an omni-absorber for $X'$ by \ref{what-appears-in-Gnp:omni-absorber}).

Let $Q_1:= \bigcup \mc{Q}_1$, and let $L := X'\setminus (F' \cup  Q_1)$.  Note that $G \setminus F'$ is the disjoint union of $Q_1$, $A$, and $L$\COMMENT{$A$ and $X'$ are edge-disjoint by the definition of omni-absorber}.  Since $G-F'$, $Q_1$, and $A$ are all $K_q$-divisible, so is $L$.

Since $L\subseteq X'$, we have by \ref{what-appears-in-Gnp:omni-absorber} and the definition of omni-absorber that there exists a $K_q$-decomposition $\mathcal{Q}_2$ of $L\cup A$. But then $\mathcal{Q}_1\cup \mathcal{Q}_2$ is a $K_q$-decomposition of $G-F'$ and hence is a $K_q$-packing of $G$ covering all but $e(F') \le e(F) \le (q-2)n+C$ edges as desired.
\end{proof}

\subsection{Proof for Random Regular Graphs}

To prove Theorem~\ref{thm:YusterRandomRegularGeneral}, we will need some recent results on the Sandwich Conjecture of Kim and Vu~\cite{KV04}.

\begin{conj}[Sandwich Conjecture]\label{conj:SandwichConj}
For every $d\gg \log n$, there are $p_{*} = (1-o(1))\frac{d}{n}, p^{*} = (1+o(1))\frac{d}{n}$ and a coupling $(G_{*},G,G^{*})$ such that $G_{*} \sim G(n,p_{*}),G^{*} \sim G(n,p^{*}), G\sim G_{n,d}$ and $\Prob{G_{*}\subseteq G\subseteq G^{*} }=1-o(1)$.
\end{conj}

Gao, Isaev, and McKay~\cite{GIM20} show that the Sandwich Conjecture holds when both $d, n-d \gg \log^4 n / \log^3 \log n$. In a separate paper, the same authors~\cite{GIM22} show the Sandwich Conjecture holds for all $d$ above that range. Combining these two results yields the following.

\begin{thm}[Combination of results from~\cite{GIM20} and~\cite{GIM22}]\label{thm:SandwichResult} 
Sandwich Conjecture holds for all $d \gg \log^4 n / \log^3 \log n$.
\end{thm}

We are now prepared to prove Theorem~\ref{thm:YusterRandomRegularGeneral}. The proof is nearly identical to that of Theorem~\ref{thm:YusterRandomGeneral} except that instead of using the divisibility fixer from \ref{what-appears-in-Gnp:div-fixer}, we use $\cG = {\rm Reserve}_{K_q}((G\setminus G_*) \cup (X'\setminus X), G\setminus G_*, X'\setminus G)$ to decompose $G \setminus G_*$. We note this requires applying Lemma~\ref{lem:what-appears-in-Gnp} with $G^*$ playing the role of $G'$ and hence requires both sides of the sandwich.

\begin{proof}[Proof of Theorem~\ref{thm:YusterRandomRegularGeneral}]
It suffices to prove the result assuming $d \geq n^{1 - \frac{1}{a} + 2\beta}$, so let $G \sim G_{n,d}$ where $d \geq n^{1 - \frac{1}{a} + 2\beta}$ and $q - 1 \mid d$ and $q \mid d\cdot n$.  Note that $G$ is $K_q$-divisible.
Let $C \geq 1$ and $\gamma > 0$ be as in Lemma~\ref{lem:what-appears-in-Gnp}.  By Theorem~\ref{thm:SandwichResult}, there exist $p_*, p^* \in [0, 1]$ such that $p_* \geq n^{-\frac{1}{a} + \beta}$ and $p^* \leq (1 + \gamma)p_*$ and a coupling $(G_*, G, G^*)$, such that $G_* \sim G(n,p_*)$ and $G^* \sim G(n,p^*)$ and \aas $G_* \subseteq G \subseteq G^*$.

By Lemma~\ref{lem:what-appears-in-Gnp} (with $G_*$ and $G^*$ playing the role of $G$ and $G'$, respectively), asymptotically almost surely, there exist $X, X', A \subseteq G_*$, hypergraphs $\cG_1 \subseteq {\rm Design}(G_* \setminus (X' \cup A), K_q)$, $\cG_2 \subseteq {\rm Reserve}_{K_q}(G_*\setminus (X' \cup A) \cup X, G_* \setminus (X' \cup A), X)$, and $\cG \coloneqq {\rm Reserve}_{K_q}((G^* \setminus G_*) \cup (X' \setminus X), G^* \setminus G_*, X' \setminus X)$, and $D,D' \geq n^\beta$ satisfying \ref{what-appears-in-Gnp:omni-absorber}--\ref{what-appears-in-Gnp:reserves2}.  We prove that $G$ has a $K_q$-decomposition assuming these events hold and $G_* \subseteq G \subseteq G^*$.

Let $D_0, \alpha > 0$ be $D_\beta$ and $\alpha$ as in Theorem~\ref{thm:NibbleReserves} with $\beta / (17q)$ and $\binom{q}{2}$ playing the roles of $\beta$ and $r$, respectively.  We may assume $n$ is large enough so that $D \geq n^\beta \geq D_0$ and $D / C \geq D^{1 - \alpha}$.  Thus, by \ref{what-appears-in-Gnp:boosted-design-hypergraph}, \ref{what-appears-in-Gnp:reserves}, and \ref{what-appears-in-Gnp:codegrees}, and by Theorem~\ref{thm:NibbleReserves} applied to $\cG_1 \cup \cG_2$, there exists a $(G_*\setminus (X'\cup A))$-perfect matching of $\mc{G}_1\cup \mc{G}_2$, which is a $K_q$-packing $\mathcal{Q}_1$ of $G_* \setminus (X' \cup A) \cup X$ such that $G_* \setminus (X' \cup A) \subseteq \bigcup \mc{Q}_1$. 

Let $\cG'\coloneqq \cG[(G\setminus G_*) \cup (X'\setminus X)]$.  Since every edge of $\cG$ contains one element of $G^* \setminus G_*$, by \ref{what-appears-in-Gnp:reserves2}, every edge of $G\setminus G_*$ has degree at least $8\binom{q}{2}D'$ in $\cG'$.  Since $\cG' \subseteq \cG$, by \ref{what-appears-in-Gnp:reserves2}, every edge of $X' \setminus X$ has degree at most $D$ in $\cG'$.
Hence by Theorem~\ref{thm:EasyPerfectMatching}, there exists a $(G\setminus G_*)$-perfect matching of $\mc{G'}$ which is a $K_q$-packing $\mathcal{Q}_2$ of $(G\setminus G_*)\cup (X \setminus X')$ such that $G \setminus G_* \subseteq \bigcup \mc{Q}_2$.

Let $Q_1:= \bigcup \mc{Q}_1$ and $Q_2:= \bigcup \mc{Q}_2$. Note that $Q_1$ and $Q_2$ are edge-disjoint. Let $L := X'\setminus (Q_1\cup Q_2)$. Since $\mathcal{Q}_1$ is a $K_q$-packing of $G_*\setminus (X' \cup A) \cup X$ and $\cQ_2$ is a $K_q$-packing of $G \setminus G_* \cup X' \setminus X$, we have that $G$ is the disjoint union of $Q_1$, $Q_2$, $A$, and $L$.  Since $G$ is $K_q$-divisible and $Q_1$, $Q_2$, and $A$ admit $K_q$-decompositions, we find that $L$ is $K_q$-divisible. 

Since $L\subseteq X'$, we have by the definition of $K_q$-omni-absorber that there exists a $K_q$-decomposition $\mathcal{Q}_3$ of $L\cup A$. But then $\mathcal{Q}_1\cup \mathcal{Q}_2\cup \mc{Q}_3$ is a $K_q$-decomposition of $G$ as desired.
\end{proof}

\subsection{Proofs of Fractional Decomposition Theorems}

We now prove Theorem~\ref{thm:YusterRandomFractionalGeneral}. The proof is essentially the same as that of Theorem~\ref{thm:YusterRandomGeneral} except that we do not use a divisibility fixer; however, without a divisibility fixer, we no longer guarantee that $L$ is $K_q$-divisible. However, we will be  able to `fractionally absorb' any $L\subseteq X$ into $A$ regardless of whether $L$ is $K_q$-divisible. To that end, we make the following definition.

\begin{definition}[Fractional $K^r_q$-omni-absorber]
Let $q > r \geq 1$ be integers, and let $X$ be an $r$-uniform hypergraph. We say a hypergraph $A$ is a \textit{fractional $K^r_q$-omni-absorber} for $X$ with \emph{decomposition family} $\mathcal{H}$ and \emph{decomposition function} $\psi_A$ if $V(X)=V(A)$, $X$ and $A$ are edge-disjoint, $\mathcal{H}$ is a family of subgraphs of $X\cup A$ each isomorphic to $K^r_q$ such that $|E(H)\cap E(X)|\le 1$ for all $H\in\mathcal{H}$, and for every $\phi:X\rightarrow [0,1]$, there exists $\psi_A(\phi):\mathcal{H}\rightarrow [0,1]$ such that $\psi_A(\phi)(e)=\phi(e)$ for $e\in X$ and $\psi_A(\phi)(e)=1$ for $e\in A$. 
\end{definition}

If $X$ is a graph and $m$ is a positive integer, we let $m*X$ denote the multigraph obtained by replacing each edge of $X$ with $m$ parallel edges. The following proposition shows a $K_q$-omni-absorber for $q(q-1)*X$ is also a fractional $K_q$-omni-absorber for $X$. 

\begin{proposition}
Let $q\ge 3$ be an integer, and let $X$ be a graph. If $A$ is a $K_q$-omni-absorber for $q(q-1)*X$, then $A$ is a fractional $K_q$-omni-absorber for $X$.
\end{proposition}
\begin{proof}
Let $A$ be a $K_q$-omni-absorber for $q(q - 1)*X$ with decomposition family $\cH$.
Let $m:= e(X)$. Let $\phi:X\rightarrow [0,1]$. Let $e_1,\ldots, e_m$ be an ordering of $X$ such that $\phi(e_1)\le \phi(e_2) \le \ldots \le \phi(e_m)$. For $i\in [m]$, let $L_i:= X\setminus \{e_j:j\in [i-1]\}$. Let $w_1=\frac{\phi(e_1)}{q(q-1)}$; for $i\in [m]\setminus \{1\}$, let $w_i := \frac{\phi(e_i)-\phi(e_{i-1})}{q(q-1)}$; let $w_{m+1} := 1-\sum_{i=1}^m w_i = 1 -\frac{\phi(e_m)}{q(q-1)}$.  Note that $\sum_{i=1}^{m+1}w_i = 1$ and that for all $i\in [m]$, we have $\sum_{j\in [i]} w_j = \frac{\phi(e_i)}{q(q-1)}$.

For each $i\in [m]$, $q(q-1)*L_i$ is $K_q$-divisible. Since $A$ is a $K_q$-omni-absorber for $q(q-1)*X$, we have that for each $i\in [m]$, there exists a $K_q$-decomposition $\cH_i \subseteq \mc{H}$ of $(q(q-1)*L_i) \cup A$. Similarly, there exists a $K_q$-decomposition $\cH_{m+1}$ of $A$. For every $H \in \cH$, we let
\begin{equation*}
    \psi_{A}(\phi)(H) := \sum_{i : H \in \cH_i} w_i.
\end{equation*}
Now for $i\in [m]$, we have 
$$\psi_{A}(\phi)(e_i) = \sum_{H\in \mc{H}: e_i\in H} \psi_{A}(\phi)(H) = \sum_{j\in [m+1]}~\sum_{H\in \mc{H}_j: e_i\in H} w_j = q(q-1)\cdot \sum_{j\in [i]} w_j = \phi(e_i),$$ 
where we used the definition of $\psi_A(\phi)(H)$, that $e_i$ is in exactly $q(q-1)$ $H\in \mc{H}_j$ for all $j\in [i]$ and in exactly zero $H\in \mc{H}_j$ for $j\in [m+1]\setminus [i]$, and that $\sum_{j\in [i]} w_j = \frac{\phi(e_i)}{q(q-1)}$ by the definition of the $w_j$. Meanwhile for $e\in A$, we have that 
$$\psi_{A}(\phi)(e) = \sum_{H\in \mc{H}: e\in H} \psi_{A}(\phi)(H) = \sum_{i\in [m+1]}~\sum_{H\in \mc{H}_i: e\in H} w_i = \sum_{i\in [m+1]} w_i = 1,$$
where we used the definition of $\psi_A(\phi)(H)$, that $e$ is in exactly one $H\in \mc{H}_i$ for all $i\in [m+1]$, and that $\sum_{i\in [m+1]} w_i=1$.
Hence $\psi_A$ is as desired.
\end{proof}

We also note that in fact the main proof from~\cite{DPI} yields a simple $K_q$-omni-absorber $A$ for $q(q-1)*X$ and hence our Theorem~\ref{thm:RandomOmniAbsorber} would yield a simple $K_q$-omni-absorber for $q(q-1)*X$ as well.

Given the above, the proof of Theorem~\ref{thm:YusterRandomGeneral} is easily modified to prove Theorem~\ref{thm:YusterRandomFractionalGeneral}: namely, we leave out the divisibility fixer altogether but then the subgraph $L$ of $X$ is not necessarily $K_q$-divisible; nevertheless since $A$ is a fractional $K_q$-omni-absorber for $X$, we find that $A\cup L$ admits a fractional $K_q$-decomposition as desired. 

\section{Sparse Absorber Constructions}\label{s:Construction}

In this section, we prove Theorems~\ref{thm:GeneralTriangleAbsorberThreshold} and~\ref{thm:GeneralQAbsorberThreshold}. Namely we introduce new sparser $K_q$-absorber constructions.




We need to recall the notion of a transformer from~\cite{BKLO16}.

\begin{definition}
Let $q\ge 3$ be an integer. Let $L$ and $L'$ be $K_q$-divisible graphs. A \emph{$K_q$-transformer} $T$ from $L$ to $L'$ is a graph such that $V(L)\cup V(L')\subseteq V(T)$ is an independent set in $T$ and both $T\cup L$ and $T\cup L'$ admit $K_q$-decompositions.     
\end{definition}

First we show in the case $q=3$ that `star' transformers with maximum rooted $2$-density at most $3+\varepsilon$ exist for any $\varepsilon > 0$ as follow.

\begin{lem}\label{lem:StarTransformerTriangle}
Let $q=3$. Let $L,L' \cong K_{1,q-1}$ such that $L\cup L' \cong K_{2,q-1}$. For every real $\varepsilon > 0$, there exists a $K_q$-transformer $T$ from $L$ to $L'$ such that $m_2(T,V(L)\cup V(L')) \le q +\varepsilon$.
\end{lem}
\begin{proof}
Let $k$ be an even integer such that $k\ge \frac{1}{\varepsilon}$. Let $V(L)\cap V(L') = \{v_0,v_{k+1}\}$. Let $V(L)\setminus V(L') = \{x\}$ and $V(L')\setminus V(L) = \{x'\}$. We now construct $T$ as follows. Let $V(T)\setminus (V(L)\cup V(L')) := \{v_i : i\in [k] \}$. We let
$$E(T) := \{ v_{i-1}v_i : i\in [k+1]\} \cup \{ xv_i, x'v_i: i\in [k] \},$$
or informally $T$ consists of a path $v_0v_1\ldots v_kv_{k+1}$ whose internal vertices are complete to $x$ and $x'$. We note that $V(L)\cup V(L')$ is an independent set in $T$ and that $T$ is a $K_q$-transformer from $L$ to $L'$. Since $T$ has degeneracy at most $3$, every $H' \subseteq T$ satisfies $e(H') \leq 3v(H')-6$, so $m_2(T) \le 3$. 

Now let $H'\subseteq T$ such that $V(H')\setminus (V(L)\cup V(L'))\ne \emptyset$. Let $m:= |V(H')\setminus (V(L)\cup V(L'))|$. First suppose that $m < k$. It follows that the degeneracy of $H'$ rooted at $V(L)\cup V(L')$ is at most $3$. Hence
$$\frac{e(H')}{m} \le \frac{3m}{m} \le 3.$$
So suppose that $m=k$. But then
$$\frac{e(H')}{m} \le \frac{e(T)}{k} \le \frac{3k+1}{k} \le 3 + \frac{1}{k} \le 3 + \varepsilon.$$
Thus, we find that $m(T,V(L)\cup V(L'))\le 3+\varepsilon$. Since $m_2(T)\le 3$, this implies that $m_2(T,V(L)\cup V(L'))\le 3+\varepsilon$ as desired.
\end{proof}

More generally, we show such `star' transformers exist with rooted $2$-density at most $q+\frac{1}{2}$ as follows. We believe that iterating this construction (which would correspond to the construction above in the $q=3$ case) would yield a rooted $2$-density at most $q+\varepsilon$ for any $\varepsilon > 0$, but we opted not to pursue this for brevity.

\begin{lem}\label{lem:StarTransformerGeneral}
Let $q\ge 3$ be an integer. Let $L,L' \cong K_{1,q-1}$ such that $L\cup L' \cong K_{2,q-1}$. Then there exists a $K_q$-transformer $T$ from $L$ to $L'$ such that $m_2(T,V(L)\cup V(L')) \le q +\frac{1}{2}$.
\end{lem}
\begin{proof}
Let $V(L)\cap V(L') = \{v_{1,j}:j\in [q-1]\}$. Let $V(L)\setminus V(L') = \{x\}$ and $V(L')\setminus V(L) = \{x'\}$. We now construct $T$ as follows. Let $V(T)\setminus (V(L)\cup V(L')) := \{v_{i,j} : i\in \{2,\ldots, q-1\},~j\in[q-1] \}$. We let
\begin{align*}E(T) := &\{ v_{i,j}v_{i,j'} : i\in \{2,\ldots,q-1\},~j<j'\in [q-1]\} \cup \{ v_{i,j}v_{i',j} : i < i'\in [q-1],~j\in [q-1]\}\\
&\cup \{ xv_{i,j}, x'v_{i,j}: i\in \{2,\ldots,q-1\},~j\in [q-1] \},
\end{align*}
or more informally $T$ consists of a $(q-1)\times (q-1)$ grid with vertices $(v_{i,j}:i,j\in [q-1])$ whose columns are cliques and whose all but first row are cliques and whose vertices not in the first row are complete to $x$ and $x'$. We note that $V(L)\cup V(L')$ is an independent set in $T$ and that $T$ is a $K_q$-transformer from $L$ to $L'$. 

We note that $T\setminus (V(L)\cup V(L'))$ has maximum degree at most $2q-5$ and that every vertex of $T\setminus (V(L)\cup V(L'))$ has exactly three neighbors in $V(L)\cup V(L')$. 

Now let $H'\subseteq T$ such that $V(H')\setminus (V(L)\cup V(L'))\ne \emptyset$. Let $m:= |V(H')\setminus (V(L)\cup V(L'))|$. It follows that
$$\frac{e(H')}{m} \le \frac{\frac{(2q-5)m}{2} + 3m}{m} = q+\frac{1}{2}.$$
Thus $m(T,V(L)\cup V(L')) \le q +\frac{1}{2}$.

Now let $H''\subseteq T$ such that $v(H'')\ge 3$. We want to show that $\frac{e(H'')-1}{v(H'')-2}\le q + \frac{1}{2}$. Let $m:= |V(H'')\setminus (V(L)\cup V(L'))|$ and $m':= v(H'')-m$. If $m=0$, then $e(H'')=0$ and hence $\frac{e(H'')-1}{v(H'')-2} \le 0$ as desired. So we assume that $m\ne 0$. Next suppose that $m' \ge 2$. It follows as above that
$$\frac{e(H'')-1}{v(H'')-2} \le \frac{e(H'')}{m} \le q+\frac{1}{2}$$
as desired. 

So we assume that $m'\le 1$. Next suppose that $v(H'') - 2 < q$. But then
$$\frac{e(H'')-1}{v(H'')-2} \le \frac{\binom{v(H'')}{2}-1}{v(H'')-2} = \frac{v(H'')+1}{2} \le \frac{q+2}{2} \le q$$
as desired since $q\ge 2$. So we assume that $v(H'')\ge q+2$. Next suppose $m'= 1$. But then
$$\frac{e(H'')-1}{v(H'')-2} \le \frac{\frac{(2q-5)m}{2}+m}{m-1} = \left(q-\frac{3}{2}\right)\cdot \left(1+\frac{1}{m-1}\right) \le \left(q-\frac{3}{2}\right)\cdot \left(1+\frac{1}{q}\right) \le q$$
as desired. Finally we assume that $m'= 0$. But then
$$\frac{e(H'')-1}{v(H'')-2} \le \frac{\frac{(2q-5)m}{2}}{m-2} = \left(q-\frac{5}{2}\right)\cdot \left(1+\frac{2}{m-2}\right) \le \left(q-\frac{5}{2}\right)\cdot \left(1+\frac{2}{q}\right) \le q$$
as desired. Thus, we find that $m_2(T)\le q+\frac{1}{2}$. Since $m(T,V(L)\cup V(L'))\le q+\frac{1}{2}$, this implies that $m_2(T,V(L)\cup V(L'))\le q+\frac{1}{2}$ as desired.
\end{proof}

Following~\cite{GKLO16} and~\cite{BGKLMO20}, we write $\nabla_q~X$ for the graph obtained from a graph $X$ by replacing every edge $e$ of $X$ with ${\rm AntiEdge}_q(e)$ (with all new vertices distinct).
Similarly, we write $\tilde{\nabla}_q~X$ for $\nabla_q~X \cup X$.

Using the star transformers from above, we show we can build sparser absorbers for a clique of anti-edges $\nabla_q~K_q$. This is accomplished by building a transformer from $(\nabla_q)^2~K_q$ to itself as follows. 

\begin{lem}\label{lem:SparseAntiCliqueAbsorber}
Let $q\ge 3$ be an integer and let $L:=\nabla_q~K_q$. If $q=3$, then for every real $\varepsilon > 0$, there exists a $K_q$-absorber $A$ for $L$ such that $m_2(A,V(L))\le 3+\varepsilon$. If $q\ge 4$, then there exists a $K_q$-absorber $A$ for $L$ such that $m_2(A,V(L))\le q+\frac{1}{2}$.
\end{lem}
\begin{proof}
Let $S_1:= L = \nabla_q~K_q$, and let $S_2:= \nabla_q~S_1$. We let $A_1:= S_2$. We now construct a graph $A_2$ that is vertex-disjoint from $A_1$ as follows: let $S'_0\cong K_q$, $S'_1:= \nabla_q~S'_0$, $S'_2 := \nabla_q~S'_1$, and let $A_2 := \bigcup_{i=0}^{2}~S'_i$.
Note there exists the natural bijection $\phi$ from $S_2$ to $S'_2$.

Now for each edge $e\in S_{2}$, we add a new $K_{q-2}$ labeled $Q_e$ (and whose vertices are labeled $q_{e,i}$ for $i\in [q-2]$) where the vertices of $Q_e$ are complete to the ends of $e$ and $\phi(e)$. We let $A_3$ be the graph with all these new edges and their ends.

Note that $\nabla_q$ preserves $K_q$-divisibility, and hence $S'_0,S'_1,S'_2,S_1,S_2$ are $K_q$-divisible. Hence for every $v\in V(S_2)$, we have that $d_{S_2}(v)$ is divisible by $q-1$ and similarly for $S'_{2}$. Thus for each $v\in V(S_2)$, there exists a partition of the edges of $S_{2}$ incident with $v$ into sets $M_{v,1},\ldots, M_{v,j_v}$ of size $q-1$. Note that applying $\phi$ to these sets also yields a partition of the edges of $S_2'$ incident with $\phi(v)$.

Let $T$ be as in Lemma~\ref{lem:StarTransformerGeneral} for $q$ if $q\ge 4$, and as in Lemma~\ref{lem:StarTransformerTriangle} for $q$ and $\varepsilon$ if $q=3$. Finally for each $v\in V(S_2)$, $j\in [j_v]$ and $i\in [q-2]$, we attach a copy of $T$ rooted at $\{ q_{e,i}: e\in M_{v,j}\}\cup \{v,\phi(v)\}$.
We let $A_4$ be the graph with all these new edges and their ends.

We let $A := \bigcup_{i=1}^4 A_i$. We note that $A$ is a $K_q$-absorber for $L$. 

Now we show $m_2(A, V(L))$ is at most $q + \frac{1}{2}$ if $q \geq 4$ and at most $3 + \eps$ if $q = 3$.  We will use Lemma~\ref{lem:2DensityConcatenate} in conjunction with the following series of claims.

\begin{claim}\label{cl:S'_0andS'_1}
    $m_2(S'_0 \cup S'_1, V(L)) \leq q$.
\end{claim}
\begin{proofclaim}
    Since $S'_0 \cong K_q$ and there are no edges with ends in both $S'_0$ and $L$, we have $m_2(S'_0, V(L)) = m_2(S'_0) = (q + 1)/2 \leq q$.
    Since there are no edges with ends in both $S'_1$ and $L$, we have $m_2(S'_1, V(S'_0 \cup L)) = m_2(S'_1, V(S'_0)) = (q + 1)/2 \leq q$ by Proposition~\ref{prop:FakeEdge} and Lemma~\ref{lem:2DensityConcatenate}.
    Therefore, the claim follows from Lemma~\ref{lem:2DensityConcatenate} with $S'_0 \cup V(L)$ and $S'_1 \cup V(S'_0 \cup L)$ playing the roles of $H_1$ and $H_2$, respectively.
\end{proofclaim}

\begin{claim}\label{cl:S_2andS'_2}
    $m_2(S_2 \cup S'_2, V(L \cup S'_1)) \leq q$.  
\end{claim}
\begin{proofclaim}
    Since there are no edges with ends in both $S_2$ and $S'_1$, we have $m_2(S_2, V(S'_1 \cup L)) = (q + 1)/2 \leq q$ by Proposition~\ref{prop:FakeEdge} and Lemma~\ref{lem:2DensityConcatenate}.
    Similarly, since there are no edges with ends in both $S'_2$ and $S_2$, we have $m_2(S'_2, V(S_2 \cup S'_1 \cup L)) \leq (q + 1)/2 \leq q$ by Proposition~\ref{prop:FakeEdge} and Lemma~\ref{lem:2DensityConcatenate}.
    Therefore, the claim follows from Lemma~\ref{lem:2DensityConcatenate} with $S_2 \cup V(S'_1)$ and $S'_2 \cup V(S_2 \cup S'_1)$ playing the roles of $H_1$ and $H_2$, respectively.
\end{proofclaim}

\begin{claim}\label{cl:A_3}
    If $q \geq 4$, then $m_2(A_3, V(S_2 \cup S'_2)) \leq q + \frac{1}{2}$.
\end{claim}
\begin{proofclaim}
    By the construction of $A_3$ and Lemma~\ref{lem:2DensityConcatenate}, it suffices to prove that if $Z\cong K_{q+2}$ and $Y\subseteq Z$ with $Y\cong K_4$, then $m_2(Z-E(Y),V(Y))\le q+\frac{1}{2}$. 

    First we show $m(Z-E(Y),V(Y))\le q+\frac{1}{2}$. To that end, let $H'\subseteq Z$ where $V(H')\setminus V(Y)\ne \emptyset$. Let $m:= |V(H')\setminus V(Y)|$. But then 
$$\frac{e(H')}{m} \le \frac{4m+\binom{m}{2}}{m} = \frac{m+7}{2} \le \frac{q+5}{2} \le q+\frac{1}{2}$$
as desired since $q\ge 4$. This proves that $m(Z-E(Y),V(Y))\le q+\frac{1}{2}$.

Thus it remains to show for the case $q\ge 4$ that $m_2(Z)\le q+\frac{1}{2}$. To that end, let $H''\subseteq Z$ where $v(H'')\ge 3$. But then 
$$\frac{e(H'')-1}{v(H'')-2} \le \frac{\binom{v(H'')}{2}-1}{v(H'')-2} = \frac{v(H'')+1}{2} \le \frac{q+3}{2}$$
since $q\ge 3$.
\end{proofclaim}

We only use Claim~\ref{cl:S_2andS'_2} when $q = 4$ because we do not have $m(A_3, V(A_1 \cup A_2)) \leq 3 + \eps$ for $q = 3$ (the value is 4).  Instead, we use a more sophisticated discharging argument to bound $m_2(A_3 \cup S_2 \cup S'_2, V(S'_1 \cup S_1))$, as follows.

\begin{claim}\label{cl:rooted-densityA_3andS_2andS'_2}
    If $q = 3$, then $m(A_3 \cup S_2 \cup S'_2, V(S'_1 \cup S_1)) \leq q$.
\end{claim}
\begin{proofclaim}
    Let $H'\subseteq A_3\cup S_2 \cup S_2'$ with $V(H')\setminus V(S'_1 \cup S_1) \neq \emptyset$.  Let $B_1 \coloneqq V(H') \cap V(S'_1 \cup S_1)$, let $B_2 \coloneqq V(H') \cap ((V(S_2) \cup V(S'_2)) \setminus B_1)$, let $B_3 \coloneqq V(H') \setminus (B_1 \cup B_2)$, and let $B \coloneqq B_2 \cup B_3$.  Let $m \coloneqq |B|$.  We will show that $e(H) / m \leq 3$.  

    To that end, we apply a discharging argument as follows. For each $v\in B_3$, let ${\rm ch}(v) := d_{H'}(v) - 3$. For each $v\in B_2$, let ${\rm ch}(v):= d_{H'\setminus B_3}(v) - 3$. 

    Note that since $q = 3$, we have $S_1 \cong C_6$ and $S_2, S'_2 \cong C_{12}$, and the sets $B_1$, $B_2$, and $B_3$ are independent in $H'$.
    Hence, $\sum_{v\in B} {\rm ch}(v) = e(H')-3m$.
    Moreover, $d_{H'\setminus B_3}(v)\le 2$ for all $v\in B_2$, so we find that ${\rm ch}(v) \le -1$ for all $v\in B_2$. Meanwhile, $d_{H'}(v)\le 4$ for all $v\in B_3$ and this equals $4$ only if $|N_{H'}(v)\cap B_2|\ge 2$ since $v$ has two neighbors in each of $B_1$ and $B_2$. So we apply a discharging rule as follows.

\vskip.1in
\noindent {\bf Discharging Rule: Every vertex $v$ of $B_3$ with $d_{H'}(v)=4$ sends $+1/2$ charge to each vertex of $N_{H'}(v)\cap B_2$.}
\vskip.1in

Let ${\rm ch}_F$ denote the final charges after applying this discharging rule. We claim that ${\rm ch}_F(v)\le 0$ for all $v\in B_2\cup B_3$. First suppose $v\in B_3$. If $d_{H'}(v)\le 3$, then ${\rm ch}_F(v) = {\rm ch}(v)\le 0$ as desired. So we assume $d_{H'}(v)=4$. But then by the discharging rule, we have that ${\rm ch}_F(v) \le {\rm ch}(v) - 2 (1/2) \le 0$ as desired. So finally we assume that $v\in B_2$. But then $|N_{H'}(v)\cap B_3|\le 2$ and hence ${\rm ch}_F(v) \le -1 + 2(1/2) \le 0$ as desired.

Thus $\sum_{v\in B} {\rm ch}(v) = \sum_{v\in B} {\rm ch}_F(v) \leq 0$. Hence $e(H')\leq 3m$ and we have $e(H') / m \leq 3$ as desired.
Thus $m_2(A_3\cup S_2 \cup S_2', V(S'_1\cup S_1))\le 3$ as desired.
\end{proofclaim}
\begin{claim}\label{cl:2-densityA_3andS_2andS'_2}
    If $q = 3$, then $m_2(A_3 \cup S_2 \cup S'_2) \leq q$.
\end{claim}
\begin{proofclaim}
    As mentioned in the previous claim, in this case, $S_2, S'_2 \cong C_{12}$.  Note also that $A_3 \cup S_2 \cup S'_2$ is $4$-regular, so 
    \begin{equation*}
        \frac{e(A_3 \cup S_2 \cup S'_2) - 1}{v(A_3 \cup S_2 \cup S'_2) - 2} = \frac{71}{34} \leq 3,
    \end{equation*}
    as desired.  Moreover, every proper subgraph $H' \subsetneq A_3 \cup S_2 \cup S'_2$ is $3$-degenerate. Thus for all $H' \subsetneq A_3 \cup S_2 \cup S'_2$ with $v(H')\ge 3$, we have
    \begin{equation*}
        \frac{e(H') - 1}{v(H') - 2} \leq \frac{3\cdot v(H') - 6}{v(H') - 2} \leq 3,
    \end{equation*}
    as desired.
\end{proofclaim}

\begin{claim}\label{cl:A_4}
    $m_2(A_4, V(A_3 \cup A_2 \cup A_1))$ is at most $q + \frac{1}{2}$ if $q \geq 4$ and at most $q + \eps$ if $q = 3$.
\end{claim}
\begin{proofclaim}
    The claim follows immediately from Lemmas~\ref{lem:StarTransformerTriangle} and~\ref{lem:StarTransformerGeneral}.
\end{proofclaim}

Now the result follows by Lemma~\ref{lem:2DensityConcatenate} together with Claims~\ref{cl:S'_0andS'_1}, \ref{cl:S_2andS'_2}, \ref{cl:A_3}, and \ref{cl:A_4} when $q = 4$ and Claims~\ref{cl:S'_0andS'_1}, \ref{cl:rooted-densityA_3andS_2andS'_2}, \ref{cl:2-densityA_3andS_2andS'_2}, and \ref{cl:A_4} when $q = 3$.
\end{proof}

Next we show how to use $K_q$-absorbers for the anti-clique combined with the $\nabla_q$ operator to build a $K_q$-absorber $A_L$ for any $K_q$-divisible graph $L$ such that the rooted $2$-density of $A_L$ at $L$ is at most the maximum of the rooted $2$-densities of an anti-clique absorber and an anti-edge as follows.

\begin{lem}\label{lem:NablaConstruction}
Let $q\ge 3$ be an integer, and let $L$ be a $K_q$-divisible graph. Let $A$ be a $K_q$-absorber for $L$, and let $B$ be a $K_q$-absorber for $\nabla_q~K_q$. If $\mathcal{Q}_1$ is a $K_q$-decomposition of $A$ and $\mathcal{Q}_2$ is a $K_q$-decomposition of $L\cup A$, then
$$A':=\nabla_q~L \cup \nabla_q~A\cup \bigcup_{Q\in \mathcal{Q}_1\cup\mathcal{Q}_2} B_Q$$
is also a $K_q$-absorber for $L$,
where $B_Q$ is copy of $B$ rooted at $\nabla_q~Q$.
Furthermore, $m_2(A',V(L)) \le \max\left\{\frac{q+1}{2},~m_2(B,V(\nabla_q~K_q))\right\}$.
\end{lem}
\begin{proof}
Note that $V(L)$ is an independent set in $A'$. Let $\mathcal{B}_1$ be a $K_q$-decomposition of $B$ and let $\mathcal{B}_2$ be a $K_q$-decomposition of $B\cup \nabla_q~K_q$. Similarly for each $Q\in \mathcal{Q}_1\cup\mathcal{Q}_2$, we let $(\mathcal{B}_1)_Q$ be the $K_q$-decomposition of $B_Q$ corresponding to $\mathcal{B}_1$ and we let $(\mathcal{B}_2)_Q$ be the $K_q$-decomposition of $B_Q\cup \nabla_q~Q$ corresponding to $\mathcal{B}_2$. Now we let
$$\mathcal{Q}_1' := \bigcup_{Q\in \mathcal{Q}_2} (\mathcal{B}_2)_Q~\cup \bigcup_{Q\in \mathcal{Q}_1} (\mathcal{B}_1)_Q$$
and
$$\mathcal{Q}_2' := \{\tilde{\nabla}_q~e:e\in L\}~\cup \bigcup_{Q\in \mathcal{Q}_1} (\mathcal{B}_2)_Q~\cup \bigcup_{Q\in \mathcal{Q}_2} (\mathcal{B}_1)_Q.$$
Note that $\mathcal{Q}_1'$ is a $K_q$-decomposition of $A'$ and $\mathcal{Q}_2'$ is a $K_q$-decomposition of $L\cup A'$. Hence $A'$ is a $K_q$-absorber of $L$ as desired. Finally, the desired inequality is a direct consequence of Proposition~\ref{prop:FakeEdge}, Lemma~\ref{lem:2DensityConcatenate} and the construction above.
\end{proof}

We are now prepared to prove Theorems~\ref{thm:GeneralTriangleAbsorberThreshold} and~\ref{thm:GeneralQAbsorberThreshold}.

\begin{proof}[Proof of Theorems~\ref{thm:GeneralTriangleAbsorberThreshold} and~\ref{thm:GeneralQAbsorberThreshold}]
These follow immediately from the existence of $K_q$-absorbers  and Lemmas~\ref{lem:SparseAntiCliqueAbsorber} and~\ref{lem:NablaConstruction}.
\end{proof}

We note the same proof also yields $K_q$-absorbers of rooted degeneracy at most $2q-2$.

\section{Concluding Remarks and Further Directions}\label{s:Conclusion}

As mentioned in Section~\ref{s:Overview}, not only did Keevash~\cite{K14} and Glock, K\"uhn, Lo, and Osthus~\cite{GKLO16} prove the Existence Conjecture, they also proved ``typicality versions''.  A graph $G$ is \textit{$(\xi, h, p)$-typical} if for any set $A \subseteq V(G)$ with $|A| \leq h$, we have $|\bigcap_{v\in A} N(v)| = (1 \pm \xi)\cdot p^{|A|}\cdot v(G)$.  Keevash's results apply to $(\xi, h, p)$-typical graphs for $p = v(G)^{-\varepsilon}$ for some small unspecified $\varepsilon > 0$ and large $h \in \mathbb N$ depending on $q$. Random graphs are \aas typical, so it is natural then to wonder whether the stronger typicality versions hold for our range of $p$ and whether our proof extends to give these. One difficulty is in the Embedding Theorem (Theorem~\ref{thm:Embed}), namely in concentrating the expected number of copies of an absorber in a typical graph. This did not seem to us to follow from typicality. Nevertheless, some quasi-random version may hold where the quasi-randomness would be whatever is needed for this concentration property to hold. Indeed, it is conceivable to us that the typical versions do not hold for the small values of $p$ we consider.

Another natural direction is to consider $F$-packings with small leave of $G(n,p)$ and $F$-decompositions of $G_{n,d}$ for other graphs $F$, in particular cycles. Similarly, it would be interesting to study these problems for $K_q^r$ in random $r$-uniform hypergraphs and more generally for any fixed hypergraph $F$.

There is also a plethora of other interesting avenues of research in combining the various questions extending the Existence Conjecture.
The Nash-Williams Conjecture~\cite{N-W70} asserts that every $K_3$-divisible $n$-vertex graph with minimum degree at least $3n/4$ admits a $K_3$-decomposition, and it has been conjectured \cite{Gu91, BKLO16} that more generally a $K_q$-divisible $n$-vertex graph with minimum degree at least $(1 - 1/(q + 1))n$ admits a $K_q$-decomposition.  
Hence it is natural to consider the threshold for linear leave of $G_p$ where $G$ is a high-minimum-degree graph (so $G(n,p)$ is the case $G=K_n$), combining the Nash-Williams-type problems and Yuster's $G(n,p)$ problem. 
Similarly, it may be possible to combine our work here with our binomial random $q$-uniform result from~\cite{DKPIV}; namely, if one chooses both edges uniformly at random with some probability $p_1$ and cliques uniformly at random with some probability $p_2$, what conditions on $p_1$ and $p_2$ are needed to ensure a packing with linear leave of $G(n,p_1)$ whose cliques are from the binomial $q$-uniform hypergraph $G^{(q)}(n,p_2)$? Clearly some trade-off between the two parameters is necessary.
Finally, it would be interesting to consider high-girth designs -- the existence of which is proved in~\cite{DPII} -- in these settings.  Namely can the packing of $G(n,p)$ with linear leave, or the decomposition of $G_{n,d}$, also be made to be high girth?

\subsection{Improving the Values of $p$ and $d$}

Another direction is to improve the values of $p$ and $d$ themselves in Theorems~\ref{thm:YusterRandom} and \ref{thm:YusterRandomRegular}. However, it seems our values are quite close to the theoretical limit possible with only using constant-sized absorbers. In particular, we note the following lemma.


\begin{lem}\label{lem:AbsLowerBound}
Let $q\ge 3$ be an integer. Let $L$ be a $K_q$-divisible graph and $A$ be a $K_q$-absorber for $L$. If every vertex of $V(A)\setminus V(L)$ is incident with an edge of $A$ and there exist $K_q$ decompositions $\mathcal{B}_1$ of $A$ and $\mathcal{B}_2$ of $L\cup A$ such that $\mathcal{B}_1\cap \mathcal{B}_2 =\emptyset$, then every vertex in $V(A)\setminus V(L)$ has degree at least $2q-2$ in $A$. 
\end{lem}
\begin{proof}
Let $v\in V(A)\setminus V(L)$. Since $v$ is not isolated in $A$, $v$ is incident with an edge $e$ of $A$. Since $\mathcal{B}_1$ is a $K_q$-decomposition of $A$, there exists $Q_1\in \mathcal{B}_1$ such that $e\in Q_1$. Similarly since $\mathcal{B}_2$ is a $K_q$-decomposition of $L\cup A$, there exists $Q_2\in \mathcal{B}_2$ such that $e\in Q_2$. 

Since $\mathcal{B}_1\cap \mathcal{B}_2=\emptyset$, we find that $V(Q_1)\ne V(Q_2)$. It follows that there exists $f$ incident with $v$ such that $f\not\in Q_2$. Since $\mathcal{B}_2$ is a $K_q$-decomposition of $L\cup A$, there exists $Q_2'\in \mathcal{B}_2$ such that $f\in Q_2'$. Note that $f\in A$. Since $A$ is a simple graph and $\mathcal{B}_2$ is a $K_q$-decomposition of $L\cup A$, it follows that $E(Q_2)\cap E(Q_2')=\emptyset$ and hence $(V(Q_2)\setminus \{v\})\cap (V(Q_2')\setminus \{v\})=\emptyset$. Hence $v$ has degree at least $(q-1)+(q-1)=2q-2$ in $A$.         
\end{proof}

\begin{cor}\label{cor:AbsLowerBound}
For each integer $q\ge 3$, $d_{\rm abs}(K_q)\ge q-1$.    
\end{cor}
\begin{proof}
Let $L$ be a $K_q$-divisible graph without $K_q$ as a subgraph. Let $A''$ be a $K_q$-absorber for $L$. Let $A'$ be obtained from $A''$ by deleting isolated vertices in $V(A)\setminus V(L)$. By definition of $K_q$-absorber, there exists a $K_q$ decomposition $\mathcal{B}_1'$ of $A'$ and a $K_q$-decomposition $\mathcal{B}_2'$ of $L\cup A'$. Let $\mathcal{B}_1 := \mathcal{B}_1'\setminus \mathcal{B}_2'$ and let $\mathcal{B}_2 := \mathcal{B}_2'\setminus \mathcal{B}_1'$. 

Let $A\coloneqq \bigcup \mathcal{B}_2$. Note that $\mathcal{B}_2$ is a $K_q$-decomposition of $A$. Since $L$ has no $K_q$ subgraph, we find that $\mathcal{B}_1$ is a $K_q$ decomposition of $L\cup A$.  Thus $A$ is a $K_q$-absorber for $L$. By construction $\mathcal{B}_1\cap \mathcal{B}_2 = \emptyset$. Hence by Lemma~\ref{lem:AbsLowerBound}, we have that every vertex in $V(A)\setminus V(L)$ has degree at least $2q-2$ in $A$. Hence $m_2(A,V(L))\ge m(A,V(L))\ge q-1$. Since $A$ is a subgraph of $A''$, we find that $m_2(A'',V(L))\ge q-1$.

Since $A''$ is arbitrary, this implies that $d_{\rm abs}(K_q)\ge q-1$ as desired.
\end{proof}

We conjecture that the correct value is as follows:

\begin{conj}
For each integer $q\ge 3$, $d_{\rm abs}(K_q)=q$.    
\end{conj} 

For $q=3$, we have by Theorem~\ref{thm:GeneralTriangleAbsorberThreshold} that $d_{\rm abs}(K_q)\le 3$; we think a more involved lower bound argument, say using discharging, could show that $d_{\rm abs}(K_3)\ge 3$, but we leave this as an open question. Similarly, Theorem~\ref{thm:GeneralQAbsorberThreshold} yields that $d_{\rm abs}(K_q)\le q+0.5$ and we think an inductive construction as in the proof of Theorem~\ref{thm:GeneralTriangleAbsorberThreshold} could show that $d_{\rm abs}(K_q)\le q$; similarly, we believe a more involved lower bound argument could show that $d_{\rm abs}(K_q)\ge q$ thus proving our conjecture. Again, though, we leave this as open.

In light of Corollary~\ref{cor:AbsLowerBound}, it seems a new approach is needed to push past this theoretical barrier in Theorems~\ref{thm:YusterRandom} and \ref{thm:YusterRandomRegular}. We believe using non-constant sized absorbers might improve the values multiplicatively (say to $p=n^{-(1+\varepsilon)/q}$ for some small $\varepsilon>0$) but that new approaches would be needed to approach the correct multiplicative value of roughly $2/q$. 
Alternatively, it is conceivable that the correct answer could lie closer to the absorber/booster thresholds and hence that a better lower bound argument for $p$ and $d$ values in $G(n,p)$ and $G_{n,d}$ is required.

For fractional decompositions of $G(n,p)$, it is not apparent that absorbers (or their fractional ilk) are needed and hence the theoretical barrier mentioned above need not necessarily apply to finding fractional decompositions. That said, as mentioned in the proof overview, we were unable to use the Boost Lemma to directly find a fractional decomposition when $p\le n^{-1/(q+.5)}$ (due to the fact that $m(K_{q+2},K_q)=q+.5$). Perhaps new ideas could be found to work around this difficulty and push past the absorber barrier for fractional decompositions. 

\subsection{Optimal Leave Number}

Finally, one last direction to discuss is the optimality of the leave we obtained. It turns out for all graphs there is a natural lower bound on how many edges must be left out of a $K_q$-packing. While this number will be far from correct for some graphs (e.g. triangle-free graphs for a $K_3$-packing), for $G(n,p)$ in our range of $p$, it is conceivable that this is number is attainable. Here is the definition.

\begin{definition}
Let $q\ge 3$ be an integer. Let $G$ be a graph. We define the \emph{optimal $K_q$-leave number} of $G$ to be the smallest integer $k$ such that $k\equiv e(G) \mod \binom{q}{2}$ and $k\ge \frac{1}{2} \cdot \sum_{v\in V(G)} (d(v) \mod (q-1))$.

We say a $K_q$-packing $\mathcal{Q}$ of $G$ is \emph{optimal} if $e(G)-e(\bigcup \mathcal{Q})$ equals the optimal $K_q$-leave number of $G$.
\end{definition}

The optimal $K_q$-leave number is a natural lower bound for the number of leftover edges (or `leave') in a $K_q$-packing as the next proposition shows.

\begin{proposition}
Let $q\ge 3$ be an integer. Let $G$ be a graph. If $\mathcal{Q}$ is a $K_q$-packing of $G$, then $e(G)-e(\bigcup \mathcal{Q})$ is at least the optimal $K_q$-leave number of $G$.
\end{proposition}
\begin{proof}
Let $Q:= \bigcup \mathcal{Q}$ and let $H:= G\setminus Q$. Since $Q$ admits a $K_q$-decomposition (namely $\mathcal{Q})$, we have that $e(Q) \equiv 0 \mod \binom{q}{2}$ and that $d_Q(v) \equiv 0 \mod (q-1)$ for all $v\in V(G)$. Hence since $e(H)=e(G)-e(Q)$, we have that $e(H) \equiv e(G) \mod \binom{q}{2}$. Similarly since $d_H(v) = d_G(v)-d_Q(v)$, we have that $d_H(v) \equiv d_G(v) \mod (q-1)$ for all $v\in V(G)$. Thus $2\cdot e(H) = \sum_{v\in V(G)} d_H(v) \ge \sum_{v\in V(G)} (d_G(v) \mod (q-1))$. Since $e(H)$ is integer and satisfies the above, it follows that $e(H)$ is at least the optimal $K_q$-leave number of $G$.    
\end{proof}

It might be helpful to the reader to note the optimal $K_3$-leave number of $G(n,p)$ is $\frac{n}{4} \pm O(\sqrt{n \log n})$ asymptotically almost surely since that is exactly half the number of odd degree vertices of $G(n,p)$. This is because the number of odd degree vertices of $G(n,p)$ is $\frac{n}{2}$ in expectation but is only concentrated over an interval of size $\Theta(\sqrt{n \log n})$. Thus one cannot prove a better bound on the leave than $\frac{n}{4} + O(\sqrt{n \log n})$, but perhaps one can prove that \aas any individual 
instance has an optimal $K_3$-packing.

Theorem~\ref{thm:YusterTriangle} yielded a leave of $n+O(1)$. In fact, a Hamilton cycle with an edge-disjoint $4$-cycle and $5$-cycle constitutes a $K_3$-divisibility fixer. Moreover such a fixer has the property that it only needs to delete at most $\frac{n}{2} + 5$ edges to fix divisibility (by taking the smaller of the two halves of the Hamilton cycle linking odd degree vertices, and then deleting either the $4$-cycle or $5$-cycle as need to fix the number of edges). 

This bound is still off by a factor of $2$ from the optimal $K_3$-leave number of $G(n,p)$. To obtain the correct bound, better results about $G(n,p)$ need to be proved. We opted not to pursue said results but rather leave these as open questions.

First we conjecture that $G(n,p)$ contains a perfect matching on its odd-degree vertices for large enough $p$.

\begin{conj}
If $p\ge n^{-1/2}$, then $G(n,p)$ asymptotically almost surely contains a perfect matching of its odd-degree vertices. 
\end{conj}

Indeed, this should be true even for $p \gg n^{-1}\cdot \log n$, but we would only need for $p\ge n^{-1/2}$ say. However, given the structure of our proof (namely using one slice to find a random reserve $X$ and another slice to find an omni-absorber $A$ for that $X$), we would actually need a stronger theorem to improve the leave. In particular, it seems that the optimal leave would need to be found in the last slice. So we conjecture the following:

\begin{conj}
Let $\phi:[n]\rightarrow \{0,1\}$ such that $|\{i\in [n]:\phi(i)=1\}|$ is even.  If $p\ge n^{-1/2}$, then $G(n,p)$ asymptotically almost surely contains a perfect matching on $\{ v_i: d_{G(n,p)}(v_i) \equiv \phi(i)-1 \mod 2\}$. 
\end{conj}



For general $q$, we note the optimal $K_q$-leave number of $G(n,p)$ here is $\frac{(q-2)n}{4}~\pm O(\sqrt{n \log n})$ asymptotically almost surely. This is because the expected number of vertices with degree congruent to $i \mod (q-1)$ is $\frac{n}{q-1}$ and hence the expected optimal $K_q$-leave number is roughly $\frac{1}{2} \cdot \sum_{i=0}^{q-2}~i\cdot \frac{n}{q-1} = \frac{(q-2)n}{4}$. Again, one cannot prove a better bound on the leave than $\frac{(q-2)n}{4} + O(\sqrt{n \log n})$, but perhaps one can prove that \aas any individual instance has an optimal $K_q$-packing. Thus we also leave the question of whether the optimal $K_q$-leave number for $G(n,p)$ for general $q$ can be attained for our range of $p$.

Lastly, one could consider the best possible leave of a $K_q$-packing of $G_{n,d}$ if one does not assume that $(q - 1) \mid d$ or $q \mid dn$.  In fact, Yuster had conjectured that \aas there exist $\lfloor dn / (q(q - 1))\rfloor$ pairwise edge-disjoint copies of $K_q$ in $G_{n,d}$ provided $(q - 1) \mid d$, which would be a $K_q$-packing with leave of size the optimal $K_q$-leave number.  However, this conjecture is easily seen to be false when $q \nmid dn$, since it is not possible to delete fewer than $\binom{q}{2}$ edges of a $d$-regular graph where $(q - 1) \mid d$ and $q \nmid dn$ and obtain a $K_q$-divisible graph\COMMENT{We need to delete edges at some vertex $v$, and the number of edges incident to $v$ that we delete must be divisible by $(q - 1)$.  Hence, we must delete at least $q - 1$ edges at $v$, and then the same argument applied to other ends of the edges we delete implies that we need to delete at least $(q-1)(q - 2)/2$ additional edges.  Therefore, we need to delete at least $q - 1 + (q - 1)(q - 2)/2 = \binom{q}{2}$ edges.}.
Instead, the best possible leave seems to be whatever number is the smallest number of edges in an $n$-vertex graph $L$ with $e(L) \equiv nd/2 \mod \binom{q}{2}$ and degrees congruent to $d \mod q - 1$.


\end{document}